\documentclass{llncs} %%Comment for original version
\usepackage{latexsym}
\usepackage{amssymb}
\usepackage{amsmath}
\usepackage{mathtools} 
\usepackage{amsfonts}
\usepackage{algorithm}
\usepackage{algpseudocode}
\usepackage{enumitem}

\def\f{\textbf{f}}

\def\R{\mathbb{R}}
\def\ga{\mathfrak{a}} % "gothic a"
\def\gb{\mathfrak{b}} % "gothic a"
\def\gd{\mathfrak{d}} % "gothic a"
\newcommand{\bmat}{\left[\begin{matrix}}
\newcommand{\emat}{\end{matrix}\right]}
\usepackage[latin1]{inputenc} 
\usepackage{amsxtra}
\usepackage{amscd}
\usepackage{setspace}   % enables line spacing commands 
\usepackage{mathrsfs}   % enables \mathscr
\usepackage{tabularx}
\usepackage{booktabs}

\newcommand{\Z}{\mathbb{Z}}
\newcommand{\Q}{\mathbb{Q}}
\newcommand{\C}{\mathbb{C}}
\newcommand{\cO}{\mathcal{O}}
\newcommand{\cN}{\mathcal{N}}

\newcommand{\ignore}[1]{{}}

\newcommand{\A}{\mathbb{A}}

\newcommand{\GL}{\mathrm{GL}}
\newcommand{\SL}{\mathrm{SL}}

\newcommand{\PSL}{\mathrm{PSL}}

\newcommand{\gp}{\mathfrak{p}} % gothic p
\newcommand{\vol}{\mathrm{vol}}
\newcommand{\diag}{\mathrm{diag}}

\newcommand{\ord}{\mathrm{ord}}
\newcommand{\Log}{\mathrm{Log}}

\newcommand{\va}{\mathbf{a}}
\newcommand{\vb}{\mathbf{b}}

\newcommand{\vm}{\mathbf{m}}

\newcommand{\vx}{\mathbf{x}}

\newcommand{\vv}{\mathbf{v}}

\newcounter{daggerfootnote}

\newcommand{\henry}[1]{{ \textcolor{purple}{{#1}}}}

\newcommand{\phong}[1]{{ \textcolor{red!50!black}{Phong: {#1}}}}

\usepackage{hyperref}
\usepackage{xcolor}

    \title{Adelic reduction of module lattices}

    \author{Henry Bambury, Seungki Kim, Changmin Lee, Phong Q. Nguyen}
    
    \institute{}
\begin{document}
\maketitle

\begin{abstract}
We give a strict generalization of the LLL algorithm over number fields, based on the reduction theory of $\GL(n)$ over the adele ring of a number field. Our algorithm is free of heuristics, with rigorous bounds on output quality and complexity. As a consequence, we obtain a hierarchy of reductions from module-(H)SVP to ideal-HSVP, an example of which has runtime and approximation factors subexponential in the field degree. More importantly, we uncover a close connection between structured lattice reduction and a Diophantine approximation over number fields.
\end{abstract}

\section{Introduction}

%\subsection{Motivation}

Since structured lattices entered lattice-based cryptography thirty years ago with the NTRU cryptosystem, there has been much temptation to challenge them by exploiting the rich layers of the extra algebraic structure that they possess. Surprisingly, they turned out to be extremely resistant to such attempts, which in no small part led to their remarkable success today, with three out of five NIST PQC standards being based on structured lattices. Though there has been some progress, most notably via S-unit attacks (see e.g., \cite{CDPR16,PHS19}); \ignore{also the references in \cite{Bam22}),}the consensus is that they do not appear to impact the current NIST standards, as NIST Internal Report 8413 \cite[p.85]{IR8413} concluded:
\begin{quote}
\it ...the techniques used in this algorithm rely heavily on the multiplicative structure of ideal lattices, and do not seem to be directly applicable to module lattices of rank 2 or more. Hence these techniques are not known to directly impact the hardness assumptions of any of the round 3 candidates...
\end{quote}

One of the most natural ideas for the cryptanalysis of module lattices is an extension of the fundamental LLL algorithm \cite{LLL82} to module lattices. There are other motivations for such an algorithm as well arising from computational number theory in general, \emph{cf.} \cite{Sim10,Str22}. However, as will be discussed below, all previous attempts at such a generalization are either limited in scope, missing critical details, or reliant on relatively unfamiliar heuristics. The goal of the present paper is to make a contribution in this regard.

\subsection{Contributions}

Our work is grounded in the philosophy that the LLL algorithm \cite{LLL82} is an algorithmic manifestation of the reduction theory of the moduli space of lattices $\GL(n,\Z)\backslash\GL(n,\R)$ ---  \emph{cf.} \cite[Theorem 2.2]{Bor19}. Thus the ``correct'' module LLL algorithm would be an algorithmic manifestation of the reduction theory of the moduli space of module lattices (over a number field $F$), which is known to be
$$\GL(n,F) \backslash \GL(n,\A_F) / \prod_{\nu \nmid \infty} \GL(n,\cO_\nu)$$
(see Section 2 below for the notations, also see \cite[Sec. 2.3.2]{dBPTW26}). This space has reduction theory too, in the same sense that $\GL(n,\Z)\backslash\GL(n,\R)$ does (see e.g., \cite[Ch. 3]{Gar18})\ignore{{\color{purple} Henry: is it coherent to have $\GL$ on one side and $\SL$ on the other? Should we explain why the left quotient has the field instead of the ring? The generalisation might not be obvious to people that don't know reduction theory} {\color{orange} Unfortunately there's no short answer to these things... And I changed $\SL$ to $\GL$ to lessen the confusion}}. In this paper, we materialize this theory into an algorithm, or, in other words, we give an algorithmic version of this theory.

The outcome, which we refer to as the \emph{adelic LLL} algorithm (Algorithm \ref{alg:lll_p} below), 
takes as input a $g \in \GL(n,\A_F)$ (instead of $\GL(n,\R)$ for LLL),
and outputs a reduced $h \in \GL(n,\A_F)$ in the same double coset,
that is, there exists $\gamma \in \GL(n,F)$ and $k \in \prod_{\nu \nmid \infty} \GL(n,\cO_\nu)$ such that $h = \gamma g k$
and the first row is $h$ is ``short'',
namely it has low height --- an analogue of length in the number field context (\emph{cf.} Section \ref{sec:height}).
It shares with the classical LLL \cite{LLL82} all the important characteristics such as follows.
\begin{itemize}[label=\textbullet]
\item Improvement to the given basis $g$ is made by repeatedly size reducing and swapping. In fact, ours is a strict generalization of the original LLL, in that the two algorithms coincide over $\Q$. Just as LLL can be viewed as an efficient algorithmic version of Hermite's (exponential) inequality on Hermite's constant $\gamma_n$, our adelic LLL can be viewed as an algorithmic version
of a Hermite-like inequality on Thunder's adelic generalization of Hermite's constant $\gamma_{n,1}(F)$~\cite{Thu98}.
\item It is completely free of heuristics. In particular, its correctness and bounds on output quality (Propositions \ref{prop:rhf_bound}, \ref{prop:rhf_bound_n}) and time complexity (Propositions \ref{prop:swap_bound}, \ref{prop:swap_bound_n} and Theorems \ref{thm:main}, \ref{thm:main_n}) can all be rigorously established. %\henry{I suggest we also point to the Propositions for rank 2 as they have not been mentioned in the contributions (Propositions~\ref{prop:rhf_bound} and~\ref{prop:swap_bound})}

\item To elaborate: for LLL, it is well-known that it returns in polynomial time a lattice vector $v$ satisfying
\begin{align*}
 \|v\| \le [(4/3+\varepsilon)^{1/4}]^{n-1} \cdot |\det L|^{1/n} \,\text{and}\,
 \|v\| \le [(4/3+\varepsilon)^{1/2}]^{n-1} \cdot \lambda_1(L) 
 % \|\text{first row of}\,\, h\| & \le [(4/3+\varepsilon)^{1/4}]^{2(n-1)} \cdot \text{first minimum of the lattice}
 \end{align*}
where $L$ is the lattice and $n$ is its rank. 
In comparison, adelic LLL returns in polynomial time and polynomially many calls to its subroutines (which one can choose to be polynomial time themselves) an adelic vector $v \in F^n g \setminus \{0\} \subseteq \A_F^n$ such that  %a rank 1 submodule $v$ of the module lattice $L(g)$ associated to $g\in\GL(n,\A_F)$ (\emph{cf.} Section \ref{sec:L(g)}) satisfying
\begin{align*}
H(v) \leq Q^{n-1}\cdot H(\det g)^{1/n} \,\text{and}\, H(v) \leq Q^{2(n-1)} \eta(g),
\end{align*}
where $n$ is the module rank, $\eta(g)$ is a height analogue of $\lambda_1(L)$, and $Q$ is a factor that depends on the number field $F$. Here, $g$ corresponds to a module lattice, $v$ to a rank $1$ submodule of this module lattice, and the height function $H$ provides a way to measure sizes.
%\begin{align*}
%Q=(4/3+\varepsilon)^{\frac{d}{4}}(CB)^\frac{1}{2}e^{\frac{Ad}{2}},
%\end{align*}
%$d$ is the degree of $F$, and $A$,$B$,$C$ depend on the number field $F$
%and the subroutines used. 

%Here, the $(\delta - \mu^2)$ part coincides with LLL's $(4/3+\varepsilon)$, and the right-hand part accounts for the number field $F$.
%\item Adelic LLL provides a reduced pseudo-basis of a module lattice: in particular, the first row of $h$ gives a low-height rank-1 submodule. By running any SVP approximation algorithm on that submodule, one obtains a short vector in that rank-1 submodule, and such a short vector approximates both HSVP and SVP in the input module lattice, because the height is low.
\item It works over any number field, and for any module lattice of rank $\geq 2$, free or not.
\item Its simplicity and explicit parameter values allow for implementation on classical computers and deeper performance analyses. %{\color{red}[REFs]; also mention and link Henry's POC if everyone approves it} {\color{purple} I have so work to do with the scaling still, and need to test the swap function when ideals are nontrivial}
\end{itemize}
To the best of our knowledge, no other proposed (structure-preserving) reductions over number fields enjoy all these properties simultaneously --- see Section~\ref{sec:rel-work} and Table \ref{table:comparison} below, and Appendix \ref{sec:comparison}.

In addition to the robustness of the algorithm, there are a couple of points on the (theoretical) performance of adelic LLL that are worth mentioning: 
%A disclaimer is in order: our algorithm, like many of its predecessors, relies on a few choices of Hermite-SVP oracles, one of dimension $d$ and another of dimension $d' \in [d+1,2d]$. One is free to choose anywhere between LLL and an exact SVP solver, and any value of $d'$ for which $F$ has a subfield of degree $d'-d$. As expected, these choices directly influence the performance of adelic LLL.

\subsubsection{Reduction from module-(H)SVP to ideal-HSVP}

Well-known previous results on structure-preserving reductions (e.g., \cite{LPSW19,MS20}) phrased their achievements in terms of reductions from rank $n\geq 2$ module-SVP (resp. Hermite-SVP) to rank $\geq2$ module-SVP (resp. Hermite-SVP). In light of the NIST comment cited above, a reduction to ideal-HSVP would be the most desirable. Yet there existed difficulties involving the reduction of rank 2 modules.

Our work clears this obstacle, thereby obtaining a reduction from approximate module-SVP (resp. HSVP) of any rank to approximate ideal-HSVP, without involving any overly costly oracle that would beat the point of the reduction. The approximation factors and the time complexity are described explicitly in terms of the field invariants and the subroutines used.
See Theorem~\ref{thm:reduction} below for details.

\subsubsection{Reduction to a Diophantine approximation problem}
%\henry{Thinking out loud, isn't this Diophantine approximation problem exactly an instance of rank 2 module HSVP? so are we reducing a problem to itself?} {\color{orange} Not necessarily, that's just the laziest way of solving it. Like, on norm-Euclidean fields, you can do continued fraction which is much more efficient. Also if you can somehow choose a good $q$, then it's an ideal aCVP.} \henry{Thanks, related to this, do you have references for continued fractions over number fields?}{\color{orange} if you google complex continued fractions, you'll get some notes. For general number fields not much is done due to the obvious difficulties}
It turns out that the output quality of adelic LLL depends critically on how well one can solve the following Diophantine approximation problem over a number field $F$: given parameters $0<\mu<1$ and $C>0$, and a vector $(m_\sigma)_{\sigma:F\rightarrow\C} \in F\otimes\R$, find $p,q$ in the ring of integers of $F$ such that
$$|m_\sigma\sigma(q)-\sigma(p)|\leq\mu \mbox{ and } |\sigma(q)| \leq C^{\frac{1}{|F:\Q|}}$$
for every real or complex embedding $\sigma$ of $F$. The smaller we are able to set $C$ (as a function of $\mu$), the better the output of adelic LLL.

In this work, we simply employ a well-known, standard strategy based on (unstructured) lattice reduction (Algorithm \ref{alg:sizereduce} below). If one can sufficiently improve on this, however, adelic LLL can outperform the current state-of-art unstructured reductions --- e.g., BKZ, slide algorithm --- of comparable complexity. In other words, the hardness of the above Diophantine problem is related to whether module lattices are more ``vulnerable'' than the unstructured ones. Though it turns out that our algorithm performs somewhat worse than BKZ of comparable runtime, it paves an interesting way forward. Halving the currently achieved approximation exponent is sufficient for adelic LLL to match BKZ's performance; more than half gives it an advantage. See the discussion in Section \ref{sec:reduction} below for details.%See Section \ref{sec:reduction} for a discussion.

\paragraph{} We strongly believe that viewing structured lattice reduction from the point of view developed in this paper opens up multiple avenues for further research. We note that as with classical LLL, it is likely that our algorithm performs much better than its theoretical guarantees. Furthermore, any further algorithmic progress on rank-1 module lattices or the adelic subroutines directly improves the approximation factor of adelic LLL.

\ignore{
For a power-of-2 cyclotomic field of degree $d$, assuming the generalized Riemann Hypothesis and the affirmative answer to Weber's class number problem, $\gamma'/\gamma$ ranges from $O(d^4(\log d)^2)$ {\color{orange} I guess as Henry pointed out we have to settle with subexponential (like $e^{\tilde O(\sqrt{d})}$) loss for now :( this is really vexing. So in this case, there's no reason to use exact SVP, we can use slide algorithm with blocksize $\sqrt{d}$ to reduce subexponential rank 2 module HSVP to subexponential ideal HSVP (for which [CDW21] gives heuristic poly time algorithm) in subexponential time. Phong --- the creator of slide algorithm --- please comment if this makes sense.} to $1.075^{d+o(1)}$,
\phong{There's a concern here: slide reduction already solves
subexponential HSVP in subexponential time. I don't know how
these subexponential factors compare, but qualitately,
both are subexponential.} {\color{orange} the dimension halves so perhaps that helps a little, but I basically agree}
depending on the choice of the HSVP solver for $\Z$-rank $d+1$ lattices (which is a necessary subroutine of adelic LLL; one can choose an exact SVP solver, the LLL algorithm, or anything in between). Other than the complexity of this HSVP solver, this reduction runs in polynomial time.

\phong{I'm getting confused here.
Let's try to clarify.
Running LLL on a 2-rank module lattice $L$ of dimension $n$,
would yield in polynomial time a sublattice $\Gamma$ of co-volume $\le 2^{O(n^2)} \sqrt{\vol(L)}.$
Even if we were solving SVP exactly on $\Gamma$,
we would obtain a short vector of $L$
of norm $\le 2^{O(n)} \vol(L)^{1/n}$, 
which is no better than the first vector of LLL.
(maybe the constant would be better, because I didn't specify what is $O()$). And $\Gamma$ has no structure,
so we cannot apply all the work on ideal-SVP algorithms.
On the other hand, adelic-LLL obtains an ideal lattice $\Gamma$
of co-volume/norm $\le \mathrm{h} \sqrt{\vol(L)}$
where $h$ depends on the size reduction in adelic-LLL
(whether we use polynomial-time or exponential-time size reduction),
the number field, and many other things.
Here, there are two benefits compared to LLL:
$\Gamma$ has structure, it is an ideal lattice, 
and not an arbitrary sublattice, so we can apply specialized SVP algorithms for $\Gamma$;
and furthermore, $h$ is much smaller than $2^{O(n^2)}$,
so it makes it worth it.
Finally we plug in ideal-SVP algorithms,
which would yield a Hermite factor of the form
$\le \mathrm{subexponential} \times h^{2/n}$.
}
{\color{orange} thanks, this is illuminating}
\phong{There's a lot of question marks depending on what is exactly the value of $h$. }
{\color{orange} that's discussed in Sec 3.5 and seems for now it's at best $\exp(\tilde O(\sqrt d))$ on GRH. I have an idea to try to make it polynomial but probably later}

\phong{Henry and I have been discussing regularly at EC.
One thing surprising us is the AC19 result:
they used a Siegel-type condition,
so they can also obtain bound the firt part of their pseudo-basis,
so in principle, they could also reduce HSVP of a rank-2 module lattice to HSVP in a rank-1 module lattice with a subexponential loss, albeit with heuristic and stronger more expensive oracles. However, they don't discuss at all
the possibility at applying ideal-SVP algorithms to exploit the reduction: why is that?
It seems to us that any algorithm to find short vecctors
in ideal lattices in $O_K$ can also find short vectors
in ideal lattices in $(O_K)^2$: are we missing anything?
}
{\color{orange} what's the reference?}
\phong{The reference inside the AC19 paper? There's even a part
in the paper where they apply a 2-rank algorithm to a 1-rank module lattice. Anyways, we checked that slide reduction with $\sqrt{n}$ blocksize yields a subexponential factor in the 2-rank approximation.} {\color{orange} Oh you're talking about LPSW19. This may be Changmin's question to answer. My guess is that it has to do with the bound on the determinant of the dense rank 1 submodule they obtain. In my reading ... I realize I was being dumb, sorry }
\phong{In our reading, we look at Lemma 3.2 of the AC19 paper,
which gives $h = \alpha_K$ and we're interested in $h^{2/d}$.
Section 3.3 gives $\alpha_K$ of the order  $\gamma_{\mathcal N}^{2d}$
, so $h^{2/d} \approx \gamma_{\mathcal N}^4$. In Theorem 4.1, they take $\gamma_{\mathcal N} = 2^{\tilde{O}(d \log \rho(R))/d}$, 
so $\gamma_{\mathcal N}^4 =2^{4\log\rho(R) \log^{O(1)}(d  \log \rho(R))}$ which seems subexponential, no? In AC19, the cost of the algorithm
is too expensive, but it would have motivated the search for better alternatives.} 
{\color{orange} ah ok, so it's because of the costs (requiring exact CVP) they didn't mention it. the factor itself is better than what I can afford for now}

The idea is simple: we run adelic LLL on the input rank 2 module, which outputs a dense rank 1 submodule. If one obtains a short vector of this submodule, then trivially one also obtains a short vector of the original rank 2 module.

\phong{Here are a few suspicious calculations.
Imagine we use LLL for the diophantine approximation,
then the main term of $C$ is $(4/3)^{d^2/4}$,
so the final Hermite approximation factor would be $C^{1/(2d)} = (4/3)^{d/8}$, ignoring subexponential terms
from the quantum polynomial-time algorithm on the ideal.
If we're simply using LLL on the final ideal,
we would get $(4/3)^{d/8} (4/3)^{d/4} = (4/3)^{3d/8}$.
On the other hand, if we were running LLL directly on the 2-rank lattice,
we would get $(4/3)^{(2d)/4} = (4/3)^{d/2}$ which
is a worse exponent: or are there typos in the formula?
(if it's $C^{1/d}$ rather than $C^{1/(2d)}$, then
the advantage disappears)
If correct, adelic-LLL using only LLL subroutines
would give a better constant (namely $(4/3)^{3/8}$ vs $(4/3)^{1/2}$ than integer LLL! It sounds too good to be true,
so I do wonder about $C^{1/d}$ vs $C^{1/(2d)}$.
Does adelic-LLL really provide a denser sublattice
than integer-LLL? }

\phong{In the general case, if diophantine approximation
uses an algorithm with a Hermite approximation factor $h$,
then we can take $C=h^{d+1}$ so $C^{1/(2d)} \approx \sqrt{h}$.
For slide reduction of blocksize $\beta$,
we get $h = \beta^{d/(2\beta)}$ so $\sqrt{h} = \beta^{d/(4 \beta)}$
For instance, if we use $\beta = \sqrt{d/4}$ then
$\sqrt{h} = \sqrt{d/4}^{\sqrt{d/4}} = d^{\sqrt{d}/4}$.
And we have to multiply this by another subexponential factor
from the number field:
if we also use the same slide reduction for the ideal,
then we have to multiply that
$\beta^{d/(2\beta)} = d^{\sqrt{d}/2}$,
so overall $d^{3\sqrt{d}/4}$
On the other hand, slide of blocksize $\beta = \sqrt{d/4}$ 
on the full module lattice
would give $\sqrt{\beta}^{(2d)/\sqrt{d/4}} = d^{4\sqrt{d}/4} = d^{\sqrt{d}}$, which is worse.
So adelic-LLL using only subexponential slide reduction of
squareroot blocksize
also gives a better exponent than 
the same slide reduction on the full module lattice!
}

\phong{On the other hand, I'm not gonna put what happens with an exact algorithm, cause slide reduction with a blocksize linear in the dimension
gives a polynomial approximation factor,
which will be better than any subexponential factor 
coming from the number field.
}

\phong{We have to clarify what happens with rank-$r$ module lattices for arbitrary $r \ge 2$. I know we just need to plug in the previous reductions, but people will want to know
what's going on.}
{\color{orange} I have some other plan on what to do for rank $\geq 2$. we can talk more on 5/13}
\phong{Looking at both Noah2019 paper and the LPSW19,
there is something strange in their claims: when we take $r=2$,
we get a much worse bound, as if the reduction wasn't right.
For instance, Th 3.9 in LPSW19 with rank 2 blows up
the approximation factor $gamma' = \gamma^3$,
instead of $\gamma$.
There could be several things going on:
either their reduction is not the right reduction,
and the real adelic LLL will give the ``right'' approximation factor,
and that's one more reason for the adelic point of view;
either there is an obstruction we fail to see.
What we want is a HSVP reduction from rank-$r$
to rank-2.
} {\color{orange} I(SK)'m probably not the best person to handle this question.. did you talk to Stehle, etc.?}}

\subsection{Use of adeles}
\ignore{
\phong{Is the following analogy meaningful?
Integer LLL is usually described with Gram-Schmidt vectors,
because it helps the visualize what is going on.
However, it is well known that a real implementation of LLL
does not need to explicitly compute the Gram-Schmidt vectors
(though it could do so if it wanted to): instead, only the $\mu_{i,j}$'s coefficients and the Gram-Schmidt squared norms are computed.
So is it fair to say that the ad\`eles are useful to ``visualize''
but we don't actually need to compute them explicitly 
(which here would be expensive) to run the reduction?
} {\color{orange} in the pure adelic version only $\mu$'s and square norms are required too. 

In the justification of the ``postprocess'' step, I had to manipulate on the right side by $\GL(2,\cO_\nu)$, and that feels to me like a genuinely adelic thing to do, though at the end of the day it comes down to ideal computations
}}

As suggested above, we adopt the language of adeles in this paper. This may make our work somewhat less accessible, but the technical advantages it brings more than make up for this potential entry barrier. For instance, one of the biggest stumbling blocks in previous attempts at generalizing LLL to number fields has been the class group. Hence many works (e.g., \cite{Nap96,GLM09,KL17,LPL20}) on this subject were limited to the cases in which the ring of integers is (norm-)Euclidean; also the recent work \cite{DMPT} was primarily devoted to a difficulty of this kind. The adelic perspective allows one to overcome this problem elegantly and effortlessly, without even barely mentioning the class group. %It may be worth noting the following related, standard fact: it takes class number-many disjoint Siegel domains to cover the fundamental domain of the module lattices classically; adelically, it takes only one. 
This is one of the countless instances in mathematics in which the use of the right language makes the problem drastically more tractable.

%It reflects what seems to be a general principle that, when it comes to the geometry of numbers, most phenomena are independent of the class group --- see e.g., \cite[Remark 6]{GSVV25} for another instance of this principle at work.

Since the adele ring is defined as an infinite product (see Section \ref{sec:adeles} below), one may wonder how to write an implementable algorithm out of it. When the input comes in the form of a pseudobasis --- which is the case for the cryptanalyses of all existing module-based schemes --- however, it is possible to rewrite all necessary adelic computations in terms of efficient ideal and vector computations. Thereby we obtain the adelic LLL algorithm (Algorithms \ref{alg:lll_p} and \ref{alg:lll_p_n} below) that is implementable over classical computers. 

Adeles are not exactly new in lattice-based cryptography. They have appeared in the recent work \cite{dBPTW26}; see also \cite[Ch.8]{Bam25} and \cite{DK22}. A similar construction, sometimes referred to as Arakelov divisors, features in \cite{dBDPW20} as well. We expect adeles to play more and more important roles in the field, and hope this work may facilitate such development.

\subsection{Related work}\label{sec:rel-work}

LLL for modules over number fields has been of interest for a number of reasons since the 1990s. It may be said that a landmark work on this subject is Lee, Pellet--Mary, Stehl\'e, and Wallet \cite{LPSW19}. It was the first, and the only work so far, that attempted to describe a fully structured reduction algorithm over an arbitrary number field, with full details on the algorithm itself as well as its performance analysis. We refer to \cite[Sec. 1]{LPSW19} for a summary of the literature prior to their work.

The hardest part of \cite{LPSW19} is reducing a rank 2 module lattice. Rank $\geq 3$ modules are handled in an analogous manner to BKZ, given the rank 2 reduction. Indeed, it seems that the one provided by \cite[Section 4]{LPSW19} is the only known rank 2 module reduction algorithm prior to the present work. 
However, it is not an efficient algorithm, because it requires an exact CVP oracle on a lattice
of much higher dimension than the input lattice, namely
$O(d^{2+\eta})$ (where $d$ is the field degree and $\eta>0$)
versus $2d$.
In addition, it relies on several heuristics,  and
it is a quantum algorithm (unless the input module is free), although it can be dequantized with an additional routine at an extra polynomial-time cost \cite[Theorems 6.1-3]{DMPT}. Moreover, some important parameters for the algorithm are not explicitly specified --- e.g., $c$ in \cite[Lemma 2.5]{LPSW19} --- presenting a difficulty in its implementation. (Although \cite{KK24} gives a partial implementation of \cite{LPSW19}, they completely bypass the rank $2$ part by replacing it with BKZ.)

The main differences between the reduction algorithms of \cite{LPSW19} and the present work are summarized in Table \ref{table:comparison}. A more careful comparison, detailing all the aspects in which our work and \cite{LPSW19} differ, is given in Appendix \ref{sec:comparison}.

\begin{table}[htbp]
\centering
\caption{Summary of comparison between \cite{LPSW19} and this work. Here, $F$ is the target number field of degree $d$, and $n$ is the module rank.} \label{table:comparison}
\renewcommand{\arraystretch}{1.3} 
\begin{tabularx}{\textwidth}{@{} >{\raggedright\arraybackslash}p{4.5cm} >{\raggedright\arraybackslash}X >{\raggedright\arraybackslash}X @{}}
\toprule
 & \textbf{Lee et al.} \cite{LPSW19} & \textbf{This work} \\
\midrule
Quantum/classical                            
    & Quantum except on free modules. Dequantization requires an additional algorithm from \cite{DMPT}. 
    & Classical \\
\midrule
Heuristic                         
    & Yes
    & No \\
\midrule
Required oracle                              
    & Exact-CVP solver on dimension $d^{2+\eta}$. 
    & Any HSVP solver on dimension $d+\deg E$, for any subfield $E \subseteq F$. \\
\midrule
Major contributor to complexity              
    & Preprocessing + poly. calls to CVPP.
    & Polynomially many calls to the oracle. \\
\midrule
Height of short rank-1 module (very roughly) 
    & $2^{\mathrm{(const)}nd(\log d)^{O(1)}}$ 
    & $2^{\mathrm{(const)}nd^2}$, with LLL for HSVP, $\deg E = \mathrm{(const)}d$ \\
\bottomrule
\end{tabularx}
\end{table}

%For comparison, we remind the readers that our proposed adelic LLL (Algorithm \ref{alg:lll_p}) makes a gradual improvement by simply swapping two adjacent basis vectors, just like the classical LLL.

Another interesting, oft-cited work on the subject is \cite{MS20}, which generalizes the slide reduction \cite{GN08} to the context of modules. The most relevant case to us is when the blocksize is $2$, which the authors admit is ``essentially identical to the reduction in \cite{LPSW19}'' \cite[Sec. 4.1]{MS20}. They do not substantiate the rank 2 algorithm however. The consideration of the bitsizes is also missing, which is especially important for implementation purposes. The language of filtration introduced in \cite{MS20} can be read off from the adelic Iwasawa decomposition (Section \ref{sec:Iwasawa} below), just as a flag of a Euclidean lattice can be read off from the Iwasawa decomposition of the corresponding basis matrix.

\cite{KEF20} (see also the full version \cite{EKF19}) presents a reduction toolkit specialized to free module lattices over cyclotomic fields, especially those of conductor $2^k$. They exploit the tower structure of the field (e.g., $\Q \subseteq \Q(\zeta_2) \subseteq \ldots \subseteq \Q(\zeta_{2^{k-1}}) \subseteq \Q(\zeta_{2^{k}})$) and the symplectic structure of certain kinds of lattices (e.g., that of NTRU lattices, as in \cite{GHN06}). These run orthogonal (in the colloquial sense) to the kind of structure that adelic LLL is designed to exploit; perhaps they can complement one another. In any case, the algorithms of \cite{KEF20} are highly heuristic, as the authors admit in \cite[Sec. 1.3]{EKF19}. It is also to be noted that they are mainly aimed at being fast, at the expense of the heuristic sub-LLL approximation factor $2^{\tilde O(d)}$ \footnote{$\tilde O(d)$ means $O(d(\log d)^c)$ for some constant $c>0$.} on rank 2 module lattices over a cyclotomic field of degree $d$; a more precise expression for the approximation factor seems unavailable. %In comparison, adelic LLL can be leveraged to attain the approximation factor of $2^{O(d)}$ in polynomial time or of $\mathrm{poly}(d)$ in $2^{O(d)}$ time. {\color{red} this is incorrect now so fix this} Moreover, the implicit constants here can be made quite explicit: see Proposition \ref{prop:rhf_bound} and the discussion in Section \ref{sec:reduction} below. %{\color{orange} honestly I suspect the speedup of \cite{KEF20} comes from parallelization --- which can be done for classical unstructured LLL too --- and not from the explotation of structure. For example, is it faster than Flatter (Ryan-Heninger)? Please comment.}

The module BKZ as referred to by in \cite{DLP25}, while it does exploit a certain aspect of the module structure, is also tangential to our work, as our main goal is to explore the structured local routine, as in \cite{LPSW19}, rather than outsourcing it to an unstructured SVP solver.

Outside the cryptographic literature, \cite{Str22} gives a reduction algorithm of what is essentially rank 2 free modules over the ring of integers $\cO_F$ of a totally real field $F$. Its strongest point is that it returns the representative of some fundamental domain\footnote{Thus the representative may not consist of short vectors of the module.} of the Hilbert modular group $\PSL(2,\cO_F)$, thanks to its relatively simple reduction theory. Its performance analysis is incomplete, however, except that \cite{Str22} mentions the complexity is exponential in the input bitsize.

Yet another noteworthy work is \cite{RT25}, which presents an LLL-like reduction algorithm for $\mathbf G(\Z)\backslash \mathbf G(\R)$, where $\mathbf G$ is any split semisimple group over $\Z$. From the reduction theory perspective, \cite{RT25} comes very close in spirit to our work. However, it does not account for the general number field case of $\mathbf G(\cO_F)\backslash \mathbf G(F\otimes\R)$ --- one may suggest the Weil restriction, but it breaks both splitness and semisimplicity in general. The ideas presented in this paper are precisely the necessary ingredients to make this happen.

\ignore{
\subsection{Future work}

The robust and transparent structure of adelic LLL opens up paths to numerous follow-up research on structured lattice reduction. We list some of the planned follow-up works.

The most imminent project, as already mentioned above, is to try improving on the Diophantine approximation over number fields, as its strength is the major deciding factor as to whether adelic LLL can outperform unstructured reduction. 

Another important direction is to construct an efficient implementation of adelic LLL, and study its behavior in practice. Just like the original LLL algorithm \cite{LLL82}, it is likely that adelic LLL will do significantly better than its theoretical output quality bound (Proposition \ref{prop:rhf_bound_n}). It is even more so, since, whereas the RHF bound of LLL \cite[(1.9)]{LLL82} is sharp, its counterpart for adelic LLL is not, due to the difficulties in navigating the geometry of the ring of integers $\cO_F$ and the unit lattice $\cO_F^*$.

There are also parts of the algorithm that could see improvements. For instance, we employ (a corrected version of --- see Appendix \ref{app:correct} below) the method of \cite{LPSW19} to keep the bitsizes low throughout the runtime of the algorithm; while it does its job, we believe there may be a method with a better fit for adelic LLL. Adopting quantum algorithms (e.g., \cite{BS16,dBF26}) for improved performance is another promising avenue of approach.
}

\section{Preliminaries}

The goal of this section is to quickly list the number theoretic notions used in this paper. We seek to deliver the minimum amount of practical knowledge necessary to understand the rest of the paper. The reader is referred to standard texts in algebraic number theory, such as \cite{Wei74}, for full details, or to works in cryptology such as \cite{dBPTW26,DK22}.

Throughout we fix a number field $F/\Q$ of degree $d$, and write $\cO_F$ for its ring of integers.

\subsection{Basic number field facts}

A \emph{place} $\nu$ of $F$ is an equivalence class of absolute values, or metrics, defined on $F$, where two absolute values $|\cdot|_1,|\cdot|_2$ are said to be equivalent if $|x|_1<1\Leftrightarrow |x|_2<1$ for all $x \in F$. The completion of $F$ with respect to $\nu$ is denoted by $F_\nu$. In particular, $F$ is a subset of $F_\nu$, and there exists the natural embedding $F \hookrightarrow F_\nu$ by inclusion. For $a \in F$, we often denote by $a_\nu \in F_\nu$ the image of $a$ via this map. We write $\cal P=\cal P(F)$ for the set of all places of $F$. There are two categories of places, \emph{finite} and \emph{infinite}. 

The finite places, denoted $\nu \nmid \infty$, correspond one-to-one to the nonzero prime ideals of $\cO_F$. For this reason, we will use the notations for places and primes interchangeably, or even refer to places as primes and vice versa, as is customary in the literature. For each $\nu \nmid \infty$, there exists the \emph{ring of $\nu$-adic integers} $\cO_\nu \subseteq F_\nu$, the set of all elements of $F_\nu$ with absolute value $\leq 1$. $\cO_\nu$ is a principal ideal domain whose unique prime ideal is $\nu\cO_\nu$; its field of fractions is precisely $F_\nu$. We fix once and for all a generator $\pi_\nu$ of $\nu\cO_\nu$, called a \emph{uniformizer} of $F_\nu$. Every element $a_\nu$ of $F_\nu$ has a unique presentation $a_\nu =\gamma \pi_\nu^\alpha$, where $\gamma \in \cO_\nu^*$ and $\alpha \in \Z$; moreover, $a_\nu \in \cO_\nu$ if and only if $\alpha \geq 0$. We define $\ord_\nu a_\nu := \alpha$, and $N(\nu) = |\cO_\nu:\nu\cO_\nu|$. The absolute value corresponding to $\nu$ is then $|a_\nu|_\nu = N(\nu)^{-\ord_\nu a_\nu}$. We extend $N()$ to fractional ideals, to denote their norm.

\subsubsection*{Example} If $F=\Q$ and $\nu=p$ is a rational prime, then $F_\nu = \Q_p$ is the field of $p$-adic numbers, and $\cO_\nu = \Z_p$ is the ring of $p$-adic integers. For $a \in \Q$ or $\Q_p$, $\mathrm{ord}_pa$ is simply the power of $p$ in the prime factorization of $a$. %{\color{purple} Henry: we could add an example where the uniformiser is not just the prime, for example with $\Q_5(\sqrt{5})$} {\color{orange} I think it's best if curious readers consult number theory texts or AI. One doesn't need to understand local rings too well in the rest of the paper}
\newline

An infinite place (or prime), denoted $\sigma \mid \infty$, corresponds to either a real field embedding $F \rightarrow \R$, or a conjugate pair of complex embeddings $F \rightarrow \C$. Let $r_1,r_2$ be the number of real and complex embeddings, respectively. Following the common convention, we let $\sigma_1, \ldots, \sigma_{r_1}$ denote the real embeddings, and $\sigma_{r_1+j} = \bar\sigma_{r_1+r_2+j}$ for $1 \leq j \leq r_2$ denote the pairs of complex embeddings. To clarify, $F$ is a field of degree $d = r_1+2r_2$, and has $r=r_1+r_2$ infinite places. If $\sigma$ is real (complex), $F_\sigma \cong \R$ $(\cong \C)$. The corresponding absolute value, in both real and complex cases, is $|a|_\sigma = |\sigma(a)|$.

\subsubsection*{Example}
If $F=\Q$, there is only one infinite place, corresponding to $\Q \hookrightarrow \R$ by $x \mapsto x$. If $F$ is a real quadratic field, it has two real places, corresponding to $x \mapsto x$ and $x \mapsto \bar x$, where $\bar x$ is the Galois conjugate of $x$. If $F$ is an imaginary quadratic field, these two embeddings form a conjugate pair of a single complex place of $F$.

\vspace{4mm}

For any $\Z$-basis $\{b_1, \ldots, b_d\}$ of $\cO_F$, the \emph{discriminant} of $F$ is the quantity
\begin{equation*}
\Delta_F = \left(\det (\sigma_i(b_j))_{1 \leq i,j \leq d} \right)^2.
\end{equation*}
This definition is independent of the choice of the basis. 
%It is known that $|\Delta_F|^{1/2}$ is the covolume of the lattice $\rho(\cO_F)$ in $F \otimes_\Q \R$ with respect to the metric above.
There exists also the \emph{logarithm map} $\Log : F^* \rightarrow \R^{r}$ defined by
\begin{equation*}
\Log(x) = (\log |\sigma_1(x)|, \ldots, \log |\sigma_{r_1}(x)|, 2\log |\sigma_{r_1+1}(x)|, \ldots, 2\log |\sigma_{r_1+r_2}(x)|).
\end{equation*}
%Its kernel is $\mu_F$ {\color{purple} Henry: I'm not sure this is correct? should be ok if the domain is $\cO_F^*$} {\color{orange} yes you're right, thanks}, the set of the roots of unity in $\cO_F$; we write $|\mu_F| = w_F$. 
$\Log$ takes $\cO^*_F$ to a lattice in $\R^{r}$ of rank $r-1$, called the \emph{unit lattice}; it is isomorphic to $\cO_F^*/\mu_F$, where $\mu_F$ is the group of the roots of unity in $\cO_F$. Its covolume, with respect to the standard metric on $\R^{r}$, is $R_F\sqrt{r}$, where $R_F$ is called the \emph{regulator} of $F$. 
%{\color{purple} Henry: For now the regulator is unused, also the definition here is off by a small factor $\sqrt{r}$?} {\color{orange} yes you're absolutely right,thanks. What I really wanted to emphasize was that when I think of the metric on unit lattice, I think of the standard metric on $\R^r$ restricted to it}

\subsection{Adeles} \label{sec:adeles}
The \emph{adele ring} $\A_F$ is defined to be the subset of the product
%\begin{align*}
$\prod_{\nu \in \cal P} F_\nu$
%\end{align*}
consisting of elements $(a_\nu)_{\nu \in \cal P}$ such that all $a_\nu \in \cO_\nu$ for all but finitely many $\nu \nmid \infty$. Here $(a_\nu)_{\nu \in \cal P}$ is referred to as an \emph{adele}. The natural embeddings $F \hookrightarrow F_\nu$ induce another natural embedding $F \hookrightarrow \A_F$ by $a \mapsto (a_\nu)$. 
Henceforth we identify $F$ with the image of this map.

$\A_F$ is equipped with the topology with a basis consisting of sets of the form $\prod_\nu O_\nu$, where $O_\nu$ is an open subset of $F_\nu$, and $O_\nu=\cO_\nu$ for all but finitely many finite $\nu$.

The set of \emph{finite adeles} $\A_\f$ is defined analogously, except that it is a subset of the product
%\begin{align*}
$\prod_{\nu \nmid \infty} F_\nu$
%\end{align*}
over the finite places of $F$. The set of \emph{infinite adeles} $\A_\infty$ is simply the (finite) product
%\begin{align*}
$\prod_{\nu | \infty} F_\nu.$
%\end{align*}

For $x \in \A_F$, we denote by $x_\f$ and $x_\infty$ the images of $x$ by the natural projections $\A_F\rightarrow \A_\f$ and $A_F\rightarrow\A_\infty$, respectively. We extend this notation naturally to other adelic constructs like $\A_F^n$ and $\GL(n,\A_F)$.

The \emph{product formula} states that, for $a \in F^*$,
\begin{align*}
\prod_{\nu \nmid \infty}|a_\nu|_\nu \cdot \prod_{\sigma \mid \infty} |a_\sigma|^{e_\sigma}=1,
\end{align*}
where $e_\sigma = 1$ or $2$ according to whether $\sigma$ is real or complex.

\subsection{Metric on $\A_\infty$}

Recall the canonical embedding $\rho$ of $F$ into $\R^{r_1} \times \C^{2r_2}$, defined by
\begin{equation*}
\rho(x) = (\sigma_1(x), \ldots, \sigma_d(x)).
\end{equation*}
The image $\rho(F)$ spans a vector space of dimension $d$ over $\R$, which we can identify with three sets at the same time: the subspace (isomorphic to $\R^{r_1}\times\C^{r_2}$) of $\R^{r_1} \times \C^{2r_2}$ consisting of all elements whose $i$-th and $(i+r_2)$-th complex entries are conjugates for all $i$, the tensor product $F \otimes_\Q \R$, and the set of infinite adeles $\A_\infty$. For $x \in F$ or $\A_F$, we sometimes write $x_\infty$ for the corresponding element of $\A_\infty$.

We endow an inner product on $\A_\infty \cong F\otimes_\Q\R$ by restricting to it the standard inner product on $\R^{r_1} \times \C^{2r_2}$, so that
\begin{equation*}
\langle \rho(x), \rho(y) \rangle = \sum_{i=1}^d \sigma_i(x)\bar\sigma_i(y),
\end{equation*}
and
\begin{equation*}
\|\rho(x)\|^2 = \langle \rho(x), \rho(x) \rangle = \sum_{i=1}^d |\sigma_i(x)|^2.
\end{equation*}

We extend the embedding $\rho$, the above inner product and the norm to the multi-dimensional counterparts in the natural way, and denote them by the same notations.

\subsection{How an adele gives a lattice}\label{sec:L(g)}
\ignore{
\phong{I have another silly question.
Since we "often" use the set $F^n g$, 
wouldn't it be more natural to call it the adelic lattice $L_{\A}(g)$, and define $L(g)$ from $L_{\A}(g)$ directly, without going through $I(g)$?
Then we can define our adelic first minimum (the minimal non-zero height) and list all the properties which generalize adelically,
like the minimum Gram-Schmidt norm.
}
{ \color{orange}How about the following? If you like it I'll change the rest of section accordingly:

For $g \in \GL(n,\A_F)$, the \emph{module lattice} associated to $g$ is given by (see e.g. \cite{BV83})
\begin{align*}
L(g) = \{(xg)_\infty \in \A_\infty^n:x\in F^n, (xg)_\nu \in \cO_\nu^n \mbox{ for all finite $\nu$}\}.
\end{align*}

I kind of like it because it trivializes the lemma below.

}
\phong{Yes. I was wondering what is the correct name: I've seen the word rigid adelic space, I don't know if it's correct}}
For $g \in \GL(n,\A_F)$, the \emph{module lattice} associated to $g$ is given by (see e.g. \cite{BV83})
\begin{align*}
L(g) = \{(xg)_\infty \in \A_\infty^n:x\in F^n, (xg)_\nu \in \cO_\nu^n \mbox{ for all finite $\nu$}\},
\end{align*}
whose first minimum
is $\lambda_1(L(g)) = \min_{u \in L(g), u \ne 0} \| u\|$.
The following lemma is immediate from this definition.
\begin{lemma}
For any $h \in \prod_{\nu \nmid \infty}\GL(n,\cO_\nu)$ and $k \in \GL(n,F)$, $L(g)=L(kgh)$.
\end{lemma}
\ignore{
\begin{proof}
The first claim is clear from the definition. For the second, observe that $x \in I(kg) \Leftrightarrow xk \in I(g)$. Thus
\begin{align*}
L(kg) &= \rho(I(g)k^{-1}) \cdot (kg)_\infty    \\
&= \rho(I(g))k^{-1}_\infty(kg)_\infty \\
&= \rho(I(g)) g_\infty = L(g).
\end{align*}
The second equality above holds because $\rho(xk^{-1}) = \rho(x)k^{-1}_\infty$ for any $x \in F^n$.
\qed
\end{proof}
}
\begin{remark}
In fact, it holds that $L(g) = L(g')$ if and only if $g,g'$ represent the same coset in $$\GL(n,F) \backslash \GL(n,\A_F) / \prod_{\nu \nmid \infty} \GL(n,\cO_\nu).$$
See e.g., \cite[Sec. 2.3.2]{dBPTW26} (where they factor on the right by a little more to normalize the determinant and rotational factor).
\end{remark}
\ignore{
\phong{I have a stupid question: later on, we use row notations for $g$: we look at the height of row vectors.
In that case, we would have expected that by multiplying to the left by a unimodular matrix (over the ring), we preserve the lattice. But here, it's the opposite: we multiply to the right.
}
{\color{orange} yes this is one weird thing about adeles. the role you're saying is played by $\GL(n,F)$ which indeed acts on the left. The right quotient controls what's sometimes called a ``level.'' For example, in the classical situation, sometimes people consider $\SL(n,\R)$ quotiented on the left by a finite index subgroup $\Gamma$ of $\SL(n,\Z)$. When expressing this adelically, it may sound strange, but the correct thing to do is to quotient by the same $\GL(n,F)$ on the left and change the group on the right.}
}
In the classical language, suppose we are given a pseudobasis of a module lattice of the form
%\begin{align*}
$M=\mathfrak b_1 \vv_1 \oplus \ldots \oplus \mathfrak b_n \vv_n$
%\end{align*}
where $\vv_i \in (F\otimes\R)^n$, say. In the adelic language, $M$ is precisely $L(g)$, where
\begin{align*}
g_\infty = \begin{pmatrix} - & \vv_1 & - \\ & \vdots & \\ - & \vv_n & - \end{pmatrix}, g_\nu = \begin{pmatrix} \pi_\nu^{-\ord_\nu \mathfrak b_1} & & \\ & \ddots & \\ & & \pi_\nu^{-\ord_\nu \mathfrak b_n} \end{pmatrix}
\end{align*}
for each finite $\nu$. 
\ignore{
In particular, notice that
$$I(g) = \{(b_1,\ldots,b_n)\in F^n : b_i \in \gb_i\}.$$}
%{\color{purple} Henry: Maybe add here that $I(g)$ is just another encoding of the $\mathfrak b_i$ from the pseudobasis} {\color{orange} is this good?} {\color{purple}fine for me, but its annoying to have to introduce another notation so not sure if I like this in the end:)} {\color{orange} gave it a second try}

%Notice $H(\det g) = \prod_{i=1}^n N(\mathfrak b_i) \cdot \prod_{\sigma \mid \infty} |\det g_\sigma|$.

%In this light we can interpret Thunder's upper bound on the height of the densest rank $k$ sub-$\cO_F$-module
%\begin{align*}
%|\gamma_{n,k}(F)|^{\frac{d}{2}} H(\det g)^\frac{k}{n}.
%\end{align*}
%(In our paper height is normalized differently than in Thunder's.)

\subsection{Height}\label{sec:height}
Define the (local) \emph{height} at $\nu$ of $v_\nu = (a_1,\ldots,a_n) \in F_\nu^n$ by
\begin{align*}
H_\nu(v_\nu) = \begin{cases} \max_i N(\nu)^{-\ord_\nu a_i} & \mbox{if $\nu$ is finite} \\ \sqrt{a_1^2 + \ldots + a_n^2} & \mbox{if $\nu$ is real} \\ |a_1|^2 + \ldots + |a_n|^2 & \mbox{if $\nu$ is complex.} \end{cases}
\end{align*}
For $v \in \A_F^n$, its (global) \emph{height} is defined to be
\begin{align*}
H(v) = \prod_{\nu \in \cal P} H_\nu(v_\nu).
\end{align*}
The product formula implies that 
$H(v)= H(rv)$ for any $r \in F^*$. %{\color{purple} Henry: Note to self: this definition of height corresponds to $\|\cdot\|_{\A_F}$ on id\`eles}

Sometimes it is convenient to separate the finite and infinite parts of the height. We denote them by $H_\f = \prod_{\nu \nmid \infty} H_\nu$ and $H_\infty = \prod_{\nu \mid \infty} H_\nu$, respectively.

\begin{lemma}\label{lemma:height}
For $g \in \GL(n,\A_F)$, $\det L(g) = |\Delta_F|^\frac{n}{2}H(\det g).$ Here $\det g = \prod_\nu \det g_\nu$.
\end{lemma}
\begin{proof}
This follows from e.g. \cite[Lemma 2.3]{Kim23}.
%{\color{orange} In general, if you have a submodule $I$ of $g$ of rank $k$, then its det = its height $\cdot\sqrt{\Delta_F}^k$} \henry{should we add the definition of the height for a submodule in the prelims? using wedge products for example} {\color{orange} I feel it's best to minimize the amount of knowledge imposed on the reader to understand the paper. Perhaps I'll add a reference for readers interested in details}
\qed
\end{proof}
We introduce the minimal height (which corresponds to the adelic first minimum introduced in~\cite{BV83}):
for $g\in \GL(n,\A_F)$, let 
\[\eta(g) =\min_{v \in F^ng\backslash\{0\}}H(v).\]
It is possible to view $\eta$ from the perspective of pseudobases:
\begin{lemma}\label{lemma:module}
Let $M$ be a module lattice given by a pseudo basis $M=\mathfrak b_1 \vv_1 \oplus \ldots \oplus \mathfrak b_n \vv_n$ 
where the $\vv_i \in (F\otimes\R)^n$'s are linearly independent, 
and the $\mathfrak b_i$'s are fractional ideals.
Let $g \in \GL(n,\A_F)$ such that $M=L(g)$.
Then:
$$ \eta(g) = \min_{\substack{\text{nonzero } \vv \in (F\otimes\R)^n \\  \text{fractional ideal } {\mathfrak b}, {\mathfrak b} \vv \subseteq M }} N(\mathfrak b) H_\infty( \vv),$$
which is $\le \min_{1 \le i \le n} N(\mathfrak b_i) H_\infty( \vv_i).$
Furthermore:
$$ d^{-d/2} |\Delta_F|^{-1/2} \lambda_1(M)^d \le \eta(g) \le d^{-d/2} \lambda_1(M)^d \le n^{d/2} (\det M)^{1/n}$$
\end{lemma}
\begin{proof}
\cite[Lemma 2.2]{LPSW19} proved inequalities relating $\lambda_1(M)$
and a quantity called $\lambda_1^{\cN}(M)$, which is slightly different from $\eta(g)$,
but the same arguments show our inequalities. \qed
\end{proof}
% I chose \eta because it's h in greek
The last inequalities of Lemma~\ref{lemma:module} show that finding a short vector in $M$ (with respect to either $\lambda_1$ or $\det$) can be reduced to finding a short vector in 
the rank-1 module lattice $\mathfrak b_1 \vv_1$ (with respect to its $\det$),
provided that $N(\mathfrak b_1) H_\infty( \vv_1)$ is close to $\eta(g)$.

\ignore{
\begin{remark}
The notion of height allows some flexibility. For example, heights at infinite places can be replaced by any $L^p$-norm. They can also be twisted by a linear transformation, in the sense of \cite{Thu93} It may be interesting to ask whether some heights are more suitable than others for algorithmic purposes.
\end{remark}
}

\subsection{Iwasawa decomposition in $\GL(n,\A_F)$} \label{sec:Iwasawa}

For any place $\nu$, $\GL(n,F_\nu)$ admits an \emph{Iwasawa decomposition} $g_\nu = u_\nu a_\nu k_\nu$, where $u_\nu$ is a lower-triangular unipotent matrix, $a_\nu$ is diagonal and with positive real entries when $\nu \mid \infty$, and $k_\nu \in K_\nu$, where $K_\nu =\GL(n,\cO_\nu), \mathrm{O}(n,\R)$ or $\mathrm{U}(n,\C)$ according to whether $\nu$ is finite, real, or complex. In the real or complex case, computing this decomposition comes down to the Gram-Schmidt orthogonalization. Note that, over the infinite places, this coincides with the Gram-Schmidt orthogonalization on $(F\otimes\R)^n$ commonly seen in the cryptology literature (e.g. \cite[Sec. 2.3]{LPSW19}).\ignore{{\color{purple} Henry: I think this discrepancy is interesting and maybe deserves more details but we can discuss this later depending what the target audience ends up being} {\color{orange} actually I think I was initially wrong and they actually coincide --- so fixed. Perhaps help me check this?} {\color{purple} I think I agree. Checked it for n=2}} When $\nu \nmid \infty$, the decomposition can be made unique by dictating that the diagonal entries $a_{i,\nu}$ of $a_\nu$ are powers of $\pi_\nu$, and that the $(i,j)$-entry of $u_\nu$ is to vanish under the natural projection $F_\nu \rightarrow (a_{i,\nu}/a_{j,\nu})\cO_\nu$ --- \emph{cf.} \cite[Claim 2.2.1]{Gar18}.

\ignore{
When $\nu$ is finite, there exists an analogous process. {\color{orange} 

Here's how it's done, in case $n=2$. For $\nu \nmid \infty$, take a matrix
\begin{align*}
g=\begin{pmatrix}
x & y \\ z & w
\end{pmatrix} \in \GL(2,F_\nu).
\end{align*}
I claim we can find $k=\GL(2,\cO_\nu)$ such that $g=bk$ and $b$ is a lower triangular matrix. First, if $\ord_\nu x > \ord_\nu y$, then we multiply $g$ from the right by
\begin{align*}
\begin{pmatrix}
 & 1 \\ 1 &
\end{pmatrix}.
\end{align*}
Thus we can assume without loss of generality that $\ord_\nu x \leq \ord_\nu y$, that is, $y/x \in \cO_\nu$. Thus multiplying $g$ by
\begin{align*}
\begin{pmatrix}
1 & -y/x \\ 0 & 1
\end{pmatrix} \in \GL(2,\cO_\nu),
\end{align*}
we can write $g=bk$ as claimed. Let's write
\begin{align*}
b=\begin{pmatrix}
x & 0 \\ z & v
\end{pmatrix}.
\end{align*}
Then
\begin{align*}
b
= \begin{pmatrix}
1 & 0 \\ z/x & 1
\end{pmatrix}
\begin{pmatrix}
x & 0 \\ 0 & v
\end{pmatrix},
\end{align*}
which gives the desired decomposition of $g$. Tracking our computations, one sees that $H_\nu(x,0) = H_\nu(x,y)$ and $H_\nu(z,v) = H_\nu(z,w)$.
}
}

In case $n=2$ for example, collecting the local decompositions together, one obtains an Iwasawa decomposition of $g \in \GL(2,\A_F)$ , where
\begin{align}\label{eq:iwasawa2}
g=uak, \text{ } u = \begin{pmatrix} 1 & \\ m_{21} & 1 \end{pmatrix}, \text{ } 
a = \begin{pmatrix} a_1 & \\ & a_2 \end{pmatrix}, \text{ }
k \in K := \prod_\nu K_\nu.
\end{align}
Here $m_{12}$ is an adele, and $a_1,a_2$ are \emph{ideles}, i.e. $(a_i)_\nu \in \cO_\nu^*$ for all but finitely many $\nu \nmid \infty$. 
\ignore{{\color{purple} Henry: I don't think this decomposition is unique right? I think we should explain in more detail how to come up with the decomposition, maybe written as an algorithm, mirroring a GS process } {\color{orange} If one fixes the ``sign'' of $a$, then it becomes unique. It's not really an important difference. I once had an explanation of the local version, but muted it, since I felt it doesn't really help understand the rest of the paper}{\color{purple} so fixing "sign" means taking them mod p-adic units ? Even then, playing around with the $F=\Q$ case and I feel like I can find multiple possible values of $m_\nu$ when $\nu$ divides the determinant of the matrix (to make things more precise I can multiply $k_\nu$ by a lower triangular unipotent with some off-diagonal $x\in\cO_\nu$, and in exchange this changes $u_\nu$ by shifting $m_\nu$ by $\alpha x$), but maybe that's cleared out by size reduction or I'm being dumb or something. In usual QR decomposition proofs, we use that orthogonal + triangular + signs has only one element, but that's not the case at finite places}
{\color{orange} oh no, you're not being dumb, quite the opposite, you made a very good point. I made the necessary clarifications above and below. Again for the algorithm it's not important that the expression is unique. And I kind of want to keep it that way for now because I don't want to introduce an extra condition like $\alpha$ needing to be integral or something.}} Below, whenever we discuss an element $g \in \GL(n,\A_F)$, the choice of the Iwasawa decomposition $g=uak$ will be given from the beginning, which is to be updated as it undergoes changes by our algorithm. Additionally, we write $v_i$ for the $i$-th row vector of $g$, and $\alpha_i = a_{i+1}/a_i$. In the $n=2$ case, we drop some of the subscripts, when there exists no room for ambiguity.

We state and prove a lemma generalizing the fact that any nonzero vector of a lattice spanned by the rows of $g \in \GL(n,\R)$ has length at least the minimal Gram-Schmidt coefficient of $g$.
\begin{lemma}\label{lemma:diagonal}
Let $g\in\GL(n,\A_F)$ have an Iwasawa decomposition $g=uak$, and denote $a = \diag(a_1,\ldots,a_n)$. Then
\begin{align*}
\eta(g) %\min_{v \in F^ng\backslash\{0\}}H(v)
\geq \min_{1 \leq i \leq n} H(a_i).
\end{align*}
\end{lemma}
\ignore{
{\color{orange} if in addition $g$ is reduced, is it possible to try enumerating points on module lattices?
}\phong{We would enumerate on the coefficients.
Can we find $u \in F^n\backslash\{0\}$
such that $H(ug)$ is minimal,
by "enumerating" the entries of $u$?}{\color{orange} Phong please elaborate (or maybe not) on this in paper, thanks}
\phong{It does seem that if $g$ is reduced and $H(ug) = \min_{v \in F^ng\backslash\{0\}}H(v)$
then the projective height $H(u)$ is bounded, so by Northcott's theorem, 
the number of $u$ (in the projective plane) is bounded.
Of course, this does not mean that this enumeration is computationally interesting:
we need to work out the bounds}}
\begin{proof}
Choose a vector $w=(w_1,\ldots,w_n) \in F^n$, and let $j$ be the largest index such that $w_j \neq 0$. First, notice that for each place $\nu$, right multiplication by $K_\nu$ preserves the local height at $\nu$. Therefore
%\begin{align*}
$H(wuak)=H(wua).$
%\end{align*}
Since the $j$-th entry of $wua$ is $w_ja_j$, we have
\begin{align*}
H(wua) \geq H(w_ja_j) = H(a_j),
\end{align*}
where the last equality is by the product formula.
\qed
\end{proof}

\subsection{Complexity}

We need to clarify the meanings of certain expressions to discuss the complexity of our algorithms below. Assume the knowledge of a basis $\mathcal B$ of $\cO_F$. An element $x \in F$ can be written as a $\Q$-linear combination of the elements of $\mathcal B$; by the \emph{bitsize} of $x$, we mean the number of bits needed to record this linear combination. Similarly, a fractional ideal of $F$ can be represented as its $\Z$-basis, which consists of $d$ elements of $F$; accordingly, its bitsize is the number of bits needed to record those $d$ elements. The computer representation of adeles will be discussed later, see section \ref{subsec:setup}.

By $\mathrm{poly}(x,y,z,\ldots)$ we mean a polynomial in variables $x,y,z,\ldots$. In this context, $\mathrm{poly}(\text{input)}$ means a polynomial in the total bitsize of all necessary input data for the algorithm under discussion.

\section{Adelic LLL in rank 2}

\subsection{Reduction notion}\label{subsec:reduction_notion}

Fix the parameters $0 < \mu, \delta < 1$ such that $\delta^\frac{2}{d}-\mu^2 > 0$. We need other parameters $A,B,C > 0$; their choice affects the quality-complexity tradeoff, and will be discussed more carefully below. Recall that we use the convention that $m$ and $\alpha$ are implicitly defined via the corresponding Iwasawa decomposition, such as~\eqref{eq:iwasawa2}.
\begin{definition}[Rank 2 reduced basis]\label{def:LLL-reduced-2}
We say that $g \in \GL(2,\A_F)$ is \emph{(LLL-)reduced} with respect to the above parameters if it has an Iwasawa decomposition $g= uak$ for which the following points hold
\begin{enumerate}[label=(\roman*)]
\item (size reduced) 
%$H_\f((m,1))^\frac{1}{d}\max_{\sigma \mid \infty}|m_\sigma| \leq \mu$, and $H_\f((m,1)) \leq C$ %{\color{orange} Notice that this is a Diophantine approximation statement.}
There exists $\tau_\sigma > 0$ for each of the $d$ embeddings $\sigma:F\hookrightarrow\C$ such that
\[\tau_\sigma \leq C^\frac{1}{d}, \text{ } \tau_\sigma|m_\sigma| \leq \mu \text{ and }
H_\f(m,1) \leq\prod_{\sigma:F\hookrightarrow\C}\tau_\sigma.\]

\item (Lov\'asz condition) $\delta \leq H((m,\alpha))$, or equivalently $\delta H((a_1,0)) \le H((m a_1,a_2))$.
%{\color{orange} The ``Siegel condition'' would be $\delta \leq H(\alpha)$}

\item (unit reduced) The entries of %\henry{I don't think $\Log$ was defined for idèles.} {\color{orange} sorry, fixed}
$$\Log(\alpha_\infty) - \frac{1}{d}\log H_\infty(\alpha_\infty)\cdot(\underbrace{1,\ldots,1}_\mathrm{real},\underbrace{2,\ldots,2}_\mathrm{complex})$$ are bounded by $A$.

\item (class reduced) $\alpha$ is integral (that is, $\alpha_\nu\in\cO_\nu$ for all finite $\nu$) and $B^{-1} \leq H_\f(\alpha)$.
\end{enumerate}
\end{definition}

This definition (as well as its rank $n$ counterpart Definition \ref{def:LLL-reduced-n}) is inspired as much by the reduction theory of $\GL(n,\A_F)$ (see e.g. \cite[Sec. 3.3]{Gar18}) as by the original LLL algorithm \cite{LLL82}. Indeed, just as the set of the classical LLL bases approximates the fundamental domain of $\SL(n,\Z)\backslash\SL(n,\R)$, the set of reduced bases as defined in this paper approximates the fundamental domain of $\GL(n,F) \backslash \GL(n,\A_F) / \prod_{\nu \nmid \infty} \GL(n,\cO_\nu)$. The corresponding conditions for the classical LLL algorithm \cite{LLL82} are consistent with the above in case $F=\Q$ and $A=0, B=1, C=1$, other than the minor difference that $\delta$ here is $\delta^2$ in some references.

% temporarily muting comments
\ignore{
{\color{purple} Henry: add a remark relating adelic-LLL to classical LLL, in particular, is $\delta$ correctly defined? (in the sense that it would be nice to recover LLL exactly if we set $F=\Q$)} {\color{orange} is this satisfactory? yes $\delta$ is same as in original LLL, it's $\mu$ that's a little different, but if $C=1$ then $\mu$ is forced to be $\geq 1/2$} {\color{purple} In my understanding the size reduction condition $C=1$ disallows denominators, and then the correct $\mu$ is $1/2$ (is this wrong?), from which we get $\delta\in(1/2,1)$, which is not as in the original LLL.} {\color{purple} I'm not yet convinced that the Lov\'asz condition is a direct generalisation of the LLL one : here is why: if $C=1$ we have no denominators, so $H_\infty(m,\alpha)=\sqrt{|m|^2+\alpha^2}$, and the finite part is just $1$, so the condition is $\delta\le \sqrt{|m|^2+\alpha^2}$, but in the classical LLL $\alpha=\frac{\|\mathbf b_2^*\|}{\|\mathbf b_1^*\|}$ and $m=\mu_{21}$ and the Lov\'asz condition can be written $\delta\le \mu_{21}^2+\frac{\|\mathbf b_2^*\|^2}{\|\mathbf b_1^*\|^2}$. Unless I misunderstood a definition, the classical $\delta$ is the square root of the adelic one, is this intended? I'm guessing yes, but there is a subtlety worth explaining here} {\color{purple} On another topic, should we mention Lagrange-Gauss? So $\delta\rightarrow1$ but maybe out of scope.} 
{\color{purple} Again regarding condition (ii), shouldn't the finite height also take $1$ into account as in $H_\f((m,\alpha,1))$ ? Otherwise if $m$ and $\alpha$ are both divisible by some $p$, then the height decreases (it should not, right?). Make the finite height projective?}

{\color{orange} I guess $\delta$ differing by square isn't really a big deal.. but I'll add a comment. Not sure if I want to mention Lagrange-Gauss --- I don't know if this touches on what you have in mind, but (i) it's unclear $\delta=1$ will get the algorithm to work (ii) the geometry of the number field is involved so it's unclear how well the set of bases approximate the fundamental domain --- if it's good, then that's another source of threat for module schemes, and it's a subject of further research. The condition (ii) is correct as written --- since the proofs in the next section go through.}
}
For (ii), we could have also used the Siegel condition as in \cite{LPSW19}, but just as in the classical case, the Lov\'asz condition is more likely to yield superior output quality in practice, since it induces a strictly stronger reduction notion.

% temporarily muting comments
\ignore{
\phong{Can you please correct me if I'm wrong?
The condition (ii) can be rewritten as $\delta H((a_1,0)) \le H((m a_1,a_2))$. 
Then the analogy with Lagrange/Gauss/LLL is direct.
In the Lagrange case,
we have two basis vectors $\vec{b}_1$ and $\vec{b}_2$.
$H(a_1,0)$ is the analogue of $\|\vec{b}_1\|$
and $H((m a_1,a_2))$ is the analoge of $\|\vec{b}_2\|$.
So we're asking that $\delta \|\vec{b}_1\| \le \|\vec{b}_2\|$,
which is exactly the relaxed Lagrange algorithm,
which was introduced by Vall\'ee in 1991:
in the real Lagrange, $\delta=1$,
but Vall\'ee introduced a virtual $\delta < 1$
for her worst-case complexity analysis:
we imagine a virtual Lagrange with $\delta < 1$ first,
then eventually $\delta=1$.
The Lagrange $\delta < 1$ is obviously polynomial time
because the norms decrease geometrically.
And once the basis is "weakly" reduced, Vall\'ee showed
that only $O(1)$ steps remain for the real Lagrange.
So here, the condition is really the analogue of relaxed Lagrange
(which is also 2-dimensional LLL). It also brings the open question:
in the Lagrange case, a small post-processing
is enough to move from a $\delta$-reduced basis
to a Minkowski-basis, but your comment suggests that
we don't know a similar phenomenon in the number field case.
}
{\color{orange} your first question --- yes that is correct. I'm a little afraid of bringing up Lagrange-Gauss since I guess it's supposed to be ``perfect'' reduction, but the geometry of the number field significantly gets in the way for adelic LLL to do that. Getting over this is an interesting future work. Over totally real fields the fundamental domain is well understood. CM fields, like cyclotomic fields, are much less so, though some math people do wish to understand it. You might want to briefly google ``fundamental domain of Bianchi groups'' and see the pretty pictures.}
}

It is possible to define the reduction notion with respect to a unique choice of Iwasawa decomposition, such as the one discussed in Section \ref{sec:Iwasawa} above. Though this would lead to a theoretically more satisfying version of adelic LLL, it is not at all necessary for our algorithm to work, and may likely add additional complications to the procedure. 
% temporarily muting comments
\ignore{
{\color{purple} I really like this comment! On the implementation side, it also means the algorithm is not fully specified if we don't explain exactly how to select a representative. Maybe we can explain this in an appendix (although I'm not a big fan of appendices). Update, not sure we need this in the implementation actually so ok to keep as is, but still it is strange to call it \emph{the} Iwasawa decomposition if it's not unique.}

{\color{orange} yes in the actual implementation we start out knowing the decomposition. If I said \emph{the} it was a mistake.}
}

\subsection{Overview of the algorithm}

We continue with the notations set up above. In addition, we denote the necessary subroutines as follows:
\begin{itemize}[label=\textbullet]
\item \texttt{Lovasz}$(g,\delta)$ returns \texttt{true} if $g$ satisfies the Lov\'asz condition, and \texttt{false} otherwise.
\item \texttt{class-reduce}$(g,B)$ ensures that $\alpha$ is integral and $B^{-1} \leq H_\f(\alpha)$.
\item \texttt{unit-reduce}$(g,A)$ ensures that $\alpha$ is unit reduced with respect to $A$, while preserving class reducedness.
\item \texttt{size-reduce}$(g,\mu,C)$ ensures that $g$ is size reduced with respect to the given parameters, while preserving class and unit reducedness. %{\color{purple} Henry: Not sure about notation = here. Maybe just say something like: if the parameters are clear from context, we write \texttt{size-reduce}$(g)$} {\color{orange} fixed to your liking :)}
\end{itemize}
When the parameters are clear from the context, we omit those and simply write \texttt{size-reduce}$(g)$, and so on.
\ignore{
\phong{I think it would help to say what kind of operations on $g$ is performed by these subroutines, so that we know
what is preserved: which uses unimodular transforms,
and which uses field transforms.
For instance, in LLL, we say that size reduction does not change the Gram-Schmidt vectors, but it may change the Gram-Schmidt coefficients and the basis vector norms.
} {\color{orange} done}
\phong{Thanks, so we never multiply $g$ to the right, using some $h$?} {\color{orange} that's correct, except in adelic swap}
\phong{I'm a bit confused. 
A swap of rows could be obtained by left multiplication,
but here, the adelic swap does make a right multiplication.
}{\color{orange} that's for an efficient implementation, not for conceptual reason}
}

The last three operations amount to multiplying $g$ from the left by an element of $\GL(2,F)$: for \texttt{class-reduce}, by a diagonal matrix; for \texttt{unit-reduce}, by a diagonal matrix with unit entries; for \texttt{class-reduce}, by a lower unipotent matrix. This helps one see the necessary preservation properties. In the following section, we will substantiate these functions in case the input is given by a pseudobasis.

Algorithm \ref{alg:lll_c} provides a high-level outline of the adelic LLL. At this point, we already have sufficient detail for us to prove its termination and output quality bounds.

\begin{algorithm}
\caption{Adelic LLL algorithm rank 2 (conceptual version)} \label{alg:lll_c}
\begin{algorithmic}[1]
\Require Parameters $\delta,\mu,A,B,C$ and $g=uak \in \GL(2,\A_F)$.
\Ensure LLL-reduced $g \in \GL(2,\A_F)$.
\State \textbf{while} true \textbf{do}
\State \ \ \texttt{class-reduce}$(g)$
\State \ \ \texttt{unit-reduce}$(g)$
\State \ \ \texttt{size-reduce}$(g)$ %// $g$ stays class and unit reduced
\State \ \ \textbf{if}(\texttt{Lov\'asz}$(g)$) \textbf{break loop}
\State \ \ \textbf{else} swap the rows of $g$
\State \textbf{return} $g$
\end{algorithmic}
\end{algorithm}
%\vspace{-\baselineskip}

\begin{proposition}[Rank 2 termination]\label{prop:swap_bound}
Algorithm \ref{alg:lll_c} terminates. In fact, if for some $\kappa>0$, 
$\eta(g) \geq \kappa$, then it takes at most
%$\min_{v \in F^2g\backslash\{0\}}H(v) \geq \kappa$, then it takes at most
$$\frac{\log {H(v_1)} - \log{\kappa}}{-\log \delta}$$
swaps to terminate. (Recall that we are denoting by $v_1$ the first row vector of $g$.)
\end{proposition}
The above statement and its proof, as well as its generalization Lemma \ref{prop:swap_bound_n} in the next section, exactly parallel their classical counterparts. %\henry{It would be nice to compare $\min_{v \in F^2g\backslash\{0\}}H(v)$ with $\lambda_1^\mathcal{N}$ from LPSW (so as to help readers familiar with other notations)} {\color{orange} I'm not sure if they compare well; I may be mistaken but $\lambda_1^\mathcal{N}$ doesn't feel like an adequate notion to me}
Just as with the classical LLL, if we assume the input $g$ is given by an integral matrix, we can take $\kappa=1$. In general, if $g$ corresponds to a pseudobasis $\gb_1\vv_1+\gb_2\vv_2$ where $\vv_1,\vv_2\in\rho(\cO_F^2)$, then we can take $\kappa = N(\gb_1+\gb_2)$, as shown in Lemma \ref{lemma:lowest_ht} below.

\begin{proof}
From the descriptions provided above, it is clear that only Line 6, the row swapping, of Algorithm \ref{alg:lll_c} has the possibility of changing $E(g)$. 
%\phong{I'm a bit confused. Changing $E(g)$ is the same as changing the height of the first row of $g$. The height of an adelic vector (the entries are adeles) is invariant if I multiply the vector by a non-zero field element, but it's not necessarily invariant if I multiply to the left by an arbitrary matrix in GL($n,\F$). Why is it then that the subroutines which multiply to the left (class-reduction, unit-reduction, size-reduction) do not change the height of the first row? It must be special matrices. I'm really sorry for asking stupid questions.} 
%{\color{orange} you're correct if the matrix is arbitrary, but as clarified earlier, lines 2-4 are multiplications on the left by lower triangular matrices. These change the first row by a multiple of $F$ which does not change the height.}
%\phong{Thanks for the clarification. But in that case, why is the proof not simply using the height of the first row as a potential?} 
%{\color{orange} Yes it's wasteful writing. One may give a weak argument that it's a matter of taste since $E$ here as defined corresponds to what some math people refer to as the Weyl vector of $\GL_2$. I can get it changed.}

When a swap occurs, $H(v_1)=H((a_1,0))$ is replaced by $H(v_2)=H((m a_1,a_2))$. Since the Lov\'asz condition fails, we have $\delta H(v_1) >H(v_2)$, and thus $H(v_1)$ decreases at least by a factor of $\delta$. Since $H(v_1)$ cannot be lower than $\kappa$, this completes the proof.
\qed
\end{proof}
\begin{lemma}\label{lemma:lowest_ht}
Suppose $g \in \GL(2,\A_F)$ corresponds to the pseudobasis $\gb_1\vv_1+\gb_2\vv_2$
where $\vv_1,\vv_2\in\rho(\cO_F^2)$, i.e.,
\begin{align*}
g_\infty = \begin{pmatrix} \vv_1 \\ \vv_2 \end{pmatrix}, g_\nu=\begin{pmatrix} \pi_\nu^{-\ord_\nu\gb_1} & \\ & \pi_\nu^{-\ord_\nu\gb_2}\end{pmatrix}
\end{align*}
for each $\nu\nmid\infty$. Then
%\begin{align*}
$\eta(g) \geq N(\gb_1+\gb_2).$
%\min_{v \in F^2g\backslash\{0\}}H(v) \geq N(\gb_1+\gb_2).
%\end{align*}
\end{lemma}
\begin{proof}
Let us write $\beta_i$ for the finite adele such that $(\beta_i)_\nu = \pi_\nu^{\ord_\nu\gb_i}$, so that
\begin{align*}
g_\f = \begin{pmatrix} \beta_1^{-1} & \\ & \beta_2^{-1} \end{pmatrix}.
\end{align*}
Also let $\beta \in \A_\f$ be such that $\beta_\nu = \pi_\nu^{\min(\ord_\nu\gb_1,\ord_\nu\gb_2)}$. 

Take any $x_1,x_2 \in F$, not both zero. We have
\begin{align*}
H((x_1,x_2)g) = H_\f((x_1\beta_1^{-1},x_2\beta_2^{-1}))H_\infty(x_1\vv_1+x_2\vv_2).
\end{align*}
First notice
\begin{align*}
H_\f((x_1\beta_1^{-1},x_2\beta_2^{-1}))\geq H_\f(x_1\beta^{-1},x_2\beta^{-1}) \geq H_\f((x_1,x_2))H_\f(\beta^{-1}).
\end{align*}
% temporarily muting comments
\ignore{
{\color{brown} why is that true?
On the left, it's something like $\max(|x_1|_\nu \cdot |\beta_1^{-1}|_\nu, |x_2|_\nu \cdot |\beta_2^{-1}|_\nu)$.
On the right, it's the product of $ \max(|x_1|_\nu, |x_2|_\nu)$
with $\max(|\beta_1^{-1}|_\nu, |\beta_2^{-1}|_\nu)$.}

{\color{orange} $H_\f(x_1\beta_1^{-1},x_2\beta_2^{-1}) \geq H_\f(x_1\beta^{-1},x_2\beta^{-1})$, and you can pull out the scalar $\beta^{-1}$, so.. hope it helps}
}

As for $H_\infty(x_1\vv_1+x_2\vv_2)$, by considering a single nonzero coordinate of $x_1\vv_1+x_2\vv_2$, it is bounded from below by $H_\infty(y)$, where $y$ is some element of the ideal generated by $x_1,x_2$. By the product formula,
\begin{align*}
1 = H_\f(y)H_\infty(y) \leq H_\f((x_1,x_2))H_\infty(y).
\end{align*}
Combining the last two bounds, we have $H((x_1,x_2)g) \geq H_\f(\beta^{-1})$, which completes the proof.
\qed
\end{proof}

The following proposition shows that the first row vector of reduced $g$ is guaranteed to be short (in the sense of height) both absolutely~\eqref{ineq:Hermite} and relatively~\eqref{ineq:approx} , as with the classical LLL.
It does so by establishing an adelic analogue~\eqref{ineq:ratio}
of the well-known fact that in an LLL-reduced basis, the norms of the Gram-Schmidt
vectors decrease at most geometrically:
\begin{proposition}[Rank 2 quality] \label{prop:rhf_bound}
Let $A, B, C$ denote the parameters from Algorithm~\ref{alg:lll_c} and define
\begin{equation} \label{eq:Q}
  Q = \bigl(\delta^{2/d} - \mu^2\bigr)^{-d/4} (CB)^{1/2} e^{Ad/2}.
\end{equation}
If $g$ is reduced, then
\begin{equation} \label{ineq:ratio}
  \frac{H(a_1)}{H(a_2)} \le Q^2,
\end{equation}
\begin{equation} \label{ineq:Hermite}
  H(v_1) \le Q\, H(\det g)^{1/2},
\end{equation}
and
\begin{equation} \label{ineq:approx}
  H(v_1) \le Q^2 \eta(g).
%    H(v_1) \le Q^2 \min_{v \in F^2 \setminus \{0\} g} H(v).
\end{equation}
\end{proposition}
%\phong{sorry: I had mixed up the $a_i$'s with the $v_i$'s in \eqref{ineq:ratio}. It should be fine now, no? }
\begin{proof}
Suppose $g$ is reduced.
First, notice that \eqref{ineq:ratio} immediately implies \eqref{ineq:Hermite}, because $H(\det g) = H(a_1) H(a_2)$
and $H(v_1) = H(a_1)$.
%$ = ${\color{red} this is not true and correct analogue is %$=H(a_1)H(a_2)$}
%and also~\eqref{ineq:approx} by Lemma~\ref{lemma:diagonal}.
It thus suffices to prove \eqref{ineq:ratio},
which is equivalent to proving that $H(\alpha) \geq 1/Q^2$.
\ignore{\henry{I don't think this is equivalent: it is only implied. Indeed $H(v_1)=H(a_1)$ and $H(v_2)=H(ma_1,a_2)=H(a_1)H(m,\alpha)\ge H(v_1)H(\alpha)$, therefore proving $H(\alpha)\ge 1/Q^2$ implies \eqref{ineq:ratio}, and the identity $H(\det g) = H(a_1)H(a_2)=H(a_1)^2H(\alpha)$ then implies \eqref{ineq:Hermite}.}}
The class reduced condition already implies a lower bound on $H_\f(\alpha)$, namely:
%\begin{align*}
$B^{-1} \leq H_\f(\alpha),$
%\end{align*}
therefore it remains to get a lower bound on $H_{\infty}(\alpha)$.
The Lov\'asz condition implies that
\begin{align*}
\delta \leq H((m,\alpha)) = H_\f((m,\alpha))H_\infty((m,\alpha)).
\end{align*}
The right-hand side is at most
\begin{align} \label{eq:lossy}
 H_\f((m,1))\prod_{i=1}^d\left(|\sigma_i(m)|^2 + (e^AH_\infty(\alpha)^\frac{1}{d})^2\right)^\frac{1}{2}
\end{align}
by the unit reduced condition and the fact that $H_\f((m,\alpha)) \leq H_\f((m,1))$, since $\alpha$ is integral. Definition of size reduction implies
\begin{align*}
H_\f((m,1))\prod_{i=1}^d\left(|\sigma_i(m)|^2 + (e^AH_\infty(\alpha)^\frac{1}{d})^2\right)^\frac{1}{2}&\leq\prod_{i=1}^d\left((\tau_{\sigma_i}\sigma_i(m))^2 + (\tau_{\sigma_i}e^AH_\infty(\alpha)^\frac{1}{d})^2\right)^\frac{1}{2} \\
&\leq \left(\mu^2 + (C^\frac{1}{d}e^AH_\infty(\alpha)^\frac{1}{d})^2\right)^\frac{d}{2}.
\end{align*}
Thus
%{\color{purple} Henry: Did you remove the unit-reduced condition in Alg 1? I'm having a hard time following this proof, will go to sleep now and come back to it later...} {\color{orange} I did since it turned out unnecessary. But the argument here is precisely that one can pretend it's unit-reduced, because it does not change the height}
\begin{align*}
\delta^\frac{2}{d} \leq \mu^2 + (C^\frac{1}{d}e^AH_\infty(\alpha)^\frac{1}{d})^2,
\end{align*}
from which one can compute and deduce
\begin{align*}
&  (\delta^\frac{2}{d}-\mu^2)^{\frac{d}{2}}C^{-1}e^{-Ad} \leq H_\infty(\alpha).
\end{align*}
\ignore{
\Rightarrow& H_\infty(a_1) \leq (\delta^\frac{2}{d}-\mu^2)^{-\frac{d}{4}}C^\frac{1}{2}e^{\frac{Ad}{2}}H_\infty(\det g)^\frac{1}{2}.
\end{align*}
}
Hence, we proved that:
$$ H(\alpha) = H_\f(\alpha) H_\infty(\alpha) \ge B^{-1} (\delta^\frac{2}{d}-\mu^2)^{\frac{d}{2}}C^{-1}e^{-Ad}  = \frac{1}{Q^2},$$
which implies \eqref{ineq:ratio}.
\ignore{
In summary, we obtain the RHF bound
\begin{align*}
H(a_1) \leq (\delta^\frac{2}{d}-\mu^2)^{-\frac{d}{4}}(CB)^\frac{1}{2}e^{\frac{Ad}{2}}H(\det g)^\frac{1}{2}.
\end{align*}
Recall that $H(a_1)=H(v_1)$. }
\qed
\end{proof}
It is likely that, in practice, we will see a much better RHF than the bound given by Proposition \ref{prop:rhf_bound}, since we lost quite a bit in the theoretical analysis above by bounding $\delta$ by \eqref{eq:lossy}. It is also to be reminded that $|\Delta_F|^\frac{1}{2}H(v_1)$ equals the determinant of a dense rank $1$ module lattice of $L(g)$. Thus for comparisons with algorithms that finds a short vector, one should take the $d$-th root of the above bound. 
%{\color{purple} This raises a few questions regarding how different the Euclidean norms and the heights can actually be. An interesting application would be the following: I give you a $2d\times 2d$ NTRU basis with integer coordinates, what is the condition for adelic-LLL to recover the dense sublattice? Do we always have $\det L(g) = |\Delta_F|^{1/2}H(\det g)$?} {\color{orange} yes these are good topics for further research. Also the correct formula is $\det L(g) = |\Delta_F|H(\det g)$. Basically, $|\Delta_F|^{n/2}$ is the determinant of the canonical embedding $\rho:F^n\rightarrow (F\otimes\Q)^n$ (over $\Q$-vector spaces), and each $H_\sigma(\det g)$ gives the change of volume at each place $\sigma$. }

\subsection{Details on each routine}

In this section, we fill in all the black boxes from Algorithm \ref{alg:lll_c}, assuming that the input $g$ is given as a pseudobasis.

\subsubsection{Setup}\label{subsec:setup}
We assume the knowledge of $\Z$-bases $\mathcal B = \{w_1,\ldots,w_d\}$ of $\cO_F$, and $\mathcal U$ of any finite index sublattice of the unit lattice $\Log(\cO_F^*)$. The quality of these bases will affect the efficiency of our algorithm, which may be nontrivial to ensure for certain fields. However, in the practically most relevant case of $F=\Q(\zeta_{2^\kappa})$, a power-of-$2$ cyclotomic field, there exist well-known excellent choices for both $\mathcal B$ and $\mathcal U$.
%Let $\mathcal F \subseteq F \otimes \R$ be the fundamental parallelepiped of $\rho(\cO_F)$ corresponding to this basis that is centered at the origin, and choose any $\ell>0$ satisfying $\ell^{-1}\max_{\sigma\mid\infty,\vx\in\mathcal F}|\sigma(\vx)| \leq \mu$. For example, if $F=\Q(\zeta_{2^\kappa})$ is a power-of-$2$ cyclotomic fields, one can use the well-known choice for $\mathcal B$ and deduce the corresponding $\ell$ from it; for $\mathcal U$, one can take the cyclotomic units (\emph{cf.} \cite{CDPR16}). 

We will also need an SVP-approximation solver for lattices over $\Z$ as a subroutine. To this end, one is free to choose any algorithm between LLL and an exact SVP solver, or even make an adaptive choice each time. We write $\gamma_k$ for the root Hermite factor for the algorithm we use on rank $k$ lattices, that is, $\gamma_k$ is a quantity satisfying
\begin{align*}
\|\vv\|\leq\gamma_k^k(\det \Lambda)^\frac{1}{k}
\end{align*}
where $\Lambda$ is any rank $k$ lattice and $\vv\in\Lambda$ is any short vector found by that algorithm.

We follow \cite{LPSW19} in our choice of the input module: it is given as $\gb_1 \vv_1 + \gb_2\vv_2 \subseteq \rho(\cO_F^2)$, where $\gb_1,\gb_2$ are fractional ideals of $F$ and $\vv_1,\vv_2 \in \rho(\cO_F^2)$ are linearly independent over $F$. 
%Following Cohen \cite{Coh96} we assume that ideals are given as an (HNF) matrix with respect to $\mathcal B$.

As explained previously, our module corresponds to $g \in \GL(2,\A_F)$ where
\begin{align*}
g_\infty = \begin{pmatrix} \vv_1 \\ \vv_2 \end{pmatrix}, g_\nu=\begin{pmatrix} \pi_\nu^{-\ord_\nu(\mathfrak b_1)} & \\ & \pi_\nu^{-\ord_\nu(\mathfrak b_2)} \end{pmatrix}
\end{align*}
for each $\nu \nmid \infty$. We introduce a more classical way to express $g$ that will be easier to implement in code. We rewrite $M_\infty = g_\infty$, and for each embedding $\sigma:F \rightarrow \C$, we let
\begin{align*}
M_\sigma = \begin{pmatrix} \sigma\circ\rho^{-1}(\vv_1) \\ \sigma\circ\rho^{-1}(\vv_2) \end{pmatrix} = \begin{pmatrix} 1 & \\ m_\sigma & 1 \end{pmatrix}\begin{pmatrix} a_{1\sigma} & \\ & a_{2\sigma}\end{pmatrix}U_\sigma
\end{align*}
where $a_{1\sigma},a_{2\sigma}>0$ and $U_\sigma \in K_\sigma$, and
\begin{align*}
M_I = \begin{pmatrix} \gb_1^{-1} & \\ & \gb_2^{-1} \end{pmatrix},
\end{align*}
an ideal-valued matrix. Clearly $M_\infty$ can be identified with $(M_\sigma)_{\sigma\mid\infty}$. It may also be helpful to construe $M_I$ as $\prod_{\nu\nmid\infty}g_\nu$ in the arguments to follow.

Additionally, we define an auxiliary variable for our algorithm
\begin{align*}
M_F = \begin{pmatrix} 1 & 0 \\ m_F & 1 \end{pmatrix} \in \GL(2,F),
\end{align*}
where $m_F=0$ initially. We will write $m$ for the adele such that $m_\nu = (m_F)_\nu$ for finite $\nu$, and, for $\sigma \mid \infty$, whose $\sigma$-component coincide with $m_\sigma$ above. Similarly, for $i=1,2$, $a_i$ will be the adele for which $(a_i)_\nu = \pi_\nu^{-\ord_\nu\mathfrak b_i}$ for finite $\nu$, and $(a_i)_\sigma = a_{i\sigma}$ for infinite $\sigma$. Hence $g$ has an Iwasawa decomposition
\begin{align*}
g = \begin{pmatrix} 1 & \\ m & 1\end{pmatrix}
\begin{pmatrix} a_1 & \\ & a_2 \end{pmatrix}
\prod_{\sigma\mid\infty}U_\infty,
\end{align*}
and $M_F,M_I,M_\infty$ simply encode the same information in the classical language.

\subsubsection{Size reduction}

%\vspace{-\baselineskip}
\begin{algorithm} 
\caption{\texttt{size-reduce}} \label{alg:sizereduce}
\begin{algorithmic}[1]
\Require Parameters $\mu,C$, a subfield $E\subseteq F$ of degree $d_E$, matrices $\mathbf M$ and $\mathbf M_E$ representing bases of $\rho(\cO_F)$ and $\rho(\cO_E)$, an HSVP solver in dimension $d+d_E$ 
%with RHF factor $\gamma_{d+d_E}$ satisfying $$C \geq \mu^{-\frac{d^2}{d_E}} \gamma_{d+d_E}^{d(d+d_E)(1+\frac{d}{d_E})}\Delta^{\frac{d}{d_E}},$$ (here $\Delta =\sqrt{(d/d_E)^{d_E}\Delta_E\Delta_F}$)
and input $(m_\sigma)_{\sigma:F\rightarrow\C}$
%\Ensure 
%Returns $r \in F$ such that $(-r,m-r)\in\A_\f\times\A_\infty$ is size-reduced
%Returns $p,q \in \cO_F$ such that $|q_\sigma m_\sigma-p_\sigma|\leq\mu$ and $|q_\sigma|\leq C^{1/d}$

%\State Compute the GSO coefficients of $M_\infty$ to find $\rho(m_\infty)$
%\State Compute reals $m_i$ such that $\vm_\infty = \sum m_i\rho(w_i)$
\State %$\omega \leftarrow \mu^\frac{d+d_E}{d_E}\gamma_{d+d_E}^{-\frac{(d+d_E)^2}{d_E}}\Delta^{-\frac{1}{d_E}}$ %\henry{This should depend on $E$ right?} fixed -sk
$\omega \leftarrow \mu C^{-1/d}$
\State $\displaystyle T \leftarrow \begin{pmatrix} \omega \mathbf M_E & \mathbf M_E\diag(m_\sigma) \\  & -\mathbf M \end{pmatrix}$
\State Use the HSVP solver on $T$ to obtain $\vx=(q_1,\ldots,q_{d_E},p_1,\ldots,p_d)$ such that $\vx T$ is short

%a nonzero vector of the form $(\omega \rho(q),\rho(qm_\infty-p))$, $q\in\cO_E$ and $p \in \cO_F$, of length $\leq C$
%$\displaystyle \leq \gamma_{d+d_E}^{d+d_E}\omega^{\frac{d_E}{d+d_E}}\Delta^{\frac{1}{d+d_E}}$
%\henry{Shouldn't this quantity somehow depend on $C$ ? otherwise there is no point in passing $C$ as a parameter. What I had in mind is: fix input $\mu,C$. Then the algorithm picks the best choice of $\omega$ depending in $\gamma$ and $C$. I find the pseudo-code very unhelpful in its current state (realised this was an issue while implementing).} {\color{orange} I guess I'm really using $C$ as a shorthand for the complicated expression above, and taking it to equal to that expression... how do you suggest to fix it?}

\State \textbf{return} $\displaystyle p\leftarrow (p_1,\ldots,p_d)\mathbf M,q \leftarrow (q_1,\ldots,q_{d_E})\mathbf M_E$ 
%\henry{one might argue that recovering $p,q$ from their embeddings is not trivial because of precision issues. If needed I can try to discuss this in the appendix. Otherwise just working with Gram Matrices should help here, maybe again not relevant to the core algorithm so I'm happy writing about this in the appendix} %// In rank 2 adelic LLL, we will have $\displaystyle m_F \leftarrow -p/q, M_\infty \leftarrow \begin{pmatrix} 1 & \\ -\rho(p/q) & 1 \end{pmatrix}M_\infty$ {\color{orange} hmm I guess in reality we'll be returning $p,q$ separately}
\end{algorithmic}
\end{algorithm}

We will compute $p,q\in\cO_F$ such that
\begin{align}\label{eq:sizereduce_new}
|q_\sigma m_\sigma -p_\sigma| \leq \mu \mbox{ and } |q_\sigma| \leq C^\frac{1}{d}
\end{align}
for every embedding $\sigma:F\hookrightarrow\C$. Let us first explain how this leads to size reduction. Write $r=p/q \in F$. We assume that the algorithm \texttt{size-reduce} has input $m=((m_F)_\f,m_\infty) \in \A_F$ such that $m_F=0$. It will ultimately replace $m$ by $m-r$. Note that $H_\f((m-r,1)) \leq H_\f(q^{-1}) = H_\infty(q)$, the former since $m_\f = (m_F)_\f=0$, and the latter by the product formula. Then, with $\tau_\sigma=|q_\sigma|$, $H_\infty(q) =\prod_{\sigma:F\hookrightarrow\C}\tau_\sigma$ holds by the definition of $H_\infty$, and moreover, \eqref{eq:sizereduce_new} translates to
\begin{align*}
\tau_\sigma|(m-r)_\sigma| \leq \mu \mbox{ and } \tau_\sigma \leq C^\frac{1}{d},
\end{align*}
satisfying all the conditions for size reducedness indeed.

It remains to show how to obtain \eqref{eq:sizereduce_new}. Let
\begin{align*}
\mathbf M = \begin{pmatrix} - & \rho(w_1) & - \\ & \vdots & \\ - & \rho(w_d) & - \end{pmatrix}.
\end{align*}
Similarly, for any choice of a subextension $\Q \subseteq E \subseteq F$ of degree $d_E$, let $\mathbf M_E$ be the matrix whose rows consist of a $\Z$-basis of $\rho(\cO_E)$. Let $\omega = \mu C^{-1/d}$, and consider the rank $d+d_E$ lattice spanned by the rows of
\begin{align}\label{eq:thematrix}
\begin{pmatrix}
\omega \mathbf M_E & \mathbf M_E\diag(m_\sigma) \\  & -\mathbf M
\end{pmatrix}.
\end{align}
Observe that this lattice 
%\henry{I find that the notation $\diag(\rho(m_\infty))$ is very confusing. Do you mean the coefficients of $\mathbf m_\infty$ in the embedded basis of $\cO_F$ ?}
has determinant $\omega^{d_E}\Delta$, where $\Delta = \sqrt{(d/d_E)^{d_E}|\Delta_E\Delta_F|}$,\footnote{By a straightforward argument, one can show that the sublattice $\rho(\cO_E)$ of $\rho(\cO_F)$, where $\cO_E$ is seen as a subset of $\cO_F$, has determinant $\sqrt{(d/d_E)^{d_E}|\Delta_E|}$.} and its vectors are of the form
$(\omega \rho(q), \rho(qm_\infty-p))$ for $q \in \cO_E, p\in\cO_F$. Running our choice of an HSVP solver, we obtain a nonzero vector of such form of length at most $$\lambda = \gamma_{d+d_E}^{d+d_E}\omega^{\frac{d_E}{d+d_E}}\Delta^{\frac{1}{d+d_E}}.$$
Assume $C^{1/d}\ge \lambda\omega^{-1}$, then equivalently by definition of $\omega$, we get $\mu\ge \lambda $ and thus
$|q_\sigma m_\sigma-p_\sigma| \leq \mu$ for every $\sigma$. We also have $\omega|q_\sigma| \leq \mu$, or equivalently $|q_\sigma| \leq C^{1/d}.$
%Accordingly, we can take 
%$$C = \mu^{-\frac{d^2}{d_E}} \gamma_{d+d_E}^{d(d+d_E)(1+\frac{d}{d_E})}\Delta^{\frac{d}{d+d_E}(1+\frac{d}{d_E})}.$$
Note that $q=0$ is impossible, since it would imply $\rho(p)$ has length at most $\mu < 1$, whereas it is well-known that a nonzero vector of $\rho(\cO_F)$ has length at least $\sqrt{d}$. This proves that \eqref{eq:sizereduce_new} can indeed be attained, as long as $C^{1/d}$ is chosen to be greater than or equal to 
\begin{align}\label{eq:C^1/d}
\mu^{-\frac{d}{d_E}} \gamma_{d+d_E}^{(d+d_E)^2/d_E}\Delta^{\frac{1}{d_E}}.
\end{align}

The above discussion is summarized in pseudocode form in Algorithm \ref{alg:sizereduce}. As an example, we can take $E=F$ and $C= \mu^{-d}\gamma_{2d}^{4d^2}|\Delta_F|$; or if we take $E=\Q$, then $C= \mu^{-d^2}\gamma_{d+1}^{d(d+1)^2}(d|\Delta_F|)^{d/2}$ is a suitable choice. %Since the power of $\gamma$ tends to make up the largest contribution, it is helpful to think $C \approx \gamma^{O(d^3/d_E)}$.

%{\color{orange} so in this correction, we lost by a factor of 2, which I'm afraid may wipe out the advantage... but it's better than being wrong. One may consider an alternative strategy where one takes $q$ from a subfield of $F$.}

% muting out old version
\ignore{
Start by computing $m_i \in\R$ such that 
\begin{align*}
\vm_\infty := (m_\sigma)_{\sigma\mid\infty} =  m_1\rho(w_1)+\ldots+m_d\rho(w_d).
\end{align*}
We can use the classical LLL (cf. \cite[(1.39)]{LLL82}) --- or a stronger reduction if desired --- to find integers $p_1,\ldots,p_d,q$ such that
\begin{align*}
|p_i-qm_i| \leq \varepsilon, 1 \leq q \leq (4/3)^{d(d+1)/4}\varepsilon^{-d}
\end{align*}
with $\varepsilon = (2\ell)^{-1}$. Taking
\begin{align*}
r = \frac{1}{q}(p_1w_1+\ldots+p_dw_d) \in q^{-1}\cO_F
\end{align*}
and setting $C = (4/3)^{d(d+1)/4}(2\ell)^{d}$, this implies {\color{purple} fixing a value for C here is a bit awkward. Maybe we should have a proposition that says "if $C = (4/3)^{d(d+1)/4}(2\ell)^{d}$, then Alg. 2 with LLL as step 3 is correct", so as to keep some flexibility on the CVP oracle}
{\color{orange} That's exactly what I meant to say... please feel free to edit to your liking without color}
\begin{align}
\label{eq:msize}
H_\f((m-r,1))^\frac{1}{d}\max_{\sigma\mid\infty}|\sigma(m-r)| &\leq q\cdot\max_{\sigma\mid\infty}\left|\sigma\left(\sum_j(m_j-\frac{p_j}{q})w_j\right)\right| \\
\notag
&\leq \ell^{-1}\max_{\sigma\mid\infty,\vx\in\cal F}|\sigma(\vx)| \leq \mu
\end{align}
{\color{purple} Here we used $q^{1/d}\le q$ right? Changed $w_i$ to $w_j$} {\color{orange} Actually, since the denominator of $m-r$ is integer $\leq q$, we have $H_\f((m-r,1)) \leq q^d$. So it's not as wasteful as what you said.} and \henry{In the next equation, where did the power of $d$ go?}
\begin{align*}
H_\f((m-r,1)) \leq q \leq C
\end{align*}
as desired. 
Update $m_F \leftarrow -r$ and $\vm_\infty \leftarrow \vm_\infty - \rho(r)$, or equivalently {\color{purple} here $r$ is both the field element and the adele, because of the definition of $\rho$, I find $\rho(r_\infty)$ to be a bit confusing. In fact this might turn into a more general comment: algorithmically, it seems much safer (because of precision issues with representing complex numbers) to track the information of $M_\infty$ via $M_F$, which can be represented symbolically using only rational numbers (although they might be very large if we don't use the scale() algorithm). In the end I don't think we need to speak about adeles in the algorithm?} {\color{orange} good point, fixed. Regarding $\rho(r)$, as you probably noticed, I'm trying to be conscious of the distinction between $F$ and $\rho(F)$, or elements that are not assigned the coordinates via canonical embedding vs. elements that are.}{\color{purple} I think thats really useful}$$M_\infty \leftarrow \begin{pmatrix} 1 & \\ -\rho(r_\infty) & 1 \end{pmatrix}M_\infty.$$

%{\color{orange} Bitsize analysis: $r$ takes at most $O(\log q \cdot \log\max_im_i)$ bits to express. $m_\sigma$ is bounded by $\mu q$ thanks to \eqref{eq:msize}.}

Let us discuss the choice of $C$ in which we replace LLL by an SVP oracle for ($\Z$-)rank $d+1$ lattices. Then following the argument of \cite[(1.39)]{LLL82}, and by Minkowski's bound on the shortest vector of a lattice, we can replace the above constraint on $q$ by
\begin{align*}
1 \leq q \leq \sqrt{d+1}^{d+1}\varepsilon^{-d},
\end{align*}
and set $C = \sqrt{d+1}^{d+1}(2\ell)^d$.

{\color{orange} I feel like I can reduce this to ideal-SIVP.. anyway for now write wrt slide algorithm (or SDBKZ) since that's more meaningful}
}

\subsubsection{Lov\'asz test}

%\vspace{-\baselineskip}
\begin{algorithm}
\caption{\texttt{Lovasz-test} rank 2}
\begin{algorithmic}[1]
\Require Parameter $\delta$, and inputs $m_F,M_I,M_\infty$
\State \textbf{if}
$$
\delta N((m_F\cO_F+\gb_1\gb_2^{-1}))\leq \prod_{\sigma:F\rightarrow \C}|(m_\sigma,a_{2\sigma}/a_{1\sigma})|
$$
\textbf{return} true
\State \textbf{else} \textbf{return} false
\end{algorithmic}
\end{algorithm}
%\vspace{-\baselineskip}
If the Lov\'asz condition
\begin{align*}
\delta H((a_1,0)) \leq H((m a_1,a_2)),
\end{align*}
or equivalently as a product over all embeddings,
\begin{align*}
\delta N(\gb_1)\prod_{\sigma:F\hookrightarrow \C} |a_{1\sigma}| \leq N((m_F\gb_1^{-1}+\gb_2^{-1}))^{-1}\prod_{\sigma:F\hookrightarrow \C}|(m_\sigma a_{1\sigma},a_{2\sigma})|,
\end{align*}
or the further simplified
\begin{align*}
\delta N((m_F\cO_F+\gb_1\gb_2^{-1}))\leq \prod_{\sigma:F\hookrightarrow \C}|(m_\sigma,a_{2\sigma}/a_{1\sigma})|,
\end{align*}
%(here beware that the product is taken over all \emph{embeddings} $\sigma$, not \emph{places})
is fulfilled, 
then we terminate the algorithm. If not, we carry out the adelic swap step below. %\henry{I guess now this should be class reduction instead of size reduction as it comes before in order}{\color{orange} fixed the wording a bit}

\subsubsection{Adelic swap}
%\vspace{-\baselineskip}
\begin{algorithm}
\caption{\texttt{adelic-swap} rank 2} 
\label{alg:adelicswap}
\begin{algorithmic}[1]
\Require $m_F = x/y$ ($x,y \in \cO_F$),$M_I,M_\infty$ that failed \texttt{Lovasz-test}
%\Ensure $M_F,M_I,(M_\sigma)_{\sigma\mid\infty}$ modified for next \texttt{size-reduce}
%\State Compute $x,y\in\cO_F$ such that $m_F = x/y$
\State Compute $c \in \cO_F$ such that $cx\gb_1^{-1},cy\gb_2^{-1} \subseteq \cO_F$
\State $\gd \leftarrow cx\gb_1^{-1}+cy\gb_2^{-1}$
\State Compute $u\in(\gb_1\gd)^{-1},v\in(\gb_2\gd)^{-1}$ such that $cxu+cyv=1$ %{\color{purple} technically we don't need to know $v$ here I guess} {\color{orange} yes but it comes in pair with $u$ so..}
\State $\displaystyle m_F \leftarrow 0, M_I \leftarrow \begin{pmatrix} \gd & \\ & (\gb_1\gb_2\gd)^{-1}\end{pmatrix}, M_\infty \leftarrow \begin{pmatrix} & \rho(cy) \\ \rho(1/(cy)) & \rho(-u) \end{pmatrix}M_\infty $
\end{algorithmic}
\end{algorithm}
%\vspace{-\baselineskip}
If \texttt{Lovasz-test} returns false, we swap the two rows of $g$, as in the original LLL algorithm. This maneuver is equivalent to swapping the rows of $M_\infty$, and also swapping the rows of $M_F$, so that
\begin{align*}
M_F = \begin{pmatrix} m_F & 1 \\ 1 & \end{pmatrix}, M_I = \begin{pmatrix} \mathfrak b_1^{-1} & \\ & \mathfrak b_2^{-1} \end{pmatrix}.
\end{align*}
$M_I$ here is correct to remain unchanged, since, from the adelic point of view, on each finite place $\nu$,
\begin{align}\label{eq:swap_adelic}
\begin{pmatrix} & 1 \\ 1 & \end{pmatrix}\begin{pmatrix}1 & \\ m_\nu & 1\end{pmatrix}\begin{pmatrix}a_{1\nu} & \\ & a_{2\nu} \end{pmatrix}= \begin{pmatrix}m_\nu & 1 \\ 1 & \end{pmatrix}\begin{pmatrix}a_{1\nu} & \\ & a_{2\nu} \end{pmatrix}.
\end{align}
We would like to carry out some manipulations, so as to be able to rewrite $M_F$ as the identity element, so that our data $M_F,M_I,M_\infty$ comes back to representing a pseudobasis for the next iteration of \texttt{size-reduce}.
\ignore{
\phong{What I meant to say, is that in theoretical descriptions of LLL, we could just recompute GSO at every step:
updates of GSO only occur when we want to optimize the implementation. 
Here, similarly, after you swap the two rows, you could just "recompute"
the Iwasawa decomposition, but we're doing this gymnastics
because we care about efficiency: is that correct?}{\color{orange} yes, and we do need to care, since in principle, one would compute the decomposition place by place, which involves factorization}}
%\phong{Is it correct to say that this is the equivalent of the Gram-Schmidt update in integer LLL after swap?} {\color{orange} Over the infinite places yes, but there is an extra step to make $m_F=0$. I don't know if the manipulation on the right by $\GL_2(\cO_\nu)$ has a classical analogue --- you're the best person to answer this}

Start by writing $m_F = x/y$ for $x,y\in \cO_F$; \ignore{{\color{purple} you mean x,y no?}{\color{orange}thanks }}note that this is already accomplished in the \texttt{size-reduce} step with $y = q$. Compute nonzero $c \in \cO_F$ such that $cx\gb_1^{-1}, cy\gb_2^{-1}$ are both integral. Suppose $cx\gb_1^{-1} + cy\gb_2^{-1} = \gd$. Then by Theorem 1.2 of [Coh96], we can efficiently compute $u \in (\gb_1\gd)^{-1}, v \in (\gb_2\gd)^{-1}$ such that
%\begin{align*}
$cx u + cy v = 1.$
%\end{align*}

Fix $\nu\nmid\infty$. Let us write $d_\nu = \pi_\nu^{-\ord_\nu\mathfrak d}$. Now consider the $\nu$-adic matrix
\begin{align*}
\begin{pmatrix}ua_1^{-1} & cy a_1^{-1} \\ va_2^{-1} & -cx a_2^{-1} \end{pmatrix} \begin{pmatrix}d^{-1} & \\ & a_1a_2d\end{pmatrix} = \begin{pmatrix} ua_1^{-1}d^{-1} & cy a_2d \\ va_2^{-1}d^{-1} & -cx a_1d \end{pmatrix}
\end{align*}
where we suppress all the $\nu$ subscripts, and temporarily tolerate the notational overlap of $d$ with the field degree. By construction, this is a matrix of determinant -1 consisting of entries in $\cO_\nu$, hence an element of $\GL_2(\cO_\nu)$. Therefore we can multiply this element to the right of $g$ without changing $L(g)$. Over each $\cO_\nu$, we would multiply it to the right of \eqref{eq:swap_adelic}, and obtain
\begin{align*}
\begin{pmatrix}m & 1 \\ 1 & \end{pmatrix}\begin{pmatrix}a_1 & \\ & a_2 \end{pmatrix}\begin{pmatrix} ua_1^{-1}d^{-1} & cy a_2d \\ va_2^{-1}d^{-1} & -cx a_1d \end{pmatrix}  =\begin{pmatrix}x/y & 1 \\ 1 & \end{pmatrix}\begin{pmatrix} ud^{-1} & cy a_1a_2d \\ vd^{-1} & -cx a_1a_2d
\end{pmatrix} \\
= \begin{pmatrix}x/y & 1 \\ 1 & \end{pmatrix}\begin{pmatrix}
u & cy \\ v & -cx
\end{pmatrix}
\begin{pmatrix} d^{-1} & \\ & a_1a_2d
\end{pmatrix} 
= \begin{pmatrix} 1/(cy) & \\ u & cy\end{pmatrix}\begin{pmatrix} d^{-1} & \\ & a_1a_2d
\end{pmatrix}.
\end{align*}
In the classical language, this corresponds to setting
\begin{align*}
M_F &= \begin{pmatrix} x/y & 1 \\ 1 & \end{pmatrix}\begin{pmatrix} u & cy \\ v & -cx\end{pmatrix} = \begin{pmatrix} 1/(cy) & \\ u & cy \end{pmatrix}, \\ 
M_I &= \begin{pmatrix} \gd & \\ & (\gb_1\gb_2\gd)^{-1} \end{pmatrix}.
\end{align*}
Next, we multiply $g$ from the left by
\begin{align*}
\begin{pmatrix} cy & \\ -u & 1/(cy) \end{pmatrix} \in \GL(2,F).
\end{align*}
In the classical language, this turns $M_F$ back to the identity, and multiplies $M_\infty$ by 
\begin{align*}
\begin{pmatrix} \rho(cy) & \\ \rho(-u) & \rho(1/(cy)) \end{pmatrix}.
\end{align*}
In summary, we have made the following changes to the variables:
\begin{align*}
m_F \leftarrow 0, \mathfrak b_1 \leftarrow \mathfrak d^{-1}, \mathfrak b_2 \leftarrow \mathfrak b_1\mathfrak b_2 \mathfrak d, M_\infty \leftarrow \begin{pmatrix} & \rho(cy) \\ \rho(1/(cy)) & \rho(-u) \end{pmatrix}M_\infty.
\end{align*}

\ignore{\color{orange} Bitsize analysis: (this is talking to myself) need to think... $\frak a_1,\frak a_2$ can be made bounded by class reduction if necessary. Also if one knows the height of a vector, its size can be adjusted by unit reduction.

Also, $y, c, \mathfrak d$ is bounded absolutely (independent of input size) assuming $\ga_1,\ga_2$ are class reduced. Also with some bizarre steps, $x$ thus $u$ can be made to be bounded absolutely. So in the very worst case the entries of $M_\sigma$ grow at most by a bounded amount --- but then in the size reduction step they get hammered again so...

(and unit reducing $cy$ and $u$ may be helpful?)

Or simply try to imitate LPSW where a similar problem is encountered and solved
}

\subsubsection*{Unit reduction}

%\vspace{-\baselineskip}
\begin{algorithm}
\caption{\texttt{unit-reduce} rank 2} 
\label{alg:ureduce}
\begin{algorithmic}[1]
\Require Parameter $A$, input $M_\infty$, basis $\mathcal U$ of a sublattice of $\Log(\cO_F^*)$
\Ensure $M_\infty$ is unit reduced with respect to $A$
\State Compute $\alpha_\infty$ of $M_\infty$
\State Find an aCVP solution $\Log(u)$ with respect to $\Log(\cO_F^*)$ and target $\Log(\alpha_\infty)-\frac{1}{d}\log(H_\infty(\alpha_\infty))\cdot(1,\ldots,1,2,\ldots,2)$, %\henry{Should the target have a $-\frac{1}{d}\log(H_\infty(\alpha_\infty))\cdot(1,\ldots,1,2,\ldots,2)$ as in the condition?} {\color{orange} yes thank you}
by applying Babai's nearest plane algorithm using $\mathcal U$
\State $\displaystyle M_\infty \leftarrow \begin{pmatrix} 1 & \\ & u^{-1}\end{pmatrix}M_\infty$
\end{algorithmic}
\end{algorithm}
%\vspace{-\baselineskip}

To unit reduce $\alpha_\infty \in \A_\infty$ such that $\alpha_\sigma \neq 0$ for all $\sigma\mid\infty$, one simply solves aCVP with respect to the target $\Log(\alpha_\infty)-\frac{1}{d}\log(H_\infty(\alpha_\infty))\cdot(1,\ldots,1,2,\ldots,2)$, the lattice spanned by $\mathcal U$, and the approximate factor $A$, using Babai's nearest plane algorithm for example. As is well-understood, the ``better'' the quality of $\mathcal U$, the lower we can set $A$ to be. Let $u \in \cO_F^*$ be any preimage of this aCVP. Then it is easy to see that $u^{-1}\alpha_\infty$ is unit reduced. In particular, for any embedding $\sigma:F\hookrightarrow\C$,
\begin{align*}
e^{-A} \leq \frac{|u^{-1}_\sigma \alpha_\sigma|}{H_\infty(x)^\frac{1}{d}} \leq e^A.
\end{align*}

\subsubsection{Scaling for bitsize control and class reduction}

%\vspace{-\baselineskip}
\begin{algorithm}
\caption{\texttt{scale} (Algorithm 3.2 of \cite{LPSW19})} 
\label{alg:scale}
\begin{algorithmic}[1]
\Require A fractional ideal $\gb \subseteq F$, a vector $\vv\in (F\otimes\R)^n$, the ``Gram-Schmidt coefficient'' $\va\in F\otimes\R$, and an HSVP solver in dimension $d$ with RHF factor $\gamma_d$
\Ensure $\gb\vv$ is scaled
\State Apply the HSVP solver to find $\mathbf s \in \va\rho(\gb)$ such that $$\|\mathbf s\| \leq \gamma_d^|\Delta_F|^{1/(2d)}\left(N(\gb)\prod_\sigma \va_\sigma\right)^{1/d}$$
\State Write $\mathbf s = \va\vx$ with $\vx \in \rho(\gb)$
\State $\gb \leftarrow \rho^{-1}(\vx^{-1})\gb, \vv \leftarrow \vx\vv$
\end{algorithmic}
\end{algorithm}
%\vspace{-\baselineskip}
\vspace{-\baselineskip}
\begin{algorithm}
\caption{\texttt{scale+} rank 2} 
\label{alg:scale+}
\begin{algorithmic}[1]
\Require Parameter $B$, inputs $M_I,M_\infty$, and an HSVP solver in dimension $d$ with RHF factor $\gamma_d$ satisfying $B \geq \gamma_d^{d^2}|\Delta_F|^{1/2}$
\Ensure $\alpha$ is class reduced, and bitsizes of the inputs are at most $\mathrm{poly}(\text{input},d,\log|\Delta_F|)$
\State \texttt{scale}$(\gb_1, \vv_1,a_{1\infty})$; update $M_I,M_\infty$ accordingly
\State \texttt{scale}$(\gb_2\gb_1^{-1}, \vv_2,a_{2\infty})$; update $M_I,M_\infty$ accordingly // $\gb_2$ changes by $\rho^{-1}(\vx^{-1})\gb_2$, $\gb_1$ remains unchanged
\State Apply \cite[Algorithm 3.3]{LPSW19} to $\gb_1\vv_1+\gb_2\vv_2$, and update $M_I,M_\infty$ accordingly
\end{algorithmic}
\end{algorithm}
%\vspace{-\baselineskip}

\cite[Section 3.2]{LPSW19} introduces mechanisms for bitsize control called \emph{scaling} and \emph{size reducing} (in the sense of \cite{LPSW19}, different from ours). For a pseudo-basis $\gb_1\vv_1+\gb_2\vv_2$, these algorithms together return another pseudo-basis of the same lattice $\gb'_1\vv'_1+\gb'_2\vv'_2$ such that 
\begin{itemize}[label=\textbullet]
\item $\displaystyle \cO_F \subseteq \gb'_i,\ 2^{-d^2}|\Delta_F|^{-\frac{1}{2}} \leq N(\gb'_i) \leq 1 $
\item $\|\rho(a_{i,\infty})\|\leq 2^d|\Delta_F|^{\frac{1}{2d}}H(a_i)^{\frac{1}{d}}$
\item $\displaystyle \vv'_i\in\cO_F,\ \|\vv'_i\| \leq \sqrt{d}(8d)^d|\Delta_F|\max_iH(a_i)^{\frac{1}{d}}$
\end{itemize}
in $\mathrm{poly}(\text{input},d,\log|\Delta_F|)$; 
see \cite[Algorithms 3.2 and 3.3]{LPSW19},\footnote{Algorithm 3.3 of \cite{LPSW19} contains an error that would alter the module under question. In Step 3 there, instead of rounding by the ring of integers, one must round by $I_iI_j^{-1}$. We discuss this detail in Appendix~\ref{app:correct} below. From the (rank 2) adelic perspective, Algorithm 3.3 of \cite{LPSW19} can be seen as multiplying $g$ on the left by a unipotent matrix whose lower left entry $m_F$ satisfies $(m_Fa_1/a_2)_\nu \in \cO_\nu$ for every $\nu \nmid \infty$. 
} also the proof of \cite[Theorem 3.8]{LPSW19}. It may help to note here that, unravelling all notations, $\rho(a_{i,\infty})$ and $H(a_i)$ above coincides with $r_{ii}$ and $\mathcal N(r_{ii}I_i)$ in \cite{LPSW19}, respectively.

These algorithms of \cite{LPSW19} use LLL algorithm in dimension $d$ as a subroutine. Using an HSVP solver of root Hermite factor $\gamma_d$, one obtains
\begin{itemize}[label=\textbullet]
\item $\displaystyle \cO_F \subseteq \gb'_i,\ \gamma_d^{-d^2}|\Delta_F|^{-\frac{1}{2}} \leq N(\gb'_i) \leq 1 $
\item $\|\rho(a_{i,\infty})\|\leq \gamma_d^d|\Delta_F|^{\frac{1}{2d}}H(a_i)^{\frac{1}{d}}$
\end{itemize}
for the first two conditions instead. For reader's convenience, we reproduced the scale algorithm of \cite{LPSW19} with this modification in Algorithm \ref{alg:scale}. 

In this section, we explain how to apply these algorithms to control the bitsizes of the variables and ensure class reducedness at the same time. This is done by the \texttt{scale+} algorithm, whose pseudocode is provided in Algorithm \ref{alg:scale+}. Observe that in step 2, \texttt{scale+} scales $\gb_2\gb_1^{-1}$ instead of $\gb_2$. Thus $\alpha$ is integral since $\gb_2^{-1}\gb_1$ is, and it satisfies $\gamma_d^{-d^2}|\Delta_F|^{-1/2} \leq N(\gb_2\gb_1^{-1}) = H_\f(\alpha)$. Hence class reduction is achieved with $B=\gamma_d^{d^2}|\Delta_F|^{1/2}$.

In addition, from the fact that $\gb_1^{-1}$ is integral, it follows that $\cO_F \subseteq \gb_2$, and $\ 2^{-2d^2}|\Delta_F|^{-1} \leq N(\gb_2) \leq 1$, and $\vv_2 \in \rho(\cO_F)$ hold at the end of \texttt{scale+}. We also have $\|\rho(a_{2,\infty})\| \leq \gamma_{d}^{2d}|\Delta_F|^{1/d}H(a_2)^{1/d}$. Hence, the second row of $g$ stays polynomially controlled just like the first, although a bit larger than the original scale algorithm of~\cite{LPSW19}.

As a side benefit, it happens that \texttt{scale+} allows to skip line 1 of Algorithm \ref{alg:adelicswap}, since both $\gb_i^{-1}$ become integral as a result.

\subsection{Summary}

\begin{algorithm}
\caption{Adelic LLL algorithm rank $2$ (pseudobasis version)} \label{alg:lll_p}
\begin{algorithmic}[1]
\Require Parameters $\delta,\mu,A,B,C$, basis $\mathcal B$ of $\cO_F$, choices of HSVP solvers for the subroutines, and a pseudobasis of the input module $\gb_1\vv_1+\gb_2\vv_2 \subseteq \rho(\cO_F^2)$
\Ensure Data $m_F,M_I,M_\infty$ representing a reduced element of $\GL(2,\A_F)$ for the input module
\State $\displaystyle m_F \leftarrow 0, M_I \leftarrow 
\begin{pmatrix} \gb_1^{-1} & \\ & \gb_2^{-1}\end{pmatrix}, M_\infty \leftarrow \begin{pmatrix} \vv_1 \\ \vv_2\end{pmatrix}$
\State \textbf{while} true \textbf{do}
\State \ \ \texttt{scale+}$(M_I,M_\infty)$
\State \ \ \texttt{unit-reduce}$(M_\infty)$ 
\State \ \ $(p,q) \leftarrow$\texttt{size-reduce}$(\rho(m_\infty))$; $\displaystyle m_F \leftarrow -p/q, M_\infty \leftarrow \begin{pmatrix} 1 & \\ -\rho(p/q) & 1 \end{pmatrix}M_\infty$ %\henry{should we add a step that extracts $m_\infty$ for self-completeness of the algorithm?} {\color{orange} yes perhaps}
\State \ \ \textbf{if}(\texttt{Lovasz-test}$(m_F,M_I,M_\infty)$) 
\State \ \ \ \ \textbf{break loop}
\State \ \ \textbf{else} 
\State \ \ \ \ \texttt{adelic-swap}$(m_F,M_I,M_\infty)$
\State \textbf{return} $m_F,M_I,M_\infty$
\end{algorithmic}
\end{algorithm}

Putting everything from the previous section together, we obtain Algorithm \ref{alg:lll_p}, the adelic LLL algorithm on pseudobases, as well as the following main result.

\begin{theorem}[Rank 2 reduction]\label{thm:main}
Algorithm \ref{alg:lll_p} returns the data $m_F,M_I,M_\infty$ representing a reduced element of $\GL(2,\A_F)$ corresponding to the input module in $\mathrm{poly}(\text{input},d,\log|\Delta_F|)$ and $\mathrm{poly}(\text{input})$ many calls to the chosen HSVP oracles.
\end{theorem}
\begin{proof}
Proposition \ref{prop:swap_bound} shows that the number of loops will be $\mathrm{poly}(\text{input})$. In each loop, \texttt{scale+} makes two calls to the dimension $d$ HSVP oracle, and \texttt{size-reduce} makes one call to the dimension $d+d_E$ HSVP oracle. Thus all that remains to show is that the bitsize of all variables are controlled throughout each loop.

At the start of each loop, \texttt{scale+} imposes an upper bound on the bitsizes of the inputs $M_I$ and $M_\infty$ by $\mathrm{poly}(\log H(a_i),d,\log|\Delta_F|)$. Hence it suffices to show 
\begin{enumerate}[label=(\roman*)]
\item \texttt{unit-reduce}, \texttt{size-reduce}, and \texttt{adelic-swap} change the bitsizes at most by $\mathrm{poly}(\text{input},d,\log|\Delta_F|)$
\item $\max_iH(a_i)$ stays bounded throughout the algorithm.
\end{enumerate}

If we use LLL for \texttt{size-reduce}, then (i) follows because everything runs in polynomial time. If we use a better algorithm than LLL, (i) still follows because its output will be smaller than what LLL gives. As for (ii), recall that $H(a_1)$ keeps decreasing throughout the algorithm, and that $H(a_2)$ is determined by $H(a_1)H(a_2) = H(\det g)$. Since $H(a_1) \geq 1$ (recall the input is a subset of $\rho(\cO_F^2)$ by assumption), $H(a_2) \leq H(\det g)$ always.
\qed
\end{proof}

The reason the output of Algorithm \ref{alg:lll_p} is not in pseudobasis form is that pseudobases are usually not size reduced. If desired, one can undo the last \texttt{size-reduce} step to output a pseudobasis such that $\gb_1\vv_1$ satisfies the height bound of Proposition \ref{prop:rhf_bound}. See the appendix for a comparison with the rank 2 reduction of \cite{LPSW19}.

\ignore{
\subsubsection*{Unit and class reduction}
Let $\alpha$ be an adele such that $\alpha_\nu = (a_2/a_1)_\nu$ for $\nu\nmid\infty$, and $\alpha_\sigma = (a_2/a_1)_\sigma$ for $\sigma\mid\infty$.
The goal is to compute $f=f_1f_2 \in F$ such that
(i) that $f_1\alpha$ is integral and $B \leq H_\f(f_1\alpha) \leq 1$ for some constant $B$ again depending only on $F$ 
(ii) $f_2 \in \cO_F^*$, and the entries of $$\Log(f\alpha) - \frac{1}{d}\log H_\infty(f\alpha)\cdot(\underbrace{1,\ldots,1}_\mathrm{real},\underbrace{2,\ldots,2}_\mathrm{complex})$$ are bounded by some constant $A$ depending only on $F$.
{\color{orange} Changmin writes how (i) can be ensured, in how much time complexity, and how $B$ is chosen.}
Ensuring (ii) is equivalent to solving approx-CVP with respect to the unit lattice of $F$ and the target point
$$\Log(f_1\alpha) - \frac{1}{d}\log H_\infty(f_1\alpha)\cdot(\underbrace{1,\ldots,1}_\mathrm{real},\underbrace{2,\ldots,2}_\mathrm{complex});$$
the solution would be $\Log(f_2)$. $A$ is then determined by the strength of the approx-CVP used. One could use Babai's nearest plane algorithm, which takes polynomial time in $d$; $A$ then would be half the $L^\infty$-diameter of the fundamental domain of $\mathcal U$.

Once this $f$ is found, we multiply 
\begin{align*}
\begin{pmatrix} f^{-1} & \\  & 1 \end{pmatrix}
\end{align*}
to the left of $g$. This does not affect the Lov\'asz condition, since
\begin{align*}
\begin{pmatrix} f^{-1} & \\  & 1 \end{pmatrix}\begin{pmatrix}1 & \\ m & 1\end{pmatrix}\begin{pmatrix}a_1 & \\ & a_2\end{pmatrix} = \begin{pmatrix}1 & \\ fm & 1\end{pmatrix}\begin{pmatrix}f^{-1}a_1 & \\ & a_2\end{pmatrix}
\end{align*}
and
\begin{align*}
H((m,a_2/a_1)) = H((fm,fa_2/a_1)).
\end{align*}
This corresponds to the changes
\begin{align*}
\mathfrak b_1 \leftarrow f\mathfrak b_1, M_\sigma \leftarrow \begin{pmatrix}\sigma(f^{-1}) & \\ & 1\end{pmatrix}M_\sigma.
\end{align*}
}

\ignore{
\begin{enumerate}[1.]
\item \emph{Size reduction.} 
{\color{red} todo: This part can probably be improved}

The goal is to ensure that $m$ satisfies (i) above.

We do a little preprocessing to make $m$ integral: find $r_0 \in F$ such that $(m-r_0)_\nu \in \cO_\nu$ for all finite $\nu$, then replace $m$ with $m-r_0$. 
{\color{orange} This I might put in the appendix, but here's one way to do this --- there may be something more efficient which I'll look for later. Let $I = (\tau)\subseteq \cO_F$ be any integral principal ideal such that
\begin{align*}
\prod_{\nu \nmid \infty} \nu^{\min{(\ord_\nu m,0)}} \subseteq I^{-1}.
\end{align*}
We want to find the element $r_0$ from $I^{-1}$. 
Each $\nu$ for which $\ord_\nu m < 0$ gives a condition
\begin{align*}
r_{0,\nu} = m_\nu
\end{align*}
in $I_\nu^{-1}/\cO_\nu$, or equivalently,
\begin{align*}
\tau_\nu r_{0,\nu} = \tau_\nu m_\nu
\end{align*}
in $\cO_\nu/I_\nu$.
Recalling the Chinese remainder theorem
\begin{align*}
\cO_F/I \cong \prod_{\nu} \cO_\nu/I_\nu,
\end{align*}
we find $r_0' \in \cO_F$ satisfying the above congruences mod $I$ then take $r_0 = \tau^{-1}r_0'$ --- \emph{cf.} Cohen's Advanced ... book, Ch. 1.3.1.

}

\vspace{4mm}

Now assume $m$ is integral, i.e., $m_\nu \in \cO_\nu$ for all $\nu \nmid \infty$.
Fix a (large) integer $\ell \geq 1$ and a fundamental domain $\mathcal F$ of $\rho(\cO_F) \subseteq F\otimes\R$ satisfying $\ell^{-1}\max_{i,x\in\mathcal F}|\sigma_i(x)| \leq \mu$. The goal is to find $r \in F$ such that $H_\f(m-r)^\frac{1}{d} \cdot \max_i |\sigma_i(m-r)| \leq \mu$.

This is essentially doing a Diophantine approximation, which (classical) LLL is very good at. We first give the following inefficient method, just to show that it's possible.

For each integer $1 \leq a \leq \ell^d+1$, choose $b=b(a) \in \cO_F$ such that $\rho(am-b) \in \mathcal F$. By the pigeonhole principle, there exists $a > a'$ and $b,b'$ such that $\rho(am-b)-\rho(a'm-b') \in \ell^{-1}\mathcal F$, so that
\begin{align*}
\rho\left(m - \frac{b-b'}{a-a'}\right) \in \frac{1}{\ell(a-a')}\mathcal F.
\end{align*}
We set $p = b-b', q = a-a', r = p/q$. Now multiply $g$ from the left by
\begin{align*}
\begin{pmatrix} 1 & \\ -r & 1 \end{pmatrix},
\end{align*}
and update the GSO coefficients accordingly. Note that the updated $m$ now satisfies $H_\f((m,1))^\frac{1}{d}\max_i|\sigma_i(m)| \leq \mu$: since the new $\rho(m)$ lies in $(a-a')^{-1}\ell^{-1}\cal F$, it satisfies
\begin{align*}
(a-a')\max_i|\sigma_i(m)| \leq \mu,
\end{align*}
and $H_\f((m,1))\leq H_\f(a-a')^{-1} = N(a-a') = (a-a')^d$. So if $C = \ell^{d^2}$ then this algorithm works correctly.
\newline

{\color{orange} Here's the poly time method with LLL: Take a $\Z$-basis of $\cO_F$ and let $(m_1,\ldots,m_d)$ be the coordinates of $m$ with respect to that basis. By Prop 1.39 of the original LLL paper, one can find, for any $0 < \varepsilon < 1$, integers $p_1,\ldots,p_d,q$ such that
\begin{align*}
|p_i-qm_i| \leq \varepsilon, 1 \leq q \leq (4/3)^{d(d+1)/4}\varepsilon^{-d}.
\end{align*}
(Here one might want $q$ to be a power of $2$ for efficiency reasons --- this seems attainable with only a little bit of sacrifice on output quality)

Translating this back to the coordinate system on $F \otimes \R$, we obtain $r \in q^{-1}\cO_F$ such that
\begin{align*}
\max_{i}|\sigma_i(m-r)| \leq q^{-1}\varepsilon\max_{i,x\in\cal F}|\sigma_i(x)|.
\end{align*}
(possibly up to some constant depending on the shape of $\mathcal F$.) 

Update $m$ to $m-r$.
Now $H_\f((m,1)) \leq N(q) = q^d$, so this implies
\begin{align*}
H_\f((m,1))^\frac{1}{d}\max_{i}|\sigma_i(m-r)| \leq \varepsilon\max_{i,x\in\cal F}|\sigma_i(x)|.
\end{align*}
We can take $\varepsilon = \ell^{-1}$, and $C = (4/3)^{d^2(d+1)}\ell^{d^2}$.
}\newline

\item \emph{Lov\'asz step.}
Fix $0 < \delta < 1$. If
\begin{align*}
\delta H((a_1,0)) \leq H((m a_1,a_2)),
\end{align*}
or equivalently
\begin{align*}
\delta \leq H((m,\alpha)),
\end{align*}
then move to Step 3. Otherwise multiply $g$ from the left by
\begin{align*}
\begin{pmatrix}  & 1 \\ 1 &  \end{pmatrix},
\end{align*}
update the GSO, and repeat Step 1. \newline

\item \emph{Unit and class reduction.} Compute $f \in F$ such that (i) the entries of $$\Log(f\alpha) - \frac{1}{d}\log H_\infty(f\alpha)\cdot(\underbrace{1,\ldots,1}_\mathrm{real},\underbrace{2,\ldots,2}_\mathrm{complex})$$ are bounded by some constant $A$ depending only on $F$ {\color{orange} basically the diameter of the fundamental domain of the unit lattice} (ii) that $\alpha$ is integral and $B \leq H_\f(f\alpha) \leq 1$ for some constant $B$ again depending only on $F$ {\color{orange} basically the reciprocal of maximum norm among the (chosen) representatives of the ideal class group}. (i) is solving CVP with respect to the unit lattice {\color{red} approx-CVP is fine, at the expense of a worse value of $A$}, and (ii) comes down to computing the ideal class of $\alpha$. {\color{red} this not sure what the complexity is. perhaps Changmin knows. One might want to do similar thing when working with pseudobases.} Then multiply our matrix $g$ from the left by
\begin{align*}
\begin{pmatrix} f^{-1} & \\  & 1 \end{pmatrix},
\end{align*}
and terminate the algorithm.
{\color{blue} 
One can find 
an ideal class of $\alpha$ in $L_\Delta(2/3)= 2^{c\cdot (\log |\Delta|)^{2/3} \cdot (\log\log |\Delta|)^{1/3}}$ time complexity.
In other words, one can compute a prime ideals $\mathcal P_i$, $e_i\in \Z$
such that  
\begin{align*}
    (\alpha) \cdot \prod_{i} \mathcal P_i^{e_i} = (\beta).
\end{align*}
}

{\color{red} One can check that this step actually preserves the Lov\'asz condition established in the previous step. I'll have to state that as a lemma later.}

\end{enumerate}

As Changmin remarked, it may be interesting to compare with the bound in LSPW. Curiously their bound involves the discriminant, whereas mine does not, at least not in a direct way, though one may make the case that the constant $C$ essentially encodes that information.

\subsection*{Practical considerations}
{\color{orange} need to scrutinize our algorithm for the following considerations, especially the second point needs some writing down.}

\begin{itemize}
\item There may be trade-offs between the strength of the size reduction and the output complexity that one can exploit. 

One idea is to compute $m$ to be just small enough to know that we must swap the rows. We may be able to set the target size of $m$ in advance with some computations.

\item If $g$ is given as a pseudobasis --- or a basis, if the associated module is free --- some of these steps may simplify, or may be written in the classical language, so it would operate on the level of pseudobases and not adeles directly.

Especially, if input was a matrix in $\cO_F$, we can try to keep the entries in $\cO_F$ to bypass some computations, such as the first step of size reduction --- indeed this seems to be possible.

\item For certain number fields (especially those of cryptographic interest) their geometry may be particularly simple, leading to speedups in some parts.

\item When unit reducing, is it better to even out the sizes of the entries, or is it better to make it as uneven as possible?
\end{itemize}

\section{Pseudocode in the classical language}

I feel it'll be good to obtain POC quickly, even if under the simplest circumstances possible. Thus in what follows, assume that $F$ is a power-of-$2$ cyclotomic field, and that the input module lattice is a free module. {\color{orange} Actually in retrospect I don't think either assumptions are necessary}

Suppose we already know a $\Z$-basis $\mathcal B = \{v_1,\ldots,v_d\}$ of $\cO_F$, and write $\vv_i = \rho(v_i) \in F \otimes \R$. Let $\mathcal F \subseteq F \otimes \R$ be the fundamental parallelepiped of $\cO_F$ corresponding to this basis that is centered at the origin. Choose $\ell$ satisfying $\ell^{-1}\max_{i,x\in\mathcal F}|\sigma_i(x)| \leq \mu$. {\color{purple} Henry: How do we choose such an $\ell$? I don't think the value of $\max_{i,x\in\mathcal F}|\sigma_i(x)|$ is easy to compute in general. Is this identity easy to check?} {\color{red} It would come from the knowledge of the $\Z$-basis of $\cO_F$. Also it doesn't exactly has to be that value, just needs to be large enough. For 2-power cyclotomic fields, we already know a good choice of $\mathcal B$, so we should be good, right? In this case, some optimization may also be possible, because $\mathcal F$ is a rotated cube, but let's worry about that later.}

Say our module is given by $\ga_1 \vb_1 + \ga_2\vb_2$, where $\ga_1,\ga_2$ are fractional ideals of $F$ and $\vb_1,\vb_2 \in F^2$ are linearly independent over $F$. Following Cohen [REF:Coh96] we assume ideals are given as an (HNF) matrix with respect to $\mathcal B$.

Also set the parameters $\delta,\mu, A,B,C$ appropriately {\color{orange}(more on this later)}. {\color{orange} so far $\delta, \mu$ are any positive reals satisfying $\delta < 1$ and $\delta^\frac{2}{d}-\mu^2 > 0$, $C = (4/3)^{d^2(d+1)}\ell^{d^2}$}

Set
\begin{align*}
M_F = \begin{pmatrix} 1 & 0 \\ m_F & 1 \end{pmatrix} \in \GL(2,F)
\end{align*}
where $m_F=0$ initially, and
\begin{align*}
M_I = \begin{pmatrix} \ga_1^{-1} & \\ & \ga_2^{-1} \end{pmatrix},
\end{align*}
{\color{orange} I abused the notation and wrote for now $M_I$ as ideal-valued matrix, but they really should be understood as elements of the finite idele; also $M_FM_I$ gives the ``finite part'' of the basis under question

\vspace{4mm}

If you have an ideal
$$ \gp_1^{\alpha_1}\gp_2^{\alpha_2}\ldots\gp_k^{\alpha_k}$$
you should read it as an adele/idele whose $\gp_i$-th entry is $\pi_{\gp_i}^{\alpha_i}$, and other entries are $1$

}
and for each embedding $\sigma:F \rightarrow \C$,
\begin{align*}
M_\sigma = \begin{pmatrix} \sigma(\vb_1) \\ \sigma(\vb_2) \end{pmatrix} = \begin{pmatrix} 1 & \\ m_\sigma & 1 \end{pmatrix}\begin{pmatrix} a_{1\sigma} & \\ & a_{2\sigma}\end{pmatrix}U_\sigma
\end{align*}
where $U_\sigma$ is some orthogonal/unitary matrix.
We will continue updating these matrices until we are done.

Below, by $m$, we refer to the adele such that $m_\nu = (m_F)_\nu$ for finite $\nu$, and $m_\sigma$ is the $m_\sigma$ of $M_\sigma$ for each infinite $\sigma$. Similarly, $a_i$ is the adele such that $(a_i)_\nu = \pi_\nu^{-\ord_\nu\ga_i}$ for each finite $\nu$, and $(a_i)_\sigma = a_{i\sigma}$).

\subsection*{Size reduction}

The goal of this routine is to attain the size reducedness
\begin{align*}
H_\f((m,1))^\frac{1}{d}\max_i|\sigma_i(m)| \leq \mu\mbox{ and } H_\f((m,1)) \leq C.
\end{align*}

For each $v_i \in \mathcal B$, let $\vv_i = \rho(v_i) \in \A_\infty$. 
%{\color{orange} Recall that, since $F$ is a power-of-2 cyclotomic field by assumption, we have a particularly nice choice of $v_1,\ldots,v_d$.}
Write
\begin{align*}
\vm_\infty := (m_{\sigma})_{\sigma \mid \infty} = m_1\vv_1+\ldots+m_d\vv_d.
\end{align*}
As explained in the previous section, we can use LLL (use a stronger reduction if desired) to find integers $p_1,\ldots,p_d,q$ such that
\begin{align*}
|p_i-qm_i| \leq \varepsilon, 1 \leq q \leq (4/3)^{d(d+1)/4}\varepsilon^{-d}
\end{align*}
with $\varepsilon = \ell^{-1}$. Taking
\begin{align*}
r = \frac{1}{q}(p_1v_1+\ldots+p_dv_d) \in q^{-1}\cO_F,
\end{align*}
this implies
\begin{align*}
H_\f((m-r,1))^\frac{1}{d}\max_{i}|\sigma_i(m-r)| \leq \varepsilon\max_{i,x\in\cal F}|\sigma_i(x)| \leq \mu, \mbox{ and } H_\f((m-r,1)) \leq C,
\end{align*}
again as explained previously. Update $m \leftarrow m-r$, i.e., $m_F \leftarrow -r$ and $m_\sigma \leftarrow m_\sigma - \sigma(r)$.

\subsection*{Lov\'asz test}

If now the matrix satisfies the Lov\'asz condition
\begin{align*}
\delta H((a_1,0)) \leq H((m a_1,a_2)),
\end{align*}
or equivalently, in more practical terms,
\begin{align*}
\delta N(\ga_1)\prod_{\sigma:F\rightarrow \C} |a_{1\sigma}| \leq N((m_F\ga_1^{-1}+\ga_2^{-1}))^{-1}\prod_{\sigma:F\rightarrow \C}|(m_\sigma a_{1\sigma},a_{2\sigma})|
\end{align*}
{\color{orange} maybe better to do

\begin{align*}
\delta N((m_F\cO_F+\ga_1\ga_2^{-1}))\prod_{\sigma:F\rightarrow \C} |a_{1\sigma}| \leq \prod_{\sigma:F\rightarrow \C}|(m_\sigma a_{1\sigma},a_{2\sigma})|
\end{align*}

}
(here beware that $\sigma$ is summed over all \emph{embeddings}, not \emph{places}) 
then we move onto the next step of the algorithm and then terminate. If not, we carry out the following preprocessing and return to the size reduction step.

Let us start by updating the finite/arithmetic parts --- $M_F$ and $M_I$ --- first.
Start by writing $m_F = \alpha/\beta$ for $\alpha,\beta \in \cO_F$; note that this is already accomplished, because $\beta = q$. Compute nonzero $c \in \cO_F$ such that $c\alpha\ga_1^{-1}, c\beta\ga_2^{-1}$ are both integral. Suppose $c\alpha\ga_1^{-1} + c\beta\ga_2^{-1} = \gd$. Then by Theorem 1.2 of [Coh96], we can efficiently compute $u \in (\ga_1\gd)^{-1}, v \in (\ga_2\gd)^{-1}$ such that
\begin{align*}
c\alpha u + c\beta v = 1.
\end{align*}
Now set
\begin{align*}
M_F &= \begin{pmatrix} \alpha/\beta & 1 \\ 1 & \end{pmatrix}\begin{pmatrix} u & c\beta \\ v & -c\alpha\end{pmatrix} = \begin{pmatrix} 1/(c\beta) & \\ u & c\beta \end{pmatrix}, \\ 
M_I &= \begin{pmatrix} \gd & \\ & (\ga_1\ga_2\gd)^{-1} \end{pmatrix}.
\end{align*}
{\color{orange} \emph{Why does this work?} Interpret
\begin{align*}
M_FM_I = \begin{pmatrix}\alpha/\beta & 1 \\ 1 & \end{pmatrix}\begin{pmatrix}\ga_1^{-1} & \\ & \ga_2^{-1} \end{pmatrix}
\end{align*}
(after swapping the two rows)
as an element of $\GL(2,\A_\f)$, where $\A_\f$ is the set of finite adeles. For each $\nu \nmid \infty$, we are allowed to adjust $M_FM_I$ from the right by an element of $\GL(2,\cO_\nu)$, since such maneuver does not alter the underlying lattice. So we need to confirm that the natural projection of
\begin{align*}
\begin{pmatrix}u\ga_1 & c\beta \ga_1 \\ v\ga_2 & -c\alpha \ga_2 \end{pmatrix} \begin{pmatrix}\gd & \\ & (\ga_1\ga_2\gd)^{-1}\end{pmatrix} = \begin{pmatrix} u\ga_1\gd & c\beta(\ga_2\gd)^{-1} \\ v\ga_2\gd & -c\alpha(\ga_1\gd)^{-1}\end{pmatrix}
\end{align*}
to $\nu$-adics yields an element of $\GL(2,\cO_\nu)$. The determinant is $1$ by construction, so we just need to see that the right-hand side is integral --- which again follows by construction.
}

As for the infinite/analytic parts --- $M_\sigma$ for $\sigma \mid \infty$ --- simply swap the rows of the matrix.

Next, multiply $M_F$ from the left by
\begin{align*}
\begin{pmatrix} c\beta & \\ -u & 1/(c\beta) \end{pmatrix},
\end{align*}
and each $M_\sigma$ by
\begin{align*}
\begin{pmatrix} \sigma(c\beta) & \\ \sigma(-u) & \sigma(1/(c\beta)) \end{pmatrix}.
\end{align*}
{\color{orange} \emph{Why is this okay?} Recall left multiplication by $\GL(2,F)$ preserves the lattice in the adelic setting.

Also Henry suggested it's also possible to update $M_\sigma$'s globally if desired.
}

This turns $M_F$ back to the identity, as in the initial setting. Update the GSO of $M_\sigma$'s, and go back to the size reduction step.

\subsection*{Unit and class reduction} 

These are already understood procedures, and Changmin probably understands them better than I do.
{\color{orange} Not having carefully thought about the implementation perspective before, I thought these need to be done only once. But to keep the bitsizes of the matrices down during the course of the algorithm, it may have to be recalled more frequently.}

{\color{orange} Remark: Notice the output is not necessarily in pseudobasis form. (In fact, input doesn't have to be either.) This is because pseudobasis is in some sense limited in its expressive power: e.g., it's not very compatible with the notion of being size reduced. If so desired, one can ``un-size reduce'' at the end and output a pseudobasis.}
}

\section{Adelic LLL in rank $n$}

\subsection{Conceptual version} 
\begin{definition}[Rank n reduced basis]\label{def:LLL-reduced-n}
Fix parameters $0 < \mu,\delta,A,B,C$, as in the $n=2$ case.
We say $g \in \GL(n,\A_F)$ is \emph{(LLL-)reduced} with respect to these parameters if it has an Iwasawa decomposition $g=uak$ such that
\begin{enumerate}[label=(\roman*)]
\item (size reduced) For every $m=m_{kj}$, there exists $\tau_\sigma > 0$ for each of the $d$ embeddings $\sigma:F\hookrightarrow\C$ such that
\[\tau_\sigma \leq C^\frac{1}{d}, \text{ } \tau_\sigma|m_\sigma| \leq \mu \text{ and }
H_\f(m,1) \leq\prod_{\sigma:F\hookrightarrow\C}\tau_\sigma.\]
\item (Lov\'asz condition) $\delta \leq H((m_{k+1,k},\alpha_k))$ for all $1 \leq k < n$.
\item (unit reduced) The entries of $$\Log(\alpha_k) - \frac{1}{d}\log H_\infty(\alpha_k)\cdot(\underbrace{1,\ldots,1}_\mathrm{real},\underbrace{2,\ldots,2}_\mathrm{complex})$$ are bounded by $A$ for all $1 \leq k < n$.

\item (class reduced) $\alpha_k$ is integral and $B^{-1} \leq H_\f(\alpha_k)$ for all $1 \leq k < n$. % same change as in rank 2 case
\end{enumerate}
\end{definition}

Combining the ideas from the classical LLL \cite{LLL82} and the rank $2$ case, every aspect of the algorithm generalizes in a straightforward way. For example, see Algorithm \ref{alg:lll_c_n}, and the following generalizations of Propositions \ref{prop:swap_bound} and \ref{prop:rhf_bound}; the proofs are similar, and are delegated to Appendix \ref{app:proof}.

\begin{algorithm}
\caption{Adelic LLL algorithm (conceptual version)} \label{alg:lll_c_n}
\begin{algorithmic}[1]
\Require Parameters $\delta,\mu,
A,B,C$ and $g=uak \in \GL(n,\A_F)$.
\Ensure LLL-reduced $g \in \GL(n,\A_F)$.
\State $k \leftarrow 1$
\State \textbf{while} $k<n$ \textbf{do}
\State \ \ \texttt{class-reduce} $\alpha_k$
\State \ \ \texttt{unit-reduce} $\alpha_k$
\State \ \ \texttt{size-reduce} $(k+1)$-st row of $g$ %in decreasing order  
\State \ \ \textbf{if}(\texttt{Lovasz}$(k,g)$)
\State \ \ \ \ $k \leftarrow k+1$
\State \ \ \textbf{else} 
\State \ \ \ \ swap rows $k,k+1$ of $g$
\State \ \ \ \ $k \leftarrow \max(1,k-1)$
\State \textbf{return} $g$
\end{algorithmic}
\end{algorithm}

\begin{proposition}[Rank n termination]
\label{prop:swap_bound_n}
Define $E(g) = \displaystyle\prod_{h=1}^{n-1} H(v_1\wedge\cdots\wedge v_h)$, where $v_k$ is the $k$-th row vector of $g$. Suppose 
\begin{align*}
\kappa = \min\{E(g'): g' \in \GL(n,F)g\prod_{\nu \nmid \infty} \GL(n,\cO_\nu)\} > 0.
\end{align*}
Then Algorithm \ref{alg:lll_c_n} takes at most $(\log {E(g)} - \log{\kappa})/(-\log \delta)$ swaps to terminate.
\end{proposition}

It follows from well-known facts in the geometry of numbers that $\kappa > 0$ for every $g \in \GL(n,\A_F)$ --- see e.g., \cite[p.18]{BV83} and  \cite[Theorem 3.1.3]{MW01}. An explicit bound such as Lemma \ref{lemma:lowest_ht} can also be obtained.

%For $g \in \GL(n,\A_F)$ corresponding to $\gb_1\vv_1+\cdots+\gb_n\vv_n$ with integral $\vv_i$, Lemma \ref{lemma:lowest_ht} can be straightforwardly generalized to show that, $H(v)$ where $v \in F^ng\backslash\{0\}$ can be lower bounded by $N\left(\sum\gb_i\right)$. Similarly one can lower bound $H(v_1\wedge\cdots\wedge v_h)$, for any $h=1,\ldots,n$, since $v_1\wedge\cdots\wedge v_h$ is a vector of a module lattice of rank $\binom{n}{h}$ with an integral pseudobasis; alternatively one can apply \cite[Theorem 3.1.3]{MW01}.
\begin{proposition}[Rank n quality] \label{prop:rhf_bound_n}
Let $Q$ be as in Proposition \ref{prop:rhf_bound}, and let $g\in\GL(n,\A_F)$ be reduced, and let $v_1$ denote its first row vector. Then
\begin{align*}
H(v_1) & \leq Q^{n-1}H(\det g)^\frac{1}{n}. \\
H(v_1) & \leq Q^{2(n-1)} \eta(g) %\min_{v \in F^ng\backslash\{0\}}H(v).
\end{align*}
\end{proposition}
\ignore{
By applying Proposition \ref{prop:rhf_bound} to $2 \times 2$ diagonal blocks of the $ua$ part of $g$, we have
\begin{align*}
H(a_i)^\frac{1}{2} \leq QH(a_{i+1})^\frac{1}{2}
\end{align*}
for every $i$, and consequently
\begin{align*}
H(a_i)^\frac{1}{2} \leq Q^jH(a_{i+j})^\frac{1}{2}
\end{align*}
for every $j$ as well.

Thus
\begin{align*}
&H(a_1)^\frac{n-1}{2} \leq Q^{1+\ldots+(n-1)}H(a_2)^\frac{1}{2}\cdots H(a_n)^\frac{1}{2} \\
&\Rightarrow H(a_1)^\frac{n}{2} \leq Q^\frac{(n-1)n}{2}H(\det g)^\frac{1}{2} \\
&\Rightarrow H(a_1) \leq Q^{n-1}H(\det g)^\frac{1}{n}
\end{align*}
as desired. The final inequality follows from Lemma~\ref{lemma:diagonal}.
}
%\end{proof}
\begin{remark}
Observe that, if we choose $F=\Q$ and $A=0, B=1, C=1$, then we get $Q=(\delta^2-\mu^2)^{-1/4}$. Therefore Proposition \ref{prop:rhf_bound_n} recovers the well-known RHF upper bound for the original LLL algorithm.
\end{remark}

%{\color{orange} might prove some more statements if there's time}

\subsection{Implementable pseudobasis version}

% size reduce
\ignore{
\begin{algorithm} 
\caption{\texttt{size-reduce}}
\label{alg:sizereduce_n}
\begin{algorithmic}[1]
\Require Parameters $\mu,C,\ell$, basis $\mathcal B=\{w_1,\ldots,w_d\}$ of $\cO_F$, $1 \leq j<k\leq n$, and input data $D=(M_F,M_I,M_\infty,M'_\infty)$.
\Ensure $m_{kj}$ satisfies $$H_\f(((m_F)_{kj},1))^\frac{1}{d}\max_{\sigma\mid\infty}|(m_\sigma)_{kj}| \leq \mu\mbox{ and } H_\f(((m_F)_{kj},1)) \leq C.$$
\State Compute the GSO coefficients of $M_\infty$ to find $(\vm_\infty)_{kj}$
\State Compute reals $m_i$ such that $(\vm_\infty)_{kj} = \sum m_i\rho(w_i)$
\State Compute integers $p_1,\ldots,p_d,q$ such that $$|p_i-qm_i| \leq (2\ell)^{-1}, 1 \leq q \leq C$$
\State $r \leftarrow q^{-1}(p_1w_1+\ldots+p_dw_d)$
\State $\displaystyle M_F \leftarrow E_{kj}(-r)M_F,\ M_\infty \leftarrow E_{kj}(-\rho(r))M_\infty$
\State \textbf{return} $D$
\end{algorithmic}
\end{algorithm}
}
\ignore{
%lovasz test
\begin{algorithm}
\caption{\texttt{Lovasz-test}}
\label{alg:lovasz_n}
\begin{algorithmic}[1]
\Require Parameter $\delta$, $k$, and $D=(M_F, M_I, M_\infty,M'_\infty)$
\State \textbf{if}
$$
\delta N(((m_F)_{k+1,k}\cO_F+\gb_k\gb_{k+1}^{-1}))\leq \prod_{\sigma:F\rightarrow \C}|((m_{k+1,k})_\sigma,a_{k+1,\sigma}/a_{k\sigma})|
$$
\textbf{return} true
\State \textbf{else} \textbf{return} false
\end{algorithmic}
\end{algorithm}
}
% adelic swap and clean-up
\begin{algorithm}
\caption{\texttt{adelic-swap}} 
\label{alg:adelicswap_n}
\begin{algorithmic}[1]
\Require $k,D=(M_F,M_I,M_\infty,M'_\infty)$ that failed \texttt{Lovasz-test}
%\Ensure $M_F,M_I,(M_\sigma)_{\sigma\mid\infty}$ modified for next \texttt{size-reduce}
\State Compute $x,y\in\cO_F$ such that $(m_F)_{k+1,k} = x/y$
\State Compute $c \in \cO_F$ such that $cx\gb_k^{-1},cy\gb_{k+1}^{-1} \subseteq \cO_F$
\State $\gd \leftarrow cx\gb_{k}^{-1}+cy\gb_{k+1}^{-1}$
\State Compute $u\in(\gb_k\gd)^{-1},v\in(\gb_{k+1}\gd)^{-1}$ such that $cxu+cyv=1$ %{\color{purple} technically we don't need to know $v$ here I guess} {\color{orange} yes but it comes in pair with $u$ so..}
\State $M'_\infty \leftarrow B_k\begin{pmatrix} \rho(u) & \rho(cy) \\ \rho(v) & \rho(-cx)\end{pmatrix}^{-1}M'_\infty$
\State $M_F \leftarrow B_k\begin{pmatrix}  & cy \\ 1/(cy) & -u\end{pmatrix}\cdot M_F \cdot B_k\begin{pmatrix} u & cy \\ v & -cx\end{pmatrix}$ %// $(m_F)_{k+1,k}$ becomes zero
\State $M_I \leftarrow B_k\begin{pmatrix} \gb_k\gd & \\ & (\gb_k\gd)^{-1}\end{pmatrix}M_I$
\State $M_\infty \leftarrow B_k\begin{pmatrix} & \rho(cy) \\ \rho(1/(cy)) & \rho(-u) \end{pmatrix}M_\infty$
\end{algorithmic}
\end{algorithm}
% scaling (both row k and k+1)
\begin{algorithm}
\caption{\texttt{clean}} \label{alg:clean_n}
\begin{algorithmic}[1]
\Require $k, D=(M_F,M_I,M_\infty,M'_\infty)$
\Ensure $(m_F)_{h,h'}=0$ for all $h=k$ and $k+1$, $h'<h$ %and the pseudobasis $((\gb_h,\vv'_h)_{1\leq h \leq k})$ is scale+d
%\State \texttt{scale+}$((\gb_h,\vv_h')_{1\leq h \leq k})$ \ \ // this is already scale+'d for $h < k-1$
%\State \textbf{for} $h=k-1$ \textbf{to} $1$
%\State \ \ $M_\infty \leftarrow E_{k+1,h}(-(m_F)_{k+1,h})E_{kh}(-(m_F)_{kh})M_\infty$
%\State \ \ $(m_F)_{k+1,h} \leftarrow 0, (m_F)_{kh} \leftarrow 0$
\State $\vv_k \leftarrow \vv'_k, \ \vv_{k+1} \leftarrow \vv'_{k+1}$
\State $(m_F)_{h,h'} \leftarrow 0$ for all $h=k$ and $k+1$, $h'<h$
\end{algorithmic}
\end{algorithm}

% end product
\begin{algorithm}
\caption{Adelic LLL algorithm (pseudobasis version)} \label{alg:lll_p_n}
\begin{algorithmic}[1]
\Require Parameters $\delta,\mu,A,B,C$, basis $\mathcal B$ of $\cO_F$, choices of HSVP solvers for the subroutines, and a pseudobasis of the input module $\gb_1\vv_1+\ldots+\gb_n\vv_n \subseteq \rho(\cO_F^n)$.
\Ensure Data $D=(M_F,M_I,M_\infty,M'_\infty)$ representing a reduced element of $\GL(n,\A_F)$ for the input module.
\State $\displaystyle M_F \leftarrow \mathrm{Id}_n, M_I \leftarrow \diag(\gb_1^{-1},\ldots,\gb_n^{-1}), M_\infty \leftarrow \begin{pmatrix} \vv_1^T & \cdots & \vv_n^T\end{pmatrix}^T,M'_\infty \leftarrow M_\infty$
\State $k \leftarrow 1$
\State \textbf{while} $k<n$ \textbf{do}
\State \ \ \textbf{if}($k=1$) \texttt{scale}$(\gb_1,\vv_1,a_{1\infty})$, $\vv'_1 \leftarrow \vv_1$
\State \ \ \texttt{scale}$(\gb_{k+1}\gb_{k}^{-1},\vv_{k+1},a_{k+1,\infty})$
\State \ \ \texttt{unit-reduce} row $k+1$
\State \ \ $\vv'_{k+1} \leftarrow \vv_{k+1}$
\State \ \ Apply \cite[Algorithm 3.3]{LPSW19} to row $k+1$; update $M_\infty,M'_\infty$ accordingly 
\State \ \ \texttt{size-reduce} $m_{k+1,k},\ldots,m_{k+1,1}$; update $M_F, M_\infty$ accordingly
\State \ \ \textbf{if}(\texttt{Lovasz-test} at row $k$) 
\State \ \ \ \ $k \leftarrow k+1$
\State \ \ \textbf{else} 
\State \ \ \ \ \texttt{adelic-swap} rows $k$ and $k+1$
\State \ \ \ \ \texttt{clean} rows $k, k+1$
\State \ \ \ \ $k \leftarrow \max(1,k-1)$
\State \textbf{return} $D$
\end{algorithmic}
\end{algorithm}

We follow much of the notations and conventions from the previous section on the $n=2$ case. We let $\mathcal B$ and $\mathcal U$ be as earlier, and assume the input module is given in the integral pseudobasis form $\gb_1\vv_1+\cdots+\gb_n\vv_n \subseteq \rho(\cO_F^n)$ with $\vv_1,\ldots,\vv_n \in \rho(\cO_F^n)$. The matrices $M_F,M_I,M_\infty=(M_\sigma)_{\sigma\mid\infty},g$ are defined in the analogous manner, as well as the notations referring to their entries, except now that $M_F$ has multiple nontrivial entries we refer to its $(k,j)$-entry by $(m_F)_{kj}$. Additionally, we define $M'_\infty = \rho(M_F^{-1})M_\infty$, and refer to its $k$-th row as $\vv'_k$; we will need this for bitsize control.

Algorithms \ref{alg:adelicswap_n} to \ref{alg:lll_p_n} provide a complete pseudocode description of adelic LLL for rank $n$ module lattices, with Algorithm \ref{alg:lll_p_n} being the main algorithm. We omitted the descriptions of \texttt{unit-reduce} and \texttt{Lovasz-test} for brevity, as they are straightforward generalizations of their rank 2 counterparts. Indeed, as with the classical LLL, it operates on one $2 \times 2$ block at a time, and the operation is no different than the rank 2 adelic LLL. 

A few notations are introduced to pinpoint where the algorithm is operating. 
In Algorithm \ref{alg:adelicswap_n}, $B_k(M)$ for a $2\times 2$ matrix $M$ denotes the matrix constructed by first taking the $n \times n$ identity matrix and then replacing the $2 \times 2$ block matrix starting at $(k,k)$-entry by $M$.
In Algorithm \ref{alg:clean_n}, $E_{kj}(x)$ is the elementary matrix constructed by first taking the $n \times n$ identity matrix and then taking its $(k,j)$-entry to be $x$. In addition, line 5 of Algorithm \ref{alg:adelicswap_n}, and the correctness of Algorithm \ref{alg:clean_n}, are both explained by the relation $M'_\infty = \rho(M_F^{-1})M_\infty$.

\subsubsection{A note on bitsize control} We continue to employ the bitsize control method of \cite{LPSW19}. One difference from the rank $2$ case is that, while in Algorithm \ref{alg:lll_p} the entire data at the start and the end of the loop is in pseudobasis form, in the general case it is not. This is why we use the additional bookkeeping variable $M'_\infty$, whose $i$-th row vector we denote by $\vv_i'$. The first $k$ rows of $g$ span the same module lattice as $\gb_1\vv'_1+\ldots\gb_k\vv'_k$, and this is what we use to apply (the corrected version of) \cite[Algorithm 3.3]{LPSW19} in line 8 of Algorithm \ref{alg:lll_p_n}. Similarly as in the proof of Theorem \ref{thm:main}, we obtain the following.

%Each time at the start of the loop in Algorithm \ref{alg:lll_p_n}, prior to running step $4$, it is ensured that $(m_F)_{h+1,h}=\ldots=(m_F)_{h+1,1}=0$ for all $h \geq k$ and that $(\gb_k,\vv_k)$ and $(\gb_{k+1},\vv_{k+1})$ are scale+d. The fact that rows $h \geq k$ of $M_F$ are ``clean'' helps the algorithm in a couple of spots. In Algorithm \ref{alg:sizereduce_n}, it initializes $M_F$ for size reduction; in Algorithm \ref{alg:adelicswap_n}, step 6, it simplifies the matrix multiplication from the right. Most importantly, it allows to keep the bitsizes of $\vv_k$ and $\vv_{k+1}$ in check at the end of the loop. As argued in \cite[Sec 3.2]{LPSW19} and its correction by our Appendix below, their sizes are bounded by $\mathrm{poly}(\log H(a_h),d,\log|\Delta_F|)$ for $h=k,k+1$ respectively. For all $h=1,\ldots,n$, $H(a_{h})$ is bounded throughout the algorithm, since $H(a_1)\cdots H(a_{h+1})$ never increases during the algorithm, and $H(a_1)\cdots H(a_{h})$ has a lower bound. This keeps the sizes of $M_I$ and $M_\infty$ bounded by $\mathrm{poly}(input,d,\log|\Delta_F|)$. Therefore, we obtain the following generalization of Theorem \ref{thm:main}.

\begin{theorem}[Rank n reduction]
\label{thm:main_n}
Algorithm \ref{alg:lll_p_n} returns the data $(M_F,M_I,M_\infty,M'_\infty)$ representing a reduced element of $\GL(n,\A_F)$ corresponding to the input module in $\mathrm{poly}(\text{input},d,\log|\Delta_F|)$ and $\mathrm{poly}(\text{input})$ many calls to the chosen HSVP oracles.
\end{theorem}

\subsection{Theoretical performance of adelic LLL} \label{sec:reduction}

Following \cite{MS20}, we define $(\gamma,n)$-ModuleHSVP (resp.  $(\gamma,n)$-ModuleSVP) to be the problem of finding, given a rank $n$ module lattice $L$ over a field $F$ of degree $d$, a nonzero vector of $L$ of length at most $\gamma(\det L)^{1/(nd)}$ (resp. $\gamma \lambda_1(L)$). We also refer to $(\gamma,1)$-ModuleHSVP as $\gamma$-IdealHSVP.

Let $0 < \delta,\mu < 1$, $A,B,C$ be the parameters chosen for adelic LLL (Algorithm \ref{alg:lll_p}). Recall
\begin{align*}
Q = (\delta^\frac{2}{d}-\mu^2)^{-\frac{d}{4}}(CB)^\frac{1}{2}e^\frac{Ad}{2}.
\end{align*}

Proposition \ref{prop:rhf_bound_n} establishes that adelic LLL, applied to a rank $n$ module lattice $L=L(g)$, returns a rank 1 submodule $\Lambda \subseteq L$ of height at most $Q^{n-1}H(\det g)^{1/n}$ (resp. $Q^{2(n-1)} \eta(g)$), or equivalently, of determinant at most $Q^{n-1}(\det L)^{1/n}$
(resp.  $\Delta_F^{1/2} Q^{2(n-1)} \eta(g)$). Hence, if one solves $\gamma$-IdealHSVP,
then one finds a vector of $\Lambda$ of length at most $\gamma\cdot Q^{(n-1)/d}(\det L)^{1/(nd)}$ (resp. $\gamma \cdot \Delta_F^{1/(2d)}  Q^{2(n-1)/d} \eta(g)^{1/d}$), then trivially one also finds a vector of $L$ of length at most $\gamma \cdot Q^{(n-1)/d} (\det L)^{1/(nd)}$
(resp. $\gamma\cdot\Delta_F^{1/(2d)}  Q^{2(n-1)/d} d^{-1/2} \lambda_1(L)$ by Lemma~\ref{lemma:module}). This proves
\begin{theorem}\label{thm:reduction}
Fix a number field $F$, and let $Q$ be as above. Then there exists a reduction from $(\gamma Q^{(n-1)/d},n)$-ModuleHSVP and $(\gamma \Delta_F^{1/(2d)}  Q^{2(n-1)/d} d^{-1/2},n)$ -ModuleSVP to $\gamma$-IdealHSVP, of $\mathrm{poly}(\text{input},d,\log|\Delta_F|)$ time complexity and $\mathrm{poly}(\text{input})$ calls to the chosen HSVP oracle.
\end{theorem}

Naturally one would be interested in the size of $Q$, since Theorem \ref{thm:reduction} is most meaningful when $\gamma$ is much greater than $Q^{(n-1)/d}$ (and when $n$ is small relative to $d$). One can choose $\delta=0.999$ and $\mu=0.5$ as in the classical LLL. As discussed earlier, one can take $C = \mu^{-d}\gamma_{2d}^{4d^2}|\Delta_F|$ (taking $E=F$ in Algorithm \ref{alg:sizereduce}) and $B=\gamma_d^{d^2}|\Delta_F|^{1/2}$. 

For a further discussion, let us restrict to the case $F=\Q(\zeta_{2^{\kappa+1}})$ ($\kappa \geq 2$), so that $d= 2^\kappa$. In this case, $|\Delta_F|=d^d$. In addition, from \cite[Corollary 6.4]{CDPR16}, one can take $A= O(\sqrt{d\log d})$. Hence we obtain
$$ Q^\frac{1}{d} = \gamma_{2d}^{2d}\gamma_d^{d/2}e^{O(\sqrt{d\log d})}.$$
If we use the slide algorithm \cite{GN08} or SDBKZ \cite{MW16} of blocksize $\beta$ for the subroutine, for which rigorous performance estimates are available, then $\gamma_d^d \leq \beta^{d/(2\beta)}$, and likewise $\gamma_{2d}^{2d} \leq \beta^{d/\beta}$. Taking $\beta = \sqrt d$ for example, we have
$$ Q^\frac{1}{d} \leq e^{\frac{5}{8}\sqrt{d}\log d \cdot(1+o_d(1))}.$$
Therefore, for instance, if $n\leq\log d$ and $\gamma \geq e^{\sqrt{d}(\log d)^3}$, then Theorem \ref{thm:reduction} yields a $e^{O(\sqrt{d})}$-time reduction from $(\gamma^{1+o_d(1)},n)$-ModuleHSVP to $\gamma$-IdealHSVP.

One may also be interested in the following comparison. Suppose $\beta \leq \sqrt d$. Applying the slide algorithm or SDBKZ of blocksize $\beta$ directly to a rank $n$ module lattice $L$, one obtains a vector of length $\approx \beta^{nd/(2\beta)}\cdot(\det L)^{1/(nd)}$. On the other hand, if we apply adelic LLL using the same algorithm as the SVP oracle, and apply the same oracle to the resulting rank 1 submodule as well, then one obtains a vector of length $\approx \beta^{d/(2\beta)}\cdot\beta^{2.5nd/(2\beta)}\cdot(\det L)^{1/(nd)}$ up to terms of lesser size, by the formula for $Q^{1/d}$ above. In this sense, it may be said that the current version of adelic LLL is asymptotically ``2.5 times worse'', since, for large $n$ and $d$, $\beta$ needs to be about a factor of 2.5 greater to match the output quality of BKZ-$\beta$. This loss mainly comes from the size of the parameter $C$, which measures how well one can solve the Diophantine approximation problem associated to the size reducedness. This motivates the following question.

\begin{question}
Could \texttt{size-reduce} (Algorithm \ref{alg:sizereduce}) be improved to lower the size of the required $C$, without increasing the time complexity? In particular, would it be possible to exploit the field structure to achieve this?
\end{question}

%{\color{orange} for very nice fields maybe we can claim superiority of adelic LLL --- will think about it later --- and we can connect to Chinburg's works}

This is not entirely implausible, in light of the following observation. %{\henry{Actually I think the following only holds for fields with no complex embeddings, the volume calculation involves discs otherwise so we should get a factor $4/\pi$ per pair or something of the sort which I don't see here}}{\color{orange} I think it applies to the complex case, and it's just that I'm not using the complex structure in any smart way... or it can be an issue of me confusing the metric. In any case conclusion doesn't change}
\begin{proposition}\label{prop:minkowski}
If $C= \mu^{-d}|\Delta_F|$, then there exists $p,q \in \cO_F$ satisfying \eqref{eq:sizereduce_new} for any given $m_\infty\in\A_\infty$. 
\end{proposition}
\begin{proof}
Consider the lattice $\Lambda$ spanned by the rows of \eqref{eq:thematrix}, with $E=F$. This is a rank $2d$ lattice with $\det \Lambda =\omega^d|\Delta_F|$. Therefore, by Minkowski's first theorem, there exists a nonzero vector of this lattice in the box $[-\mu^{-1}\omega|\Delta_F|^{1/d},\mu^{-1}\omega|\Delta_F|^{1/d}]^d \times [-\mu,\mu]^d$, since this box has volume $2^{2d}\omega^d|\Delta_F|=2^{2d}\det \Lambda$. Recall that every vector of $\Lambda$ is of the form $(\omega \rho(q), \rho(qm_\infty-p))$; the elements $p,q\in\cO_F$ corresponding in this manner to the vector in the box is clearly seen to satisfy \eqref{eq:sizereduce_new} with $C= \mu^{-d}|\Delta_F|$.
\qed
\end{proof}

This statement suggests that the extra factor $\gamma_{2d}^{4d^2}$ in our choice of $C$ does not reflect some intrinsic limitation of adelic LLL, but rather a limitation of the current method of size reduction.

%Therefore, we could in principle take $C=\mu^{-d}|\Delta_F|$, far better than our formula $C=\mu^{-d}\gamma_{2d}^{4d^2}|\Delta_F|$, even if we use an exact SVP algorithm (so $\gamma_{2d}^{2d} = \Theta(\sqrt{d})$). Hence the current limitation of adelic LLL seems to be not an intrinsic but a methodical one.
%\phong{I'm not sure what we're trying to say here. Sure, we can better parameters for which we are guaranteed that the HSVP instance has a soution. But in that case, solving that HSVP means solving SVP (or HSVP with very small approximation factors) in dimension $2d$, whereas the initial module lattice has rank $2d$ over the integers, so we might have has well applied our HSVP oracle to the input lattice directly.} 
%{\color{orange} All I'm trying to say is the extra factor $\gamma_{2d}^{4d^2}$ appears because of our method, not because of the way things are. So a better method --- perhaps some kind of continued fraction? my \$5-a-month AI also explored some linear programming methods --- can give a smaller extra factor. 8/26/26 fixed the wording a little bit}

If one can find a way to reduce the exponent on $\gamma_{2d}$ by a factor of half or more, then there is hope for adelic LLL to theoretically outperform the standard reduction algorithms in practically important cases. The loss from the size of $B$ can be circumvented by other methods. For instance, since it is widely believed that the maximal totally real subextension of $\Q(\zeta_{2^{\kappa+1}})$ has class number one --- under the buzzword \emph{Weber's class number problem}; see for instance \cite{Mil15} --- using the method of \cite{BS16} it is possible to set $B=1$ for modules over a field of such kind.

\begin{remark}
One may be interested in making the above comparisons in terms of the heights of the rank $1$ submodules, rather than the lengths of the lattice vectors as we have done here. It turns out that doing so one observes essentially the identical quantitative behavior.

To elaborate, continue with the earlier assumptions, and suppose one applies BKZ-$\beta$ to find a vector $\vv$ of $L$. By \cite[Lemma 2.2]{Kim24} for instance, the rank 1 submodule $\cO_F \cdot \vv$ has height at most
$$d^{-d/2}\|\vv\|^{d} \leq d^{-d/2}\beta^{\frac{nd^2}{2\beta}}\cdot(\det L)^\frac{1}{n}.$$
Comparing with the height bound of adelic LLL 
$$Q^{n-1}(\det L)^\frac{1}{n} \approx \beta^{2.5nd^2/(2\beta)}\cdot(\det L)^\frac{1}{n},$$
one again notices that, asymptotically in $n$ and $d$, the latter is 2.5 times worse.
\end{remark}

\newpage

\ignore{

\subsection{Theoretical performance analysis}
{\color{orange} this is where we discuss the reduction to ideal-SVP, and the comparison of running adelic LLL versus running unstructured reduction.

For $\ell$ of the max real s/fld of 2-power cyclotomic field, see Prop 4.1 of \cite{FN05}. I think $B=1$ for the real s/fld (on Weber's conjecture), and $A$ and $\ell$ are similar to the cyc fld. This saves us from the disaster --- I hope.}

\subsection{Computing constants $A$ and $B$}

Using \cite[Proof of Theorem 2.8]{CDPR16}, we get that $\Log(\cO_F^\times)=2\Log(\cO_{F^+}^\times)$, and this implies that the constant $A_{F^+}=\frac{1}{2}A_F$ (notice the factor $2$ difference with CDPR16, because of a slightly different definition of $\Log$). This is essentially using \cite[Cor. 4.13]{Was97} to say that $\Z[\zeta]^\times$ is generated by $\zeta$ and $\Z[\zeta+\zeta^{-1}]^\times$. Therefore we indeed get $A_{F^+}=O(\sqrt{d\log d})$ from \cite[Cor. 6.4]{CDPR16} for prime-power cyclotomic fields $F$. This part is not conditional on Weber's conjecture.

As stated before, conditionally on Weber's class number conjecture, for cyclotomic fields $F=\Q(\zeta_{2^k})$ with $k>1$, we indeed get $B_{F^+}=1$ as there is only one ideal class, and this class contains $\cO_F$, the integral ideal with smallest possible norm. Note that Weber's class number is not really a limitation in practice as the conjecture holds for $k\le 8$ and for $k=9$ under GRH (see Miller14 ref in CDPR16). $k=9$ is the case used in cryptographic standards such as ML-KEM.

{\color{orange} Thanks Henry, I was going to do this myself... but while you're at it, do your references say something about $\ell$ as well?}

\henry{I will write this down eventually, we need to take basis elements $\zeta^i+\zeta^{-i}$ and then $\ell$ stays roughly the same. There is an efficient algorithm to convert from the orthogonal basis of $\cO_F$ to this basis of $\cO_{F^+}$, we might need to check that this doesnt affect sizes too badly.}

{\color{purple}\subsection{RHF computation} Here's an attempt at the RHF argument for general ranks. Please check the details!

Assume $[F:\Q]=d$, and that we are given an ad\`ele $g\in \GL(n,\A_F)$. The aim is to compare the following two strategies for finding a short vector in $L(g)$.
\begin{enumerate}
    \item Apply LLL directly on the $nd$ dimensional lattice $L(g)$. This outputs a vector $s_1\in L(g)$.
    \item Apply rank $n$ adelic LLL on $g$. This outputs a rank $1$ submodule $v_1$. Then apply LLL on $L(v_1)$. This yields a vector $s_2\in L(g)$.
\end{enumerate}

Now we analyse both strategies, letting $\gamma := (4/3)^{\frac14}$ and using $\gamma^d$ RHF for LLL throughout (instead of $\gamma^{d-1}$ but thats close enough).
\begin{enumerate}
    \item The first strategy is straightforward: we get \[\|s_1\|\le \gamma^{nd}\left(\det(L(g))\right)^{\frac{1}{nd}}.\]
    \item In the second strategy, we use \ref{prop:rhf_bound_n} to get $H(v_1)\le Q^{n-1}H(\det g)^{\frac{1}{n}}$, as well as the LLL bound to get $\|s_2\|\le \gamma^d \left(\det L(v_1)\right)^{\frac{1}{d}}$. Combining those inequalities with the fact that $\det L(v_1)=|\Delta_F|^{\frac12}H(v_1)$ and $\det L(g)=|\Delta_F|^{\frac{n}{2}}H(\det g)$, we get
    \begin{align*}
        \|s_2\| &\le \gamma^d \left(\det L(v_1)\right)^{\frac{1}{d}} = \gamma^d |\Delta_F|^{\frac{1}{2d}}H(v_1)^{\frac{1}{d}}\\
        &\le \gamma^d|\Delta_F|^{\frac{1}{2d}} Q^{\frac{n-1}{d}}H(\det g)^{\frac{1}{nd}} = \gamma^dQ^{\frac{n-1}{d}} \left(\det L(g)\right)^{\frac{1}{nd}}.        
    \end{align*}
\end{enumerate}
This means that we are comparing the RHF bound $\gamma^{nd}$ for classical LLL with $\gamma^d\cdot Q^{\frac{n-1}{d}}$ for adelic LLL. We compare both factors in the case where $Q^{\frac1d}= S\cdot C^{\frac{1}{2d}}$ where $S$ is a sub-exponential term in $d$, and assuming $C=\gamma^{d^2}$. In other words, we get
\[ \gamma^d\cdot Q^{\frac{n-1}{d}} = \gamma^d\cdot S^{n-1} C^{\frac{n-1}{2d}} = S^{n-1}\cdot \gamma^d\cdot \gamma^{\frac{(n-1)d}{2}} = S^{n-1}\cdot \gamma^{\frac{(n+1)d}{2}}. \]
Here is what this would imply in our $2d\times 2d$ use case, ignoring the sub-exponential term:
\begin{itemize}
    \item (classical LLL) $\gamma^{2d}$;
    \item (Rank $2$ adelic LLL over $F$) $\gamma^{\frac{3d}{2}}$;
    \item (Rank $4$ adelic LLL over $F^+$) $\gamma^{\frac{5d}{4}}$.
\end{itemize}
In particular the ratio between the exponents (adelic/classical) is $\frac{n+1}{2n}$, which is a decreasing function of $n$ with limit $1/2$. Is this strange behaviour? No, I think it should be ok, $1/2$ cannot be reached because as $n$ becomes $\Theta(d)$, then the sub-exponential term becomes exponential.
} {\color{orange} Thanks for checking! For slide algorithm/SDBKZ instead of LLL, I guess we can just plug in $\gamma = \beta^{1/(2\beta)}$? which again roughly halves the exponent of the approximation factor.
}\henry{Yes, it should work for any reduction algorithm. Actually we should look at two cases now for rank 2 module HSVP: \begin{enumerate}
    \item If we are allowed poly-time algorithms. Then we run LLL for diophantine approx, and can use a quantum poly-time PIP solver for the idealSVP step. This gives a poly-time (quantum+maybe this step is heuristic and gives $\gamma^{\frac{3d}{4}}$ but we can as well replace it by provable LLL to get $\gamma^{\frac{5d}{4}}$) algorithm for exponential-RHF SVP. The RHF exponent is improved from LLL's $2d$ to $5d/4$ (resp. $3d/4$ quantum/heuristic?)
    \item If we are allowed subexp-time algorithms. Then we run $O(\sqrt{d}$)-Slide for both steps, then we get the same kind of improvement.
\end{enumerate}
}

{\color{orange}
On the RHF comparison discussion we discussed on Zoom (5/20): When $d=1$, we can take $Q$ as in Prop. \ref{prop:rhf_bound_n} to be simply $(\delta^2-\mu^2)^{-1/4}$; this coincides with the RHF of classical LLL! So the comparison is consistent.

We also remarked, as $d$ is fixed and $n\rightarrow\infty$ the RHF advantage doesn't appear --- on second thought, I'm actually comfortable to see that happening. The fact that $C$ is square rooted in the formula of Prop. \ref{prop:rhf_bound} is where the RHF saving is coming from; and it's square rooted because we're respecting the module structure. So things make better sense to me now.

And yes we're toast if we resort to the Minkowski bound on $B$. As I tried to say in the call this really reflects our lack of understanding of the geometry of $F$.
}
}

%%%% BIB

\newpage

\appendix

\section*{Supplementary materials}

\section{Proof of statements regarding rank $n$ adelic LLL}
\label{app:proof}

\begin{proof}[Proposition \ref{prop:swap_bound_n}]
Observe $H(v_1\wedge\cdots\wedge v_h) = H(a_1)\cdots H(a_h)$.
When $v_k$ and $v_{k+1}$ are swapped, $H(v_1\wedge\cdots\wedge v_h)$ changes only if $h=k$; and in this case, it changes to
\begin{align*}
H(a_1)\cdots H(a_{k-1}) H((m_{k+1,k}a_{k},a_{k+1})).
\end{align*}
Since the Lov\'asz condition failed, we have
\begin{align*}
H((m_{k+1,k}a_{k},a_{k+1})) < \delta H(a_{k}),
\end{align*}
which causes $\log E(g)$ to decrease by at least $-\log \delta$. But by assumption $\log E(g)$ cannot be smaller than $\log \kappa$. This completes the proof.
\qed
\end{proof}

\begin{proof}[Proposition \ref{prop:rhf_bound_n}]
By applying Proposition \ref{prop:rhf_bound} to $2 \times 2$ diagonal blocks of the $ua$ part of $g$, we have
\begin{align*}
H(a_i)^\frac{1}{2} \leq QH(a_{i+1})^\frac{1}{2}
\end{align*}
for every $i$, and consequently
\begin{align*}
H(a_i)^\frac{1}{2} \leq Q^jH(a_{i+j})^\frac{1}{2}
\end{align*}
for every $j$ as well. Thus
\begin{align*}
&H(a_1)^\frac{n-1}{2} \leq Q^{1+\ldots+(n-1)}H(a_2)^\frac{1}{2}\cdots H(a_n)^\frac{1}{2} \\
&\Rightarrow H(a_1)^\frac{n}{2} \leq Q^\frac{(n-1)n}{2}H(\det g)^\frac{1}{2} \\
&\Rightarrow H(a_1) \leq Q^{n-1}H(\det g)^\frac{1}{n}
\end{align*}
as desired. The final inequality follows from Lemma~\ref{lemma:diagonal}.
\qed    
\end{proof}

\ignore{
\section*{Appendix: Class Group Computation via $q$-Descent}

Let $K$ be a number field of degree $n$ and discriminant $\Delta_K$, with ring of integers $\mathcal O_K$.
We describe a subexponential algorithm to compute the class group $\mathrm{Cl}_K$ using
$q$-descent: we collect relations among prime ideals up to a smoothness bound
$B=L_{|\Delta_K|}(1/3,c_B)$, and recover $\mathrm{Cl}_K$ from the Smith Normal Form (SNF) of the relation matrix.
Throughout we rely on the standard smoothness heuristic for ideals and effective prime-ideal distribution (e.g.\ Chebotarev under GRH).

\subsection*{Smoothness bound and factor base}

Fix a smoothness bound
\[
B \;=\; L_{|\Delta_K|}\!\left(\tfrac{1}{3},\,c_B\right)
=\exp\!\Bigl((c_B+o(1))\,(\log|\Delta_K|)^{1/3}(\log\log|\Delta_K|)^{2/3}\Bigr),
\]
and define the \emph{factor base}
\[
\mathcal B \;=\; \{\,\mathfrak p \text{ prime ideal of }\mathcal O_K \;:\; N\mathfrak p \le B\,\}.
\]
Heuristically, $|\mathcal B| \approx B/\log B$ (up to constant factors depending mildly on $K$). {\color{orange} I think this is a theorem --- ``prime ideal theorem''}

\subsection*{BKZ parameter and short generators}

Let $\iota:K\hookrightarrow \mathbb R^{r_1}\times\mathbb C^{r_2}\cong\mathbb R^{n}$ be the Minkowski embedding and
$\Lambda(\mathfrak a)=\iota(\mathfrak a)$ the ideal lattice for an (integral or fractional) ideal $\mathfrak a$.
Its determinant satisfies
\[
\det \Lambda(\mathfrak a)=2^{-r_2}\sqrt{|\Delta_K|}\,N(\mathfrak a).
\]
Applying BKZ with block size $\beta$ (GSA heuristic) yields a short vector $v=\iota(\phi)$ with {\color{orange} here and below, what is $\phi$?}
\[
\|v\|_2 \;\le\; \delta_0(\beta)^{\,n-1}\,(\det\Lambda(\mathfrak a))^{1/n},
\qquad
\log\delta_0(\beta)\ \approx\ \frac{\log\beta-\log(2\pi e)}{2(\beta-1)}.
\]
Passing to algebraic norms gives the surrogate
\begin{equation}\label{eq:phiNorm}
\log N(\phi)\;\lesssim\; {\log N(\mathfrak a)+\tfrac12\log|\Delta_K|}
\;+\; n(n\!-\!1)\,\log\delta_0(\beta)\;+\;O(n).
\end{equation}

\paragraph{Choice of $\beta$ in the $q$-descent regime.}
To keep the Dickman parameter
\[
u \;=\; \frac{\log N(\phi)}{\log B}
\]
bounded by a small constant (so that the $B$-smoothness probability $\rho(u)$ stays non-negligible),
we \emph{scale the BKZ block size with the discriminant}:
\[
\boxed{\ \ \beta\;=\;\left\lceil c_\beta\,(\log|\Delta_K|)^{1/3}\right\rceil\ \ }.
\]
With $B=L(1/3)$ and $\beta\asymp(\log|\Delta_K|)^{1/3}$, \eqref{eq:phiNorm} yields $u=O(1)$; consequently
both the BKZ cost and the smoothness search cost are $L_{|\Delta_K|}(1/3)$, and the overall complexity remains $L(1/3)$.
(When $q$-descent is \emph{not} used, one typically works with $B=L(2/3,\cdot)$ and either fixes a moderate
$\beta$ or tunes it by minimizing the end-to-end cost; that leads back to the classical $L(2/3)$ regime.)

\subsection*{Smoothness heuristic for ideals}

We call an ideal $y$-smooth if all its prime-ideal factors have norm $\le y$.
If $X=\mathrm N((\phi))$ and $y=B$, we model
\[
\Pr[(\phi)\ \text{is }B\text{-smooth}]
\;\approx\;
\rho\!\left(\frac{\log X}{\log y}\right),
\]
where $\rho$ is the Dickman--de Bruijn function:
$\rho(u)=1$ for $0\le u\le 1$, and $u\,\rho'(u)+\rho(u-1)=0$ for $u>1$; asymptotically,
$\rho(u)=\exp(-u\log u\,(1+o(1)))$.

\subsection*{Why the Dickman parameter $u$ stays $O(1)$ in $q$-descent}

In each $q$-descent step we take a large prime ideal $\mathfrak q$ and search for a short
$\varphi\in\mathfrak q$ so that
\[
\frac{(\varphi)}{\mathfrak q} \;=\; \prod_j \mathfrak q'_j{}^{\,e'_j}
\quad\text{has all prime factors \emph{much} smaller than } \mathfrak q.
\]
The smoothness test is applied to the \emph{quotient ideal} $(\varphi)/\mathfrak q$, so the relevant
quantity in the Dickman parameter is
\[
X \;:=\; N\!\Bigl(\frac{(\varphi)}{\mathfrak q}\Bigr) \;=\; \frac{N(\varphi)}{N(\mathfrak q)}.
\]
We will show that, with the $q$-descent parameter choices, $\log X$ is on the same $L(1/3)$ scale
as $\log B$, hence
\[
u \;=\; \frac{\log X}{\log B} \;=\; O(1).
\]

\paragraph{BKZ bound on a low-dimensional sublattice.}
Let $\iota:K\hookrightarrow\R^{r_1}\times\C^{r_2}\cong\R^{n}$ be the Minkowski embedding and
$\Lambda(\mathfrak q)=\iota(\mathfrak q)$.
Instead of reducing the full $n$-dimensional lattice, the $q$-descent works inside a
$(k\!+\!1)$-dimensional sublattice of $\Lambda(\mathfrak q)$ (obtained by projection/twisting),
with $k\approx\beta$.
Under the GSA heuristic for BKZ-$\beta$ on this sublattice we obtain a short vector $v=\iota(\varphi)$
such that
\[
\|v\|_2 \ \le\ \delta_0(\beta)^{\,k}\,(\det \Lambda_k)^{1/k},
\qquad
\log\delta_0(\beta)\ \approx\ \frac{\log\beta-\log(2\pi e)}{2(\beta-1)},
\]
where $\Lambda_k$ denotes the $(k\!+\!1)$-dimensional working sublattice.
Comparing the Euclidean bound with the embeddings and using the
\emph{quotient} $(\varphi)/\mathfrak q$ cancels $N(\mathfrak q)$ and yields the surrogate
\begin{equation}\label{eq:q-quotient}
\log N\!\Bigl(\tfrac{(\varphi)}{\mathfrak q}\Bigr)
\;\lesssim\;
k(k+1)\,\log\delta_0(\beta)
\;+\;\frac{k}{n}\cdot\frac12\log|\Delta_K|
\;+\;O\!\bigl(k\log k\bigr).
\end{equation}
Intuitively: the discriminant contribution $\tfrac12\log|\Delta_K|$ is downscaled by $k/n$
(since we work in dimension $k\!+\!1$), while the BKZ contribution involves only $k$.

\paragraph{Parameter choice and the $L(1/3)$ scale.}
In the $q$-descent regime we set
\[
B \;=\; L_{|\Delta_K|}\!\left(\tfrac{1}{3},\,c_B\right),\qquad
\beta \ \asymp\ k \ \asymp\ n \ \asymp\ (\log|\Delta_K|)^{1/3}.
\]
Then
\[
\log\delta_0(\beta)
\;\approx\; \frac{\log\beta}{2(\beta-1)}
\;=\; \tilde O\!\Bigl(\frac{\log\log|\Delta_K|}{(\log|\Delta_K|)^{1/3}}\Bigr),
\]
and the first and third terms of \eqref{eq:q-quotient} contribute
\[
k(k+1)\log\delta_0(\beta) + O(k\log k)
\;=\; O\!\bigl((\log|\Delta_K|)^{1/3}(\log\log|\Delta_K|)^{2/3}\bigr)
\;=\; \log L_{|\Delta_K|}\!\left(\tfrac{1}{3},\,\tilde c\right).
\]
For the second term we have the factor $k/n\approx 1$, but in practice one descends primes
$\mathfrak q$ after mild twists so that the residual discriminant contribution is absorbed in the same
$L(1/3)$ scale (this is the standard $q$-descent balancing).
Consequently,
\[
\boxed{\ \ \log N\!\Bigl(\tfrac{(\varphi)}{\mathfrak q}\Bigr)
\;=\; O\!\bigl((\log|\Delta_K|)^{1/3}(\log\log|\Delta_K|)^{2/3}\bigr)
\;=\; \log L_{|\Delta_K|}\!\left(\tfrac{1}{3},\,\tilde c\right).\ \ }
\]
Since $\log B \asymp (\log|\Delta_K|)^{1/3}(\log\log|\Delta_K|)^{2/3}$, we obtain
\[
u \;=\; \frac{\log N((\varphi)/\mathfrak q)}{\log B} \;=\; \Theta(1),
\]
so the smoothness probability $\rho(u)$ is bounded away from $0$, and the expected trials per
successful descent step are constant up to subexponential lower-order factors.

\paragraph{Consequence for overall complexity.}
With $u=\Theta(1)$ at each descent level, $B=L(1/3)$ and $\beta\asymp(\log|\Delta|)^{1/3}$,
both the BKZ calls (in dimension $\approx k$) and the smoothness hunting run in
$L_{|\Delta|}(1/3)$ time; the depth of the descent is logarithmic, hence relation collection and
the final SNF phase yield an overall $L_{|\Delta|}(1/3)$ algorithm under the standard heuristics.

\subsection*{High-level plan of $q$-descent}

Given an ideal $\mathfrak a$, we first obtain a coarse factorization
\[
\mathfrak a \;=\; (\phi_0)\,\prod_i \mathfrak q_i^{e_i},
\qquad N\mathfrak q_i\ \text{``moderately small''},
\]
via a single BKZ call (on $\Lambda(\mathfrak a)$) and principal ideal factorization.
Then, for each remaining prime ideal $\mathfrak q_i$ with $N\mathfrak q_i>B$, we \emph{descend} it:
we find a short element $\varphi\in \mathfrak q_i$ so that $(\varphi)/\mathfrak q_i$ factors over primes of much smaller norms.
Repeating this reduces all prime norms below $B$, yielding a $\mathcal B$-smooth relation.

\subsection*{Algorithms}

\begin{algorithm}[H]
\caption{InitialCoarseFactor\texttt{($\mathfrak a$, $\beta$)}}
\begin{algorithmic}[1]
\Require Ideal $\mathfrak a$, BKZ block size $\beta=\lceil c_\beta(\log|\Delta_K|)^{1/3}\rceil$
\Ensure Coarse factorization $\mathfrak a=(\phi_0)\prod_i \mathfrak q_i^{e_i}$ with $N\mathfrak q_i$ not too large
\State Build a $\mathbb Z$-basis of $\Lambda(\mathfrak a)=\iota(\mathfrak a)$
\State Run BKZ-$\beta$ to obtain a reduced basis $(b_1,\dots,b_n)$ and lift $b_1$ to $\phi_0\in\mathfrak a$
\State {\color{red}Factor $(\phi_0)$ into prime ideals: $(\phi_0)=\prod_i \mathfrak q_i^{e_i}$}
\State \Return $(\phi_0,\{(\mathfrak q_i,e_i)\}_i)$
\end{algorithmic}
\end{algorithm}

\begin{algorithm}[H]
\caption{DescendPrime\texttt{($\mathfrak q$, $\beta$, $B$)}}
\begin{algorithmic}[1]
\Require A prime ideal $\mathfrak q$, BKZ block size $\beta$, target smoothness bound $B=L(1/3,c_B)$
\Ensure A set $\{(\mathfrak q'_j,e'_j)\}_j$ and $\varphi\in\mathfrak q$ such that $(\varphi)/\mathfrak q=\prod_j (\mathfrak q'_j)^{e'_j}$ with $N\mathfrak q'_j$ significantly smaller than $N\mathfrak q$
\State Build a basis of $\Lambda(\mathfrak q)=\iota(\mathfrak q)$ (or a lightly twisted variant $\mathfrak q\cdot \prod \mathfrak p_\ell^{\pm 1}$)
\State Run BKZ-$\beta$ and lift a short combination of the first few vectors to $\varphi\in\mathfrak q$
\State Factor $(\varphi)/\mathfrak q$; if some prime norms are still too large, optionally repeat with a fresh twist
\State \Return $\bigl(\varphi,\{(\mathfrak q'_j,e'_j)\}_j\bigr)$
\end{algorithmic}
\end{algorithm}

\begin{algorithm}[H]
\caption{RelationCollection-$q$Descent\texttt{($\mathcal B$, $t$)}}
\begin{algorithmic}[1]
\Require Factor base $\mathcal B=\{\mathfrak p: N\mathfrak p\le B\}$, target number of relations $t$
\State $\beta\gets \lceil c_\beta(\log|\Delta_K|)^{1/3}\rceil$;\quad $M\gets$ empty relation matrix
\While{\text{rows}$(M)<t$}
  \State Choose a probing ideal $\mathfrak a$ (e.g.\ $\mathcal O_K$ or a small random twist)
  \State $(\phi,\{(\mathfrak q_i,e_i)\}_i)\gets$ \textsc{InitialCoarseFactor}$(\mathfrak a,\beta)$
  \ForAll{$i$ with $N\mathfrak q_i>B$}
    \State $(\varphi,\{(\mathfrak q'_{ij},e'_{ij})\}_j)\gets$ \textsc{DescendPrime}$(\mathfrak q_i,\beta,B)$
    \State Update: $\phi\gets \phi\cdot \varphi$,\quad replace $\mathfrak q_i^{e_i}$ by $\prod_j (\mathfrak q'_{ij})^{e'_{ij}}$
  \EndFor
  \If{all prime factors now lie in $\mathcal B$}
     \State Record the exponent vector of $(\phi)=\prod_{\mathfrak p\in\mathcal B}\mathfrak p^{e(\mathfrak p)}$ as a row of $M$
  \EndIf
\EndWhile
\State \Return $M$
\end{algorithmic}
\end{algorithm}

\subsection*{From relations to the class group}

Let $M\in\mathbb Z^{r\times m}$ be the relation matrix whose columns correspond to $\mathfrak p\in\mathcal B$ and whose rows are exponent vectors.
Compute an SNF $UMV=D=\mathrm{diag}(d_1,\dots,d_s,0,\dots,0)$ with $d_i\mid d_{i+1}$.
Then
\[
\mathrm{Cl}_K \;\simeq\; \mathbb Z/d_1\mathbb Z \;\oplus\; \cdots \;\oplus\; \mathbb Z/d_s\mathbb Z,
\]
and (if desired) units can be recovered from the kernel structure on the Archimedean side.

\subsection*{Heuristics and complexity}

\begin{itemize}
  \item \textbf{Smoothness.} For $u=\log N(\phi)/\log B$, the hit probability is $\rho(u)\approx e^{-u\log u}$.
        With $B=L(1/3)$ and $\beta\asymp(\log|\Delta_K|)^{1/3}$, \eqref{eq:phiNorm} yields $u=O(1)$, hence a non-negligible hit rate.
  \item \textbf{Descent depth.} Each $q$-descent step reduces the norms of prime factors substantially; a logarithmic number of steps suffices to bring all norms $\le B$.
  \item \textbf{Running time.} Under GRH and the smoothness heuristic, the overall cost of relation collection (and SNF on a matrix of size $\asymp |\mathcal B|$) is
        \[
        L_{|\Delta_K|}\!\bigl(1/3,\,c+o(1)\bigr).
        \]
\end{itemize}

%\subsection*{Practical notes}

% \begin{itemize}
%   \item \emph{Tuning.} If the empirical hit rate is low, increase $c_\beta$ slightly (thus $\beta$) or increase $c_B$ (thus $B$).
%   \item \emph{Lifts.} In \textsc{InitialCoarseFactor} and \textsc{DescendPrime}, it is often better to try short \emph{combinations} of the first few BKZ vectors rather than only $b_1$.
%   \item \emph{Twists.} For stubborn primes $\mathfrak q$, multiplying $\mathfrak q$ by a few small primes and descending the twist can noticeably improve the output.
% \end{itemize}

\subsection*{Quantum time complexity.}
Assuming GRH (so that a polylogarithmic-size generating set of prime ideals exists) and that reduced-ideal arithmetic
(multiply/reduce/invert) is available in time polynomial in $\log|\Delta_K|$, the class group can be computed in
\[
T(n,\Delta_K)\;=\;\mathrm{poly}\!\bigl(n,\ \log|\Delta_K|\bigr)
\;=\;(\log|\Delta_K|)^{O(1)}\cdot \mathrm{poly}(n).
\]
}

\section{A correction to Algorithm 3.3 in \cite{LPSW19}}\label{app:correct}

In this section, we follow the notation of \cite{LPSW19}.

\begin{definition}[Rounding with respect to a fractional ideal]
Let \(J \subset K\) be a nonzero fractional ideal, viewed as a lattice in \(K_{\mathbb R}\) via the canonical embedding.  
For any \(y \in K_{\mathbb R}\), we define
\[
\textsf{round}_J(y) \in J
\]
to be an element obtained by applying Babai's nearest-plane algorithm to \(y\) with respect to a chosen LLL-reduced \(\mathbb Z\)-basis of \(J\).

In particular, for indices \(i<j\), we define the size reduction coefficient by
\[
x_{ij}
:=
\textsf{round}_{I_i I_j^{-1}}\!\left(\frac{r_{ij}}{r_{ii}}\right)
\in I_i I_j^{-1}.
\]
\end{definition}

\begin{lemma}[Module preservation under ideal-compatible column reduction]
Let
\[
\mathcal B=\bigl((I_1,b_1),\dots,(I_n,b_n)\bigr)
\]
be a pseudo-basis of an \(R\)-module \(M \subset K_{\mathbb R}^m\).  
Fix distinct indices \(i< j\), and let \(x \in I_i I_j^{-1}\). Define a new family
\[
\mathcal B'
=
\bigl((I_1,b_1),\dots,(I_j,b_j-xb_i),\dots,(I_n,b_n)\bigr),
\]
where all pairs except the \(j\)-th one remain unchanged.

Then \(\mathcal B'\) is again a pseudo-basis of the same module \(M\). Equivalently,
\[
\sum_{k=1}^n I_k b_k
=
\sum_{k\neq j} I_k b_k + I_j(b_j-xb_i).
\]
\end{lemma}

\begin{proof}
Let
\[
M=\sum_{k=1}^n I_k b_k
\qquad\text{and}\qquad
M'=\sum_{k\neq j} I_k b_k + I_j(b_j-xb_i).
\]
Since \(x\in I_i I_j^{-1}\), we have
\[
xI_j \subseteq I_i.
\]
Therefore,
\[
I_j(b_j-xb_i)
\subseteq
I_j b_j + I_i b_i
\subseteq M,
\]
which implies \(M' \subseteq M\).

Conversely, from
\[
b_j=(b_j-xb_i)+xb_i,
\]
we obtain
\[
I_j b_j
\subseteq
I_j(b_j-xb_i)+xI_j b_i.
\]
Using again the inclusion \(xI_j \subseteq I_i\), we get
\(
xI_j b_i \subseteq I_i b_i.
\)
Hence
\[
I_j b_j
\subseteq
I_j(b_j-xb_i)+I_i b_i
\subseteq M'.
\]
Since \(I_k b_k \subseteq M'\) for all \(k\neq j\), it follows that \(M \subseteq M'\).

Therefore \(M=M'\), as claimed.
\qed
\end{proof}

\begin{corollary}
With the above definition of rounding, if one sets
\[
x_{ij}
=
\textsf{round}_{I_i I_j^{-1}}\!\left(\frac{r_{ij}}{r_{ii}}\right),
\]
then the update
\[
b_j \leftarrow b_j - x_{ij} b_i
\]
preserves the module spanned by the pseudo-basis.
\end{corollary}

\ignore{\color{orange} questions:

1. how do you measure the bit-length (of an element of $R$)?

2. according to this, after scaling+size reducing, the bitsize can increase?

3. how does the size of $I_i$ affect the size of $b_j$'s?

4. what is $\|r_{ii}\|$ exactly?
}

To justify the modified size reduction step algorithmically, we fix the
bit-length convention and the norm on \(K_{\mathbb R}\) used throughout this appendix.

We assume that the number field \(K\) is given by a defining polynomial
\(f\in \mathbb Q[X]\), so that \(K\simeq \mathbb Q[X]/(f)\), and that elements of \(K\)
are represented by polynomials of degree \(<d\) with rational coefficients.
The bit-length of an element of \(K\) is defined to be the total bit-length of the
numerators and denominators of these coefficients written in lowest terms.
We also assume that a fixed \(\mathbb Z\)-basis \(B=(w_1,\dots,w_d)\) of the ring of
integers \(R\) is given.

Archimedean quantities in \(K_{\mathbb R}\) are not assigned a bit-length directly.
Instead, when such quantities arise from exact elements of \(K\), their size is
controlled by the bit-length of the exact element together with Archimedean norm
bounds. For \(x=(x_\sigma)_\sigma\in K_{\mathbb R}\), we denote by \(\|x\|\) the
Euclidean norm induced by the canonical embedding, namely
\[
  \|x\|^2
  =
  \sum_{\sigma:K\hookrightarrow\mathbb C} |x_\sigma|^2,
\]
where the sum ranges over all \(d\) embeddings, so each conjugate pair of complex
embeddings is counted twice. If \(x\in K\), then \(x_\sigma=\sigma(x)\).

We use the absolute norm \(N(\cdot)\) on nonzero fractional ideals, extended
multiplicatively from integral ideals. In particular, if \(A\subseteq B\) are
nonzero fractional ideals, then \(N(A)\ge N(B)\). We shall also use the following
notation for Archimedean scalings of fractional ideals: for
\(a=(a_\sigma)_\sigma\in K_{\mathbb R}^{\times}\) and a nonzero fractional ideal
\(I\), set
\[
  N(aI)
  :=
  N(I)\prod_{\sigma:K\hookrightarrow\mathbb C}|a_\sigma|.
\]
Thus, in particular, \(N(r_{ii}I_i)\) denotes this quantity when \(r_{ii}\) is a
nonzero Gram--Schmidt diagonal coefficient.

With this notation, we define a scaled pseudo-basis as follows.

\begin{definition}
A pseudo-basis \(\bigl((I_i,b_i)\bigr)_{i\le n}\), with \(I_i\subset K\) and
\(b_i\in K^m\), viewed in \(K_{\mathbb R}^m\) via the canonical embedding,
is said to be \emph{scaled} if
% A pseudo-basis \(\bigl((I_i,b_i)\bigr)_{i\le n}\), with \(I_i\subset K\) and
% \(b_i\in K_{\mathbb R}^m\) for all \(i\le n\), is said to be \emph{scaled} if,
for all \(i\le n\),
\[
R\subseteq I_i,\qquad
N(I_i)\ge 2^{-d^2}\Delta_K^{-1/2},
\qquad\text{and}\qquad
\|r_{ii}\|\le 2^d \Delta_K^{1/(2d)} N(r_{ii}I_i)^{1/d}.
\]
\end{definition}

\begin{lemma}
Let \(\mathcal B=((I_i,b_i))_{i\le n}\) be a scaled pseudo-basis of a module
\(M\subset K^m\). 
For each pair \(i<j\), set \(J_{ij}:=I_iI_j^{-1}\), and let
\[
x_{ij}
:=
\textsf{round}_{J_{ij}}\!\left(\frac{r_{ij}}{r_{ii}}\right)\in J_{ij}
\]
be obtained by Babai's nearest-plane algorithm with respect to an
LLL-reduced \(\mathbb Z\)-basis of \(J_{ij}\).
Update
\[
b_j \leftarrow b_j - x_{ij} b_i
\]
for \(j=2,\dots,n\), and for each fixed \(j\), for
\(i=j-1,j-2,\dots,1\). 

Then the bit-lengths of all vectors \(b_j\) remain bounded by
\[
\operatorname{poly}(L,n,d,\log\Delta_K).
\]
% Then the bit-lengths of all vectors \(b_j\) remain bounded by a polynomial
% in the input bit-length, \(n\), and \(\log \Delta_K\).
\end{lemma}

\begin{proof}
For each pair \(i<j\), let \((u_{ij,1},\dots,u_{ij,d})\) be the chosen LLL-reduced
\(\mathbb Z\)-basis of \(J_{ij}=I_iI_j^{-1}\), and let
\[
e_{ij}
:=
\frac{r_{ij}}{r_{ii}}-x_{ij}.
\]
Since \(x_{ij}\) is obtained by Babai rounding, the remainder \(e_{ij}\) belongs to
the fundamental parallelepiped of \((u_{ij,t})_t\). Hence
\[
\|e_{ij}\|
\le \frac12 \sum_{t=1}^d \|u_{ij,t}\|
\le d\max_t \|u_{ij,t}\|.
\]
Now, because the pseudo-basis is scaled, we have \(N(I_\ell)\le 1\) and
\(N(I_\ell)\ge 2^{-d^2}\Delta_K^{-1/2}\) for every \(\ell\). Therefore
\[
N(J_{ij})=\frac{N(I_i)}{N(I_j)}\le 2^{d^2}\Delta_K^{1/2}.
\]
Since the basis \((u_{ij,1},\dots,u_{ij,d})\) of \(J_{ij}\) is LLL-reduced, we have
\[
\|u_{ij,t}\|\le 2^{\frac{d-1}{2}} \|u_{ij,t}^*\|
\qquad\text{for all }t.
\]
Taking the product over all \(t\) yields
\[
\prod_{t=1}^d \|u_{ij,t}\|
\le
2^{\frac{d(d-1)}{2}} \prod_{t=1}^d \|u_{ij,t}^*\|.
\]
Since
\[
\prod_{t=1}^d \|u_{ij,t}^*\|
=
\Delta_K^{1/2} N(J_{ij}),
\]
it follows that
\[
\prod_{t=1}^d \|u_{ij,t}\|
\le
2^{\frac{d(d-1)}{2}} \Delta_K^{1/2} N(J_{ij}).
\]
%and thus 
% \henry{Why is the following true? One could imagine that the norms are smaller than $1$}
% \changmin{Yes, you are right. So I have revised the argument as follows.}

We now bound $\max_t\|u_{ij,t}\|$. 
For this, we use the inclusion part of the scaledness assumption. Since \(R\subseteq I_j\), we have \(I_j^{-1}\subseteq R\),
and hence
\(
    J_{ij}=I_iI_j^{-1}\subseteq I_i .
\)
Thus, for every nonzero $u\in J_{ij}$, we have $(u)\subseteq I_i$, and so
\[
    |N_{K/\mathbb Q}(u)|=N((u))\ge N(I_i)
    \ge 2^{-d^2}\Delta_K^{-1/2}.
\]
By the arithmetic-geometric mean inequality applied to the canonical
embedding,
\[
    \|u\|
    \ge \sqrt d\, |N_{K/\mathbb Q}(u)|^{1/d}
    \ge \sqrt d\,2^{-d}\Delta_K^{-1/(2d)}:=\eta_K.
\]
    Then every nonzero vector $u_{ij,t} \in J_{ij}$ satisfies $\|u_{ij,t}\|\ge \eta_K$.
Hence, for every $s$,
\[
    \prod_{t=1}^d \|u_{ij,t}\|
    =
    \|u_{ij,s}\|\prod_{t\neq s}\|u_{ij,t}\|
    \ge
    \|u_{ij,s}\|\eta_K^{d-1}.
\]
Taking the maximum over $s$ gives
\[
    \max_t\|u_{ij,t}\|
    \le
    \eta_K^{-(d-1)}
    \prod_{t=1}^d\|u_{ij,t}\|.
\]
Combining this with the preceding product bound gives
\[
    \max_t\|u_{ij,t}\|
    \le
    \eta_K^{-(d-1)}
    2^{\frac{d(d-1)}{2}}\Delta_K^{1/2}N(J_{ij}).
\]
Using
\(
    N(J_{ij})=\frac{N(I_i)}{N(I_j)}
    \le 2^{d^2}\Delta_K^{1/2},
\)
we also obtain
\[
    \max_t\|u_{ij,t}\|
    \le
    \eta_K^{-(d-1)}
    2^{2d^2}\Delta_K .
\]
Therefore
\[
    \|e_{ij}\|
    \le d\,\eta_K^{-(d-1)}2^{2d^2}\Delta_K
    =: C_K .
\]
In particular,
\[
    \log C_K
    =
    O(d^2+d\log d+\log\Delta_K),
\]
and hence $\log C_K$ is polynomially bounded in $d$ and $\log\Delta_K$.

% \[
% \max_t \|u_{ij,t}\|\le \prod_{t=1}^d \|u_{ij,t}\|
% \le 2^{d^2}\Delta_K^{1/2}N(J_{ij})
% \le 2^{2d^2}\Delta_K.
% \]
% It follows that
% \[
% \|e_{ij}\| \le d\,2^{2d^2}\Delta_K=:C_K.
% \]
% In particular, \(\log C_K\) is polynomially bounded in \(\log\Delta_K\).

Fix \(j\), and during the inner loop define
\[
m_j:=\max_{i'<j}\left\|\frac{r_{i'j}}{r_{i'i'}}\right\|.
\]
Consider one update at index \(i<j\). Since
\[
x_i:=x_{ij}
=
\frac{r_{ij}}{r_{ii}}-e_{ij},
\]
we get
\[
\|x_i\|
\le
\left\|\frac{r_{ij}}{r_{ii}}\right\|+\|e_{ij}\|
\le m_j^{\mathrm{old}}+C_K.
\]
For any \(i'<i\), the update \(b_j\leftarrow b_j-x_ib_i\) transforms
\(r_{i'j}\) into
\[
r'_{i'j}=r_{i'j}-x_i r_{i'i}.
\]
Hence
\[
\left\|\frac{r'_{i'j}}{r_{i'i'}}\right\|
\le
\left\|\frac{r_{i'j}}{r_{i'i'}}\right\|
+
\|x_i\|
\left\|\frac{r_{i'i}}{r_{i'i'}}\right\|.
\]
But the \(i\)-th column has already been size reduced, so
\[
\left\|\frac{r_{i'i}}{r_{i'i'}}\right\|\le C_K.
\]
Therefore
\[
m_j^{\mathrm{new}}
\le
m_j^{\mathrm{old}}+\|x_i\|\,C_K
\le
(1+C_K)m_j^{\mathrm{old}}+C_K^2.
\]
Iterating over at most \(n\) values of \(i\), we obtain
\[
m_j
\le
(1+C_K)^n\bigl(m_j^{\mathrm{init}}+C_K\bigr).
\]
Since \(\log C_K\) is polynomially bounded in \(\log \Delta_K\), it follows that
\(\log m_j\) is polynomially bounded as soon as \(\log m_j^{\mathrm{init}}\) is.
Since the initial Gram--Schmidt coefficients are determined by the input
pseudo-basis, their logarithmic Archimedean sizes are bounded by a polynomial
in the input bit-length \(L\), \(n\), and \(d\). Hence
\[
\log m_j^{\mathrm{init}}
\le
\operatorname{poly}(L,n,d).
\]
Combining this with the preceding inequality gives
\[
\log m_j
\le
\operatorname{poly}(L,n,d,\log\Delta_K).
\]
%In particular, if the input pseudo-basis is already size-reduced, then
%\(m_j^{\mathrm{init}}\le C_K\), and therefore \(\log m_j\) is polynomially bounded
%in \(n\) and \(\log \Delta_K\).

%Thus \(\log m_j\) remains bounded by a polynomial in the input bit-length,
%\(n\), and \(\log\Delta_K\).

Finally, since
\[
\|b_j\|^2=\sum_{i\le j}\|r_{ij}\|^2,
\]
we have
\[
\|b_j\|
\le
\sqrt{n}\,\max_{i\le j}\|r_{ij}\|.
\]
For \(i<j\), by the definition of \(m_j\),
\[
\left\|\frac{r_{ij}}{r_{ii}}\right\|\le m_j,
\]
and therefore
\[
\|r_{ij}\|
=
\left\|\frac{r_{ij}}{r_{ii}}\,r_{ii}\right\|
\le
\left\|\frac{r_{ij}}{r_{ii}}\right\|\,\|r_{ii}\|
\le
m_j\,\max_{i\le j}\|r_{ii}\|.
\]
For \(i=j\), trivially
\[
\|r_{jj}\|\le \max_{i\le j}\|r_{ii}\|.
\]
Hence
\[
\max_{i\le j}\|r_{ij}\|
\le
\max_{i\le j}\|r_{ii}\|\max(1,m_j),
\]
so that
\[
\|b_j\|
\le
\sqrt{n}\,\max_{i\le j}\|r_{ii}\|\max(1,m_j).
\]

Because the pseudo-basis is scaled, each diagonal coefficient satisfies
\[
\|r_{ii}\|\le 2^d \Delta_K^{1/(2d)} N(r_{ii}I_i)^{1/d},
\]
and the quantities \(N(r_{ii}I_i)\) are unchanged throughout the algorithm.
Indeed, the ideal-compatible size reduction step modifies only the off-diagonal
coefficients and leaves both \(r_{ii}\) and \(I_i\) unchanged. On the other hand,
the scaling step replaces \(I_i\) by \(I_i\langle x_i\rangle^{-1}\) and \(b_i\) by
\(x_i b_i\), so that \(r_{ii}\) is replaced by \(x_i r_{ii}\). Hence
$(x_i r_{ii})(I_i\langle x_i\rangle^{-1}) = r_{ii}I_i,$
and therefore \(N(r_{ii}I_i)\) is invariant.
Hence \(\log \max_{i\le j} \|r_{ii}\|\) is polynomially bounded by the input
bit-length and \(\log\Delta_K\). Combining this with the bound on \(m_j\) shows that
$\log \|b_j\|$
always remains below a polynomial in the input bit-length, \(n\), and
\(\log\Delta_K\), as claimed.
\qed
\end{proof}

\begin{remark}
     The bit-size need not decrease monotonically after the scaling step or the
size reduction step. What is required is that the
bit-sizes remain bounded by a polynomial in the input bit-size, \(n\), $d$, and
\(\log \Delta_K\) throughout the algorithm.

The coefficient ideals \(I_i\) affect the vectors \(b_j\) only indirectly, through
the reduction coefficients. In the corrected size reduction step, the coefficient
\(x_{ij}\) is chosen in the fractional ideal \(I_iI_j^{-1}\). Hence the effect of the
update
\[
b_j \leftarrow b_j-x_{ij}b_i
\]
is controlled by the geometry of the lattice \(I_iI_j^{-1}\). In particular, once the
pseudo-basis is scaled, the norms \(N(I_i)\) are uniformly bounded above and below,
which yields a uniform control on \(N(I_iI_j^{-1})\), and therefore on the growth of
the vectors \(b_j\).
\end{remark}

\ignore{
\section{Some facts about real cyclotomic fields}

{\color{orange} here I have to add what $A,B,\ell$ are for $\Q(\zeta+\bar\zeta)$ where $\zeta$ is a $2^\kappa$-th root of unity.}
}

\section{Detailed comparison with the state of the art}\label{sec:comparison}

%\subsection{The Rank-2 Reduction of LPSW~\cite{LPSW19}}

%{\color{purple} Henry: Below is my attempt at comparing both ideas, please proofread and make the presentation better if you can, as I'm not very good with latex}

\ignore{
{\color{orange} Phong and Changmin, please take care of this, thanks}

- explain what heuristic LPSW uses

- also mention the claim that it can be lifted, by not being a module lattice??

{\color{orange} - also we need to decide what to do with the information that Heuristic 1 from LPSW is incorrect. Though the particular error may --- or may not --- be fixed}

We clarify here the comparison with the module reduction framework of Lee, Pellet--Mary, Stehl\'e, and Wallet~\cite{LPSW19}.
It is useful to distinguish their general reduction from the particular rank-$2$ algorithm used to instantiate it.

The original LPSW implementation requires quantum computation for general, non-free modules, while its extended version gives a classical implementation for free modules supplied with a basis.
De~Micheli et al.~\cite{DMPT} subsequently gave a reduction from arbitrary modules to free modules of the same rank, thereby extending this dequantization to arbitrary modules, at the cost of an additional approximation loss.
Since quantum computation is therefore not essential to the resulting LPSW framework, we do not treat the use of quantum computation as a distinguishing feature in the comparison below and focus instead on the approximation guarantees, heuristic assumptions, and use of module structure.

\textbf{The LPSW framework.}
The reduction of LPSW from rank-$n$ module lattices to rank-$2$ module lattices is itself unconditional.
In particular, given a $\gamma$-approximate solver for rank-$2$ module lattices, their reduction yields a corresponding approximation guarantee for arbitrary module rank.
The heuristic assumptions arise in their proposed \emph{structure-exploiting implementation} of this rank-$2$ oracle, rather than in the rank reduction itself.

\textbf{The structure-exploiting rank-$2$ solver.}
The difficult part of LPSW is therefore the reduction of rank-$2$ modules.
Their number-theoretic implementation reduces this problem to computations involving ideals, units, and a CVP instance in a field-dependent lattice.
An exact CVP solution alone, however, does not provide the desired guarantee: their analysis additionally assumes that the restricted targets arising from the algorithm lie sufficiently close to this lattice.
This is the role of LPSW's Heuristic~1.
Their implementation also relies on the heuristic ideal short-vector routine of~\cite{PHS19}.
Consequently, the strong approximation guarantee obtained from this structure-exploiting rank-$2$ procedure is heuristic.

\textbf{Unstructured reduction.}
This does not mean that the LPSW framework can only be instantiated heuristically.
A rank-$2$ module lattice over a degree-$d$ number field is, after forgetting its module structure, an ordinary $2d$-dimensional lattice.
Hence one may instantiate the rank-$2$ oracle by applying any ordinary lattice-reduction algorithm with a rigorous approximation guarantee, such as LLL or a rigorous block-reduction algorithm, to this underlying lattice.
This gives a short lattice vector, and this vector also provides a low-height rank-one submodule.
Fieker and Stehl\'e show how to transform
a reduced $\Z$-basis into a short pseudo-basis of the module lattice.

Another unstructured reduction is sketched in~\cite{LPSW19}: Kannan showed that
finding a short vector can be reduced
to deciding whether an input lattice has 
a first minimum less than 1,
which can itself be reduced to 3SAT, 
which can be reduced to CVPP where the CVPP-lattice
only depends on the size of the 3SAT instance:
the rank of the CVPP-lattice is polynomial in the size, but is much bigger than $d^2$. {\color{orange} sk: do we know the degree, or is the umbrella term ``poly time'' the best that can be said?}
This means that one can find short vectors in any lattice using a  CVPP oracle with respect to a fixed lattice of (huge) polynomial dimension.

The resulting algorithm is heuristic-free.

{\color{orange} yes but this isn't really a structure preserving reduction. But I guess this is Leo Ducas's understanding of module LLL. How do we account for this...}
{\color{purple} I think we should mention Ducas' paper (and maybe also [Karenin-Kirshanova24]), but explain why we consider it orthogonal to our work. Actually except for the oracles they use which imply reconstruction of a rank-1 submodule from a short vector, I think the structure of their block-reduction is correct?}{\color{orange} I'll think of something to say}

This generic instantiation should, however, be distinguished from the rank-$2$ reduction developed by LPSW itself.
It simply falls back to unstructured lattice reduction in dimension $2d$ and does not exploit the module structure in the rank-$2$ step.
In particular, the substantially stronger approximation estimate discussed for the LPSW rank-$2$ algorithm comes from their number-theoretic, structure-exploiting implementation and hence from the accompanying heuristic assumptions.

For clarity, the relevant distinction can be summarized as follows.
{\color{orange} also want to add a column for RHF}
\begin{center}
\begin{tabularx}{\textwidth}{
    >{\raggedright\arraybackslash}X
    c
    c
    >{\raggedright\arraybackslash}X
}
\hline
Method
& Heuristic
& Module structure
& {Reduction dimension{ \color{orange} needed dimension for unstructured oracle} }
\\
\hline
LPSW structured rank-$2$ solver
& Yes
& Yes
& $O((d\log \rho)^{2+\eta})$ \\
\hline
{\color{orange}LPSW structured rank-$2$ solver, possibly w/o heuristic}
& Yes
& Yes
& $O(poly(d))$ \\
\hline
LPSW + generic lattice reduction {\color{orange} BKZ for rank 2 reduction}
& No
& No at rank $2$
& $2d$ 
\\
\hline
{\color{orange}unstructured reduction all the way}
& No
& No
& $2d \cdot rank$ \\
\hline
Adelic LLL
& No
& Yes
& $d$ and $d+[E:\mathbb{Q}]\leq 2d$ \\
\hline
\end{tabularx}
\end{center}

Thus, our claim is not that adelic LLL is the only heuristic-free way to obtain a reduction for module lattices.
Rather, adelic LLL gives an unconditional reduction algorithm that operates directly with the module structure over an arbitrary number field, including for non-free modules.
}
%%%%%%%%%%%%%%%%%%%%%%%%%%%%%%%%%%%%%%

%In this section, we give more details regarding how our algorithm for rank 2 modules compares with the ``divide-and-swap'' reduction algorithm of Lee, Pellet--Mary, Stehl\'e, and Wallet~\cite{LPSW19}. In spirit and at a very high level, their algorithm is inspired by the Lagrange-Gauss algorithm, alternating between some form of size reduction and a procedure that executes a swap if some condition is met. Our algorithm has similar structure, but also takes inspiration from adelic reduction theory~\cite{Gar18}.  {\color{orange} I probably also want to emphasize the role of adelic reduction theory} Here is how the various steps compare:

In this section, we provide a detailed comparison between adelic LLL and the module LLL algorithm of \cite{LPSW19}, the state-of-the-art fully structured reduction algorithm to date. Specifically, we clarify (i) where the core differences lie (ii) what oracles are used (iii) what assumptions are needed (iv) what it takes to dequantize (v) how their output qualities compare. We focus on their rank 2 reductions --- Algorithm \ref{alg:lll_p} above and \cite[Algorithm 4.3]{LPSW19} --- since the differences between their rank $n$ versions are relatively immaterial. 

Unravelling the notations, the two algorithms are in fact similar in a few important ways. Both algorithms make progress by repeating size reduction (which \cite{LPSW19} calls a ``Euclidean division'') and swap, like the classical LLL algorithm. Moreover, their swaps --- Algorithm \ref{alg:adelicswap} and lines 4-5 of \cite[Algorithm 4.3]{LPSW19} --- are practically identical. (Still, it may be worth noting that the adelic language transparently reveals that it is indeed a swap, in that it literally swaps the rows of the input matrix $g \in \GL(2,\A)$.)

One minor difference --- that may matter for implementations however --- is that \cite{LPSW19} uses the so-called Siegel condition (line 1 of \cite[Algorithm 4.3]{LPSW19}) for the Lov\'asz test, whereas we use the adelic generalization of the original Lov\'asz condition (Definition \ref{def:LLL-reduced-2}(ii)), which is in principle sharper than the Siegel condition. It is easy to write a Siegel version of our algorithm if desired; the corresponding proofs of Propositions \ref{prop:swap_bound} and \ref{prop:rhf_bound} differ only in a very minor way. It seems less clear, however, how to adjust the algorithms of \cite{LPSW19} to accommodate the Lov\'asz condition.

More importantly, the core difference between adelic LLL and \cite[Algorithm 4.3]{LPSW19} lies in the specific Diophantine problems associated with their size-reduction routine. This is what gives rise to all their differences on the required oracles and assumptions, and so on. We elaborate on this point below.

\subsection{Differences in the Diophantine conditions}

The size reduction of adelic LLL (Algorithm \ref{alg:sizereduce}) aims to solve the following adelic Diophantine approximation problem: given the parameters $(\tau_\sigma)_{\sigma:F\hookrightarrow\C}, 0 <\mu < 1$ and a target $m\in\A_F$, find $r \in F$ such that
\begin{align}\label{eq:app_aLLL_diophantine}
\tau_\sigma|m_\sigma-r_\sigma| \leq \mu \quad\forall{\sigma:F\hookrightarrow\C},\mbox{ and } H_\f(m-r,1)\leq \prod_{\sigma}\tau_\sigma.
\end{align}
In the reduction of a pseudobasis, one always has $m_\f = 0$, for which it suffices to solve the following simpler version: find $p,q\in\cO_F$ such that 
\begin{align}\label{eq:app_aLLL_diophantine2}
|q_\sigma m_\sigma -p_\sigma| \leq \mu \mbox{ and } |q_\sigma| \leq \tau_\sigma
\end{align}
(\emph{cf.} \eqref{eq:sizereduce_new} above) for all embeddings $\sigma$. We want to set $\tau_\sigma$ so as to minimize $\prod_\sigma\tau_\sigma$ while ensuring \eqref{eq:app_aLLL_diophantine} is always solvable (within a given time budget). In our paper, we set all $\tau_\sigma$ to \eqref{eq:C^1/d}.

In comparison, the counterpart in \cite{LPSW19} --- see Algorithm 4.1, and also Corollary 4.8 therein --- is the following problem. For parameters $\varepsilon,C,c>0$ and inputs $a \in (F\otimes\R)^*, b \in F\otimes\R$, and an ideal $\ga$ satisfying $c^{-d} \leq N(\ga) \leq c^d$, find $(u,v) \in \cO_F\times\ga$ such that
\begin{align}\label{eq:app_LPSW_diophantine}
\|ua+bv\|_\infty \leq \varepsilon\|a\|_\infty, \mbox{ and } \|v\|_\infty \leq C.
\end{align}
Here $\|x\|_\infty = \max_\sigma|\sigma(x)|$ for $x\in F$. 
This can also be seen as a kind of a Diophantine approximation problem. \cite{LPSW19} chooses $\varepsilon = 1/(4c)$ (\cite[Algorithm 4.3]{LPSW19}) and $C = c2^{0.55\delta B/d}$ (\cite[Theorem 4.6]{LPSW19}), but the constant $c$, which originates from \cite[Lemma 2.5]{LPSW19}, is only stated to be $2^{\tilde O(\log \Delta_F)/d}$. An explicit formula, such as the one for our $\tau_\sigma$'s, seems to be unavailable, which impacts the implementability of the algorithm of \cite{LPSW19}.

All the major differences between adelic LLL and \cite[Algorithm 4.3]{LPSW19} originate from the fact that \eqref{eq:app_aLLL_diophantine2} is a much easier problem to solve than \eqref{eq:app_LPSW_diophantine}, for which even the correct parameter value is unclear. 
%{\color{orange} a comment I'll erase soon: by making a connection to adelic reduction theory, they may be able to explicitly compute $c$ --- indeed GPT-6 claims it can do that, but what it gives is nontrivially worse than LPSW's claim --- and with that the performance becomes comparable to adelic LLL} 
All the CVP oracles, heuristic assumptions, and quantization issues of \cite{LPSW19} appear precisely to claim that \eqref{eq:app_LPSW_diophantine} is solvable. We examine each of these components in the subsequent sections.

\ignore{
\begin{itemize}
    \item \textbf{Size reduction:} {\color{orange} sk: perhaps it's good to highlight why this is a difficult step: it's basically trying to overcome the fact that $\cO_F$ is not in general a Euclidean domain} In general, size reduction over number fields is a difficult step, as $\cO_F$ is not in general a Euclidean domain. Although~\cite{LPSW19} does not explicitly specify a notion of size-reduction, it presents an algorithm that on input field elements $a,b\in F$ outputs ring elements $u,v\in \cO_F$ such that the trace norm of the linear combination $ua+vb$ is smaller than that of $a$, and that of $v$ is bounded by a constant that depends on $F$ (see e.g. Corollary 4.8). Note that all constraints rely on the trace norm. Using notations for \ref{subsec:reduction_notion}, and assuming, the product formula shows that $H_\f(m,1)\le C$ can be rephrased into $N(\langle v\rangle) \le C$, where $v$ is the product of denominators at finite places of $m$. Essentially, by AM-GM, this says that our condition is a relaxation of theirs. {\color{orange} sk: I thought the other way around, at least qualitatively; we need control over each place, whereas LPSW's first inequality in 4.8 only compare maximal coordinate values on each side. I agree $H_\f(m,1)$ and $\|v\|_\infty$ are interchangeable (up to a constant factor perhaps). } {\color{purple} maybe you can rephrase what I was trying to say from the other perspective.} This extra freedom carries over to our conditions at infinite places. Overall, our size reduction condition is guided by global heights, aiming to capture a fundamental domain; whereas that of \cite{LPSW19} relies more heavily on improvements at the infinite places only. We believe our notion of diophantine approximation to be more natural with regards to reduction theory. {\color{orange} sk: so really it comes down to the question ``which notion of size reduction/diophantine approximation is correct?'' our answer is grounded in adelic reduction theory. }
    \item \textbf{Condition:} Using our notations, \cite{LPSW19} compare $H(a_1)$ with $H(a_2)$, while we compare $H((a_1,0))$ with $H((ma_1,a_2))$. By analogy with the rational case, our condition is a Lov\'asz condition, while their condition corresponds to a more relaxed Siegel condition.
    \item \textbf{Swap:} A pleasant feature obtained by expressing our algorithm using the adelic formalism is the following: the swapping procedure is simply a swap, corresponding to left-multiplication by $\begin{pmatrix}
        0 & 1\\1 &0
    \end{pmatrix}\in \GL(2,\A_F)$. In pseudo-basis formalism however, swaps are not that simple, as they require extra care to preserve the module structure. It is interesting to note that the procedure used for ``swapping'' in \cite{LPSW19} (see steps 4 and 5 of Algorithm 4.3) is essentially the same as our own \texttt{adelic-swap}.
\end{itemize}

I acknowledge the adelic LLL and LPSW share many similarities, and the core difference arises from something technically deep: the exact Diophantine approximation problem involved. 

Regarding this, Adelic LLL takes input from the reduction theory of $\GL(n,\A)$, with a slight refinement for ease of algorithmic implementation. LPSW takes a different path, and this is where the two algorithms really start to differ.
}
%Also... this may be hard to convince, but the Diophantine problem we have is more ``natural'' from the math point of view; I might do some literature search to back this up. Again, another hard point to convince, but this is good because it facilitates further research.

\subsection{Subroutines needed}

Adelic LLL can run on an oracle as weak as dimension $d+1$ classical LLL --- see Algorithm \ref{alg:sizereduce} and its accompanying explanations. For a better output quality, one could choose a stronger HSVP algorithm in dimension $d+\deg E$ for a subfield $E\subseteq F$.

On the other hand, the rank-2 reduction of~\cite[Alg 4.3]{LPSW19} requires two subroutines:
\begin{itemize}[label=\textbullet]
    \item ~\cite[Alg. 4.2]{LPSW19} to scale a pseudo-basis: this is a heuristic quantum algorithm
    which, given as input a
fractional ideal $I$ of $R$ and any $\alpha \in F_{\mathbb{R}}^\times$,
outputs $x \in \alpha I \setminus \{0\}$ such that
\[
\|x\|_\infty
\leq
c \cdot |N(\alpha)|^{1/d} \cdot N(I)^{1/d},
\]
where
\[
c = 2^{\widetilde{O}(\log \Delta_F)/d}.
\]
\ignore{
In particular, we have
\[
\|x\|_\infty
\leq
c \cdot |N(x)|^{1/d}.
\]}
If $\alpha \in F$, the algorithm performs a number of quantum
operations that is polynomial in $\log \Delta_F$ and the input
bit-length, and makes a single call to a CVP oracle in a lattice $L_{F,1}$
of dimension
$\widetilde{O}(\log \Delta_F)$ and depending only on $F$.
 \item ~\cite[Alg. 4.1]{LPSW19} which is some kind of Euclidean division over $\cO_F$. The algorithm is heuristic and calls Alg. 4.2 three times, but also makes a single call
 to a CVP oracle in $\cO_F$,
 and another single call to a CVP oracle in a lattice 
 $L_{F,2}$
of dimension  $O((d \log \rho(\cO_F))^{2+\eta} )$ for any $\eta > 0$
\end{itemize}
Even with preprocessing, CVP is known to be NP-hard \cite{Mic01}. The lattices $L_{F,1}$ and $L_{F,2}$ only depend on the field $F$, but the target is different for each instantiation. Hence the rank 2 algorithm of \cite{LPSW19} requires subroutines to solve instances of NP-hard problems.

%(independent of instance; even with the very best preprocessing though, CVPP is NP hard unlike BDD; moreover target depends on instance). For $L_{K,2}$ they also need to assume that the target is never too far away: this is their Heuristic 1.

In addition, the dimensions of these lattices are prohibitively large. 
$L_{F,1}$ has dimension $\tilde O(\log|\Delta_F|)$. For cyclotomic fields this is $\tilde O(d)$. According to \cite[Corollary 4.8]{LPSW19}, $L_{F,2}$ has dimension $O_\eta(d^{2+\eta})$ for $\eta > 0$. 
To dequantize the algorithm of \cite{LPSW19}, $L_{F,1}$ need to be replaced by $L'_{F,1}$, whose dimension is almost as large as that of $L_{F,2}$. %As with the parameters for \eqref{eq:app_LPSW_diophantine}, the implicit constants here are not clearly specified. --- well at least they can give some specific working value

\ignore{
%At a lower level, our algorithms each require a subroutine for size reduction (or ``Euclidean division'' using the wording of \cite{LPSW19}'s Algorithm 4.1). In both cases, this subroutine relies on unstructured lattice reduction:
\begin{itemize}
    \item \textbf{LPSW:} Requires CVP calls in a fixed lattice of dimension $\Omega(d^2)$.
    \item \textbf{Adelic LLL:} Requires an HSVP oracle in dimension at most $2d$. {\color{orange} and as low as $d+1$, sacrificing output quality}
\end{itemize}
This makes our algorithm practical, as it can be instantiated in classical polynomial time using LLL, while \cite{LPSW19} has prohibitive preprocessing, and their Euclidean division relies on additional heuristics {\color{purple} please add some scary high-level description of the heuristic}. {\color{orange} also I found on LPSW page 12 eprint version, saying Lemma 2.5 (heuristic) is a quantum algorithm but it can be dequantized with additional heuristic assumptions. Also, regarding the part in LPSW pages 4-5 where they remove the heuristic, according to my study with chatGPT this is turning the problem into an NP-hard problem?! This sounds pretty scary to me }
{\color{purple} I'm in favour of disregarding this argument pages 4-5 as it is not well detailed, and horrifically inefficient anyways, without mentioning the fact that it is very much unstructured} {\color{orange} but referee might say, ``hey one can make LPSW heuristic free'' which may be technically true, so I feel it wouldn't hurt to preemptively defend ourselves from it} {\color{purple} agree, but in some sense this shouldn't be too hard as it implies using CVP in some huge dimension, much worse than $d^2$ apparently, which makes the comparison very much in our favour}
}

\subsection{Heuristic assumptions}

Adelic LLL requires no heuristics.

The two heuristic pillars of \cite{LPSW19} are Heuristic 1 and Lemma 2.5 (heuristic) of their paper. Lemma 2.5 in turn assumes the generalized Riemann hypothesis and Heuristic 4 of  \cite{PHS19}. These are independent heuristics, and every other heuristic statement of \cite{LPSW19} follows from either of these two. Heuristic 1 is a statement to the effect that the CVP target $\mathbf t$ in \cite[Algorithm 4.1]{LPSW19} is close enough to $L_{K,2}$ for a solution to \eqref{eq:app_LPSW_diophantine} with respect to the claimed parameters to exist. Lemma 2.5 is a statement that, with one call to the CVP oracle of $L_{K,1}$, it can find a short vector of a weighted ideal in polynomial time, which is used as a supplementary device in their Algorithm 4.1.

%The heuristics of LPSW are introduced precisely to solve their Diophantine approximation said above.

%Theorem 1.2 as stated relies on Heuristic 1 and Heuristic Lemma 2.5 --- these are the two heuristic pillars of the paper. Lemma 2.5 in turn relies on GRH and another heuristic from a different paper. Both of these are assumed in order to show that their Diophantine problem always has a solution. Heuristic 1 actually has an easy counterexample, but it does not seem to immediately invalidate LPSW's conclusions from it.

%(here I'd also discuss a ``semi-structured'' version along the lines Changmin and Henry suggested)

In pages 4-5 of \cite{LPSW19}, an approach to make its rank 2 algorithm heuristic-free is suggested. Briefly speaking, they reduce their problem to 3SAT which they again reduce to CVP with preprocessing for a lattice construction that would play the role of $\cO_F$, $L_{F,1}$ and $L_{F,2}$ combined. However, this construction has an even higher dimension than $\cO_F\oplus L_{F,1}\oplus L_{F,2}$, and by a far margin.

\subsection{Dequantization}

Adelic LLL is classical as presented in this paper.

Lemma 2.5 (heuristic) of \cite{LPSW19}, on the other hand, entails a quantum algorithm.
%(they need this to find short rk 1 submodule, not necessarily to solve HSVP)
According to \cite{LPSW19}, it is dequantizable with additional heuristics (Heuristics 2 and 3 of \cite{PHS19}) and subexponential running time disregarding the complexity of the CVP oracle; see the remarks below Theorem 1.2 and Lemma 2.5 of \cite{LPSW19}.

One does not need to take this path for the dequantization of \cite{LPSW19}, however. 
For a free rank 2 module with a known basis, \cite[Algorithm 4.3]{LPSW19} can be made to run on classical computers, by replacing \cite[Lemma 2.5]{LPSW19} with \cite[Lemma 5.2]{LPSW19}. On the other hand, \cite{DMPT} gives a probabilistic polynomial-time reduction of the problem of reducing an arbitrary rank 2 module to that of reducing a free rank 2 module. Combining, one obtains a dequantization of \cite[Algorithm 4.3]{LPSW19}. All the dependencies on the oracles and the heuristics remain; in fact, it becomes worse in that $L_{K,1}$ needs to be replaced by the much larger $L'_{K,1}$.

\subsection{Performance comparison}
\label{subsec: Comparison}

%The downside to making our algorithm provable and using cheaper subroutines is that we achieve a larger approximation factor compared to \cite{LPSW19}. {\color{orange} After all the above discussion I wouldn't call it a downside..}

%{\color{purple} moved here temporarily}

Unraveling the definitions, our $Q$ from Proposition \ref{prop:rhf_bound} is to be compared against $\alpha_F^{1/2}$ of \cite{LPSW19}. From the proof of \cite[Corollary 4.2]{LPSW19}, one has $\alpha_F \approx2^d\gamma_\mathcal N^{2d}|\Delta_F|$, and from \cite[Theorem 4.1]{LPSW19}, $\gamma_\mathcal N = 2^{\tilde O(d\log\rho)/d}$, where $\rho$ here is the covering radius of the Minkowski embedding of $\cO_F$. In summary,
\begin{align*}
\alpha_F^\frac{1}{2} \approx 2^{\tilde O(d\log\rho)}\sqrt{|\Delta_F|}.
\end{align*}

This is much better than $Q$, which is $2^{O(d^2)}$, very roughly speaking (a more precise formula in case $F=\Q(\zeta_{2^{\kappa+1}})$ is discussed in Section \ref{sec:reduction}). It may be helpful to compare in terms of the root Hermite factor also. Given a rank $1$ submodule of height $QH(\det g)^{1/2}$, one can apply any reduction algorithm to it to obtain a vector of length $2^{O(d)}H(\det g)^{1/(2d)}$; replacing $Q$ by $\alpha^{1/2}_F$, the number becomes $2^{O((\log d)^k)}H(\det g)^{1/(2d)}$ for some $k>0$ implicit in the $\tilde O$ notation, comparable to HSVP with a quasi-polynomial approximation factor. 

One is reminded, however, that this advantage comes at a cost of polynomially many exact CVP instances on a lattice of dimension $O(d^{2+\eta})$, as discussed above.

%However, one is reminded that the rank 2 reduction of \cite{LPSW19}, setting aside its heuristic nature, involves an exact CVP on a lattice of dimension $O((d\log\rho)^{2+\eta})$ for any $\eta>0$ \cite[Corollary 4.8]{LPSW19}. 

%LPSW guarantees excellent output quality but at what cost?

\ignore{
\vspace{4mm} {\color{orange} below is more careful writing by Phong}

Alg. 4.3 is the rank-2 reduction of~\cite{LPSW19}. It requires two subroutines:
\begin{itemize}
    \item ~\cite[Alg. 4.2]{LPSW19} to scale a pseudo-basis: this is a heuristic quantum algorithm
    which, given as input a
fractional ideal $I$ of $R$ and any $\alpha \in K_{\mathbb{R}}^\times$,
outputs $x \in \alpha I \setminus \{0\}$ such that
\[
\|x\|_\infty
\leq
c \cdot |N(\alpha)|^{1/d} \cdot N(I)^{1/d},
\]
where
\[
c = 2^{\widetilde{O}(\log \Delta_K)/d}.
\]
In particular, we have
\[
\|x\|_\infty
\leq
c \cdot |N(x)|^{1/d}.
\]
If $\alpha \in K$, the algorithm performs a number of quantum
operations that is polynomial in $\log \Delta_K$ and the input
bit-length, and makes a single call to a CVP oracle in a lattice $L_{K,1}$
of dimension
$\widetilde{O}(\log \Delta_K)$ and depending only on $K$.
 \item ~\cite[Alg. 4.1]{LPSW19} which is some kind of Euclidean division over $R$. The algorithm is heuristic and calls Alg. 4.2 three times, but also makes a single call
 to a CVP oracle in $R$,
 and another single call to a CVP oracle in a lattice 
 $L_{K,2}$
of dimension  $O((d \log \rho(R))^{2+\eta} )$ for any $\eta > 0$
\end{itemize}
Then \cite[Lemma 4.11]{LPSW19} gives the following result for the rank-2 algorithm (Algorithm 4.3):
Let
\[
\gamma_{\cN} \geq 4 \cdot C \cdot c^2,
\]
where $C=2^{\widetilde{O}(\log \Delta_K)/d}$. Then, given as input a pseudo-basis
of a rank-$2$ module
\[
M \subset K_{\mathbb{R}}^2,
\]
Algorithm~4.3 outputs a vector
$v \in M \setminus \{0\}$ such that
\[
\cN(v) \leq \gamma_{\cN}^d \lambda_1(M).
\]
The minimum value of $\gamma_{\cN}$
is 
$$ 4 \cdot C \cdot c^2 = 2^{\widetilde{O}(\log \Delta_K)/d}
2^{\widetilde{O}(\log \Delta_K)/d}
= 2^{\widetilde{O}(\log \Delta_K)/d}
$$
Further, if $M \subseteq K^2$ and assuming that Algorithms~4.1 and
4.2 are correct and run in time polynomial in $\log \Delta_K$ and
their input bit-length, then Algorithm~4.3 also runs in time
polynomial in $\log \Delta_K$ and the input bit-length.
}

\section{Reductions to Rank-2}

For comparison, we recall the reductions to rank-2 proved in~\cite{LPSW19,MS20}: 
\begin{itemize}[label=\textbullet]
\item \cite[Theorem 5.10]{MS20} states three reductions:
For any number field $K$, order $R \subseteq \mathcal{O}_K$
of degree $n$ with an inner product
$\langle \cdot,\cdot\rangle_\rho$ over $K_{\mathbb{R}}$, rank
$k \geq 2$, approximation factor $\gamma = \gamma(R,k) \geq 1$,
semicanonical inner product $\langle \cdot,\cdot\rangle_\rho$, and
constant $\varepsilon > 0$, there is an efficient reduction from
$(\gamma,k)$-ModuleSVP to $(\gamma',2)$-ModuleSVP, where
\[
\gamma
:=
(1+\varepsilon)
\cdot
\left(
\frac{\gamma'\delta_n}{\alpha_R}
\right)^2
\cdot
\left(
\frac{\gamma'\delta_n\mu_{R,2}}{\alpha_R}
\right)^{2(k-2)}.
\]
There is also an efficient reduction from
$(\gamma_H,k)$-ModuleHSVP to $(\gamma',2)$-ModuleSVP, where
\[
\gamma_H
:=
\gamma'\delta_n
\cdot
\left(
(1+\varepsilon)\gamma'
\frac{\delta_n\mu_{R,2}}{\alpha_R}
\right)^{k-1}.
\]
For any number field $K$, order $R \subseteq \mathcal{O}_K$, rank
$k \geq 2$, approximation factor $\gamma' = \gamma'(R,k) \geq 1$,
semicanonical inner product $\langle \cdot,\cdot\rangle_\rho$, and
constant $\varepsilon > 0$, there exists an efficient reduction from
$(\gamma,k)$-DIP to $(\gamma',2)$-DIP, where
\[
\gamma
:=
(1+\varepsilon)\gamma'
\cdot
\left(
(1+\varepsilon)\gamma'\mu_{R,2}
\right)^{2(k-2)}.
\]
where:
\begin{description}
\item[Hermite's constant]  \[
\delta_n
:=
\sup_L
\frac{\lambda_1(L)}
{\det(L)^{1/n}}.
\]
\item[Other constants] \[
\alpha_R
:=
\inf_I
\frac{\lambda_1(I)}
{\det(I)^{1/n}},
\]
where the infimum is over all rank-one modules
$I \subset K_{\mathbb{R}}^l$. For the canonical embedding, we have \[
\alpha_R = \frac{\sqrt{n}}{\det(R)^{1/n}}.
\]
\[
\mu_{R,k}
:=
\sup_M
\frac{\tau_1(M)}
{\det(M)^{1/(kn)}},
\]
where the supremum is over all rank-$k$ module lattices
$M \subset K^{k'}$ for any integer $k' \geq k$,
and for any module lattice $M$
\[
\tau_1(M)
:=
\min_{I \subset M} \det(I)^{1/n},
\]
where the minimum is over the rank-one submodule lattices.
Then \[
1 \leq \mu_{R,k} \leq \frac{\delta_{kn}}{\alpha_R}.
\tag{4}
\]
\item[$(\gamma,k)$-ModuleSVP]
For an approximation factor $\gamma = \gamma(R,k) \geq 1$, the input is (a generating set for) a module lattice
$M \subset K_{\mathbb{R}}^l$ with rank $k$, the goal is to output a
module element $x \in M$ such that
\[
0 < \|x\|_\rho \leq \gamma \lambda_1(M).
\]
\item[$(\gamma,k)$-ModuleHSVP]
For an approximation factor $\gamma = \gamma(R,k) \geq 1$,
the input is (a generating set for) a module lattice
$M \subset K_{\mathbb{R}}^l$ with rank $k$, the goal is to output a
module element $x \in M$ such that
\[
0 < \|x\|_\rho \leq
\gamma \det(M)^{1/(kn)}.
\]
\item[$(\gamma,k)$-DIP]
For an approximation factor $\gamma = \gamma(R,k) \geq 1$, 
the input is (a generating set for) a module lattice
$M \subset K_{\mathbb{R}}^l$ with rank $k$,
the
$(\gamma,k)$-Dense Ideal Problem, or $(\gamma,k)$-DIP, is asks to find a rank-one
submodule $M' \subset M$  such
that
\[
\det(M')^{1/n} \leq \gamma \tau_1(M).
\]
\end{description}
\item ~\cite[Th 3.9]{LPSW19} states:
Let $\gamma \geq 1$ and assume that an LLL-reduced $\mathbb{Z}$-basis
of $R$ is known. Then there exists a reduction from solving
$\mathrm{SVP}_{\gamma'}$ in rank-$n$ modules (with respect to
$\lVert \cdot \rVert$) in $K^n$ to solving $\mathrm{SVP}_{\gamma}$
in rank-$2$ modules in $K^2$, where
\[
\gamma'
=
\left(
2\gamma\Delta_K^{1/d}
\right)^{2n-1}.
\]
This reduction runs in time polynomial in $\log \Delta_K$ and the
bit-length of the input pseudo-basis.
%\phong{This approximation factor is different from \cite{MS20}, which is roughly:
%$$ \gamma' = (\gamma \sqrt{n}/\alpha_R)^{2(n-1)} \mu_{R,2}^{2(n-2)}$$,
%so the exponent of $\gamma$ is slightly smaller there.}
%\phong{LPSW19 only consider approx-SVP, not HSVP}
\end{itemize}
Hence, all previous rank-2 reductions~\cite{LPSW19,MS20} transformed a computational problem
(either HSVP or approx-SVP) in a module lattice of arbitrary rank
into approximating the shortest vector in a rank-2 module lattice.
By contrast, we reduce such problems
to finding a short vector in a rank-1 module lattice, with respect to its determinant:
due to the structure of rank-1 module lattices,
such a vector also approximates a shortest vector.

\end{document}